\documentclass[aop, preprint]{imsart}

\startlocaldefs

\usepackage{amsmath}
\usepackage{amsthm}
\usepackage{amssymb}
\usepackage{enumitem}
\usepackage{placeins}
\usepackage{graphicx}
\usepackage{float}
\usepackage{caption}
\usepackage{subcaption}
\usepackage{ amssymb }
\usepackage{epstopdf}
\usepackage[utf8]{inputenc}
\usepackage[english]{babel}
\usepackage{ dsfont }
\usepackage[square,numbers,sort&compress]{natbib}
\usepackage{amsfonts}
\usepackage{scalerel}[2016/12/29]
\usepackage{bbm}

\theoremstyle{plain}
\newtheorem{theorem}{Theorem}[section]
\newtheorem{corollary}[theorem]{Corollary}

\newtheorem{problem}[theorem]{Problem}
\newtheorem{prop}[theorem]{Proposition}
\newtheorem{lemma}[theorem]{Lemma}

\theoremstyle{definition}
\newtheorem{definition}[theorem]{Definition}
\newtheorem{example}[theorem]{Example} 

\theoremstyle{plain}

\newcommand{\N}{\mathbb{N}}

\newcommand{\p}{\mathbb{P}}

\newcommand{\expt}{\mathbb{E}}

\newcommand{\floor}[1]{{\left\lfloor #1 \right\rfloor}}

\newcommand{\sset}{\subset}

\newcommand{\la}{\lambda}
\newcommand{\al}{\alpha}
\newcommand{\Om}{\Omega}

\newcommand{\mathand}{\;\text{and}\;}
\newcommand{\mathfor}{\;\text{for}\;}

\newcommand{\be}{\beta}

\newcommand{\Ga}{\Gamma}
\newcommand{\ep}{\epsilon}

\newcommand{\sig}{\sigma}

\newcommand{\scrE}{\mathcal{E}}
\newcommand{\Zd}{\mathds{Z}^4_\uparrow}
\newcommand{\scrA}{\mathcal{A}}
\newcommand{\scrD}{\mathcal{D}}
\newcommand{\scrG}{\mathcal{G}}
\newcommand{\scrP}{\mathcal{P}}

\newcommand{\scrQ}{\mathcal{Q}}

\newcommand{\scrV}{\mathcal{V}}

\newcommand{\scrL}{\mathcal{L}}
\newcommand{\scrH}{\mathcal{H}}
\newcommand{\scrS}{\mathcal{S}}
\newcommand{\scrN}{\mathcal{N}}

\newcommand{\scrF}{\mathcal{F}}
\newcommand{\scrX}{\mathcal{X}}

\newcommand{\supp}{\text{supp}}
\newcommand{\Z}{\mathds{Z}}

\newcommand{\R}{\mathds{R}}

\usepackage{MnSymbol}

\usepackage{yfonts}
\usepackage{ wasysym }

\newcommand{\eqd}{\stackrel{d}{=}}

\newcommand{\cvgd}{\stackrel{d}{\to}}

\newcommand{\X}{\times}

\newcommand{\symdif}{\triangle}

\newcommand{\id}{\text{id}}

\newcommand{\smin}{\setminus}
\newcommand{\lf}{\left}
\newcommand{\rg}{\right}

\newcommand{\bu}{\mathbf{u}}
\newcommand{\bv}{\mathbf{v}}
\newcommand{\bw}{\mathbf{w}}
\newcommand{\bx}{\mathbf{x}}
\newcommand{\by}{\mathbf{y}}

\RequirePackage[colorlinks,citecolor=blue,urlcolor=blue]{hyperref}

\arxiv{}
\endlocaldefs

\begin{document}
	
	\begin{frontmatter}
		
		\title{Hidden invariance of last passage percolation and directed polymers}
		\author{Duncan Dauvergne}
		\runtitle{Hidden invariance of LPP}

\begin{abstract} 
	\qquad Last passage percolation and directed polymer models on $\Z^2$ are invariant under translation and certain reflections. When these models have an integrable structure coming from either the RSK correspondence or the geometric RSK correspondence (e.g. geometric last passage percolation or the log-gamma polymer), we show that these basic invariances can be combined with a decoupling property to yield a rich new set of symmetries. Among other results, we prove shift and rearrangement invariance statements for last passage times, geodesic locations, disjointness probabilities, polymer partition functions, and quenched polymer measures. We also use our framework to find `scrambled' versions of the classical RSK correspondence, and to find an RSK correspondence for moon polyominoes. The results extend to limiting models, including the KPZ equation and the Airy sheet.
\end{abstract}

\begin{keyword}[class=MSC2010]
	\kwd[Primary ]{60K35}
\kwd[; secondary ]{05A15, 05A19, 05E10}
\end{keyword}

\begin{keyword}
\kwd{RSK correspondence}
\kwd{geometric RSK correspondence}
\kwd{last passage percolation}
\kwd{directed polymers}
\kwd{KPZ}
\kwd{Robinson-Schensted correspondence}
\kwd{Airy sheet}
\kwd{moon polyominoes}
\end{keyword}

\end{frontmatter}

\section{Introduction}
\label{S:intro}
Consider an array of random variables $M = \{M_x: x \in \Z^2\}$. For an up-right path $\pi$ in $\Z^2$ (henceforth simply a path), the weight of $\pi$ in the environment $M$ is given by
$$
M(\pi) = \sum_{x \in \pi} M_x.
$$
For two points $x, y \in \Z^2$, we say that $x \nearrow y$ if $x_1 \le y_1, x_2 \le y_2$. Let 
$$
\Zd = \{u = (u^-; u^+) \in \Z^2 \X \Z^2 : u^- \nearrow u^+ \}.
$$
For any $u \in \Zd$, we say that $\pi$ is a $u$-path if $\pi$ goes from $u^-$ to $u^+$. Define the last passage value
$$
Z_M(u) = \max_\pi M(\pi),
$$
where the $\max$ is taken over all $u$-paths $\pi$. For $\bu = (u_1, \dots, u_k) \in (\Zd)^k$, we also define the multi-point last passage value
$$
Z_M(\bu) = \max_\pi \sum_{i=1}^k M(\pi_i),
$$ 
where the max is taken over all $k$-tuples of disjoint paths $\pi = (\pi_1, \dots, \pi_k)$, where each $\pi_i$ is a $u_i$-path. Note that this is only defined for vectors $\bu$ for which a set of disjoint paths exists. Letting $\scrE$ denote this set of vectors, we can think of $Z_M$ as a function from $\scrE$ to $\R$.

Last passage percolation is one of the most tractable models in the KPZ (Kardar-Parisi-Zhang) universality class for random growth. Over the past twenty-five years, the KPZ universality class has been the focus of intense and fruitful research. For a gentle introduction to this area, see \citep{romik2015surprising}. Surveys and reviews focussing on more recent developments include \citep{corwin2012kardar,  quastel2011introduction, borodin2014integrable, borodin2016lectures, zygouras2018some}.

Any i.i.d. last passage percolation model is invariant under translation and certain reflections:
\begin{enumerate}[nosep, label=(\Roman*)]
	\item For $c \in \Z^2$, define $(T_cM)_x = M_{x + c}$. Then $Z_M \eqd Z_{T_c M}$.
	\item Define $(R_1 M)_{(x, y)} = M_{(y, x)}$. Then $Z_M \eqd Z_{R_1 M}$.
	\item Define $(R_2 M)_{(x, y)} = M_{(-x, -y)}$. Then $Z_M \eqd Z_{R_2 M}$.
\end{enumerate}
For general weights, we shouldn't necessarily expect that there are distributional symmetries beyond the ones listed above.

On the other hand, for a related integrable model known as Brownian last passage percolation, Borodin, Gorin, and Wheeler \citep{borodin2019shift} recently proved a new and remarkable shift invariance. Translated to the context of last passage percolation on $\Z^2$, this would read as follows.
\begin{itemize}
	\label{I:iv}
	\item [(IV)] Let $u_1, \dots, u_k, v \in \Zd$, let $c \in \Z^2$, and set $v' = v + (c, c)$. Suppose that for any $i$, every $u_i$-path intersects every $v$-path and every $v'$-path. Then
	$$
	(Z_M(u_1), \dots, Z_M(u_k), Z_M(v)) \eqd (Z_M(u_1), \dots, Z_M(u_k), Z_M(v')).
	$$
\end{itemize}
See Figure \ref{fig:simple} for an example of this invariance in the most basic $k=1$ case.

This potential invariance is much more nontrivial than (I)-(III) above, and its proof for Brownian last passage percolation in \citep{borodin2019shift} is very indirect. The proof proceeds by interpreting Brownian last passage percolation as a certain degeneration of the inhomogeneous coloured stochastic six vertex model, and then proving the analogue of (IV) for that model. Moreover, the proof in the context of the inhomogeneous coloured stochastic six vertex model is itself difficult, relying on the Yang-Baxter integrability of that model, a delicate induction argument, and Lagrange interpolation. The full power and complexity of the coloured stochastic six vertex model is necessary in their proof; the same proof would fail for any degenerations. Since it is not known whether the coloured stochastic six vertex model degenerates to any of the lattice last passage percolation models discussed above, it is not clear how to apply their proof and result in this setting.

\begin{figure}
	\centering
	\includegraphics[width=0.4\textwidth]{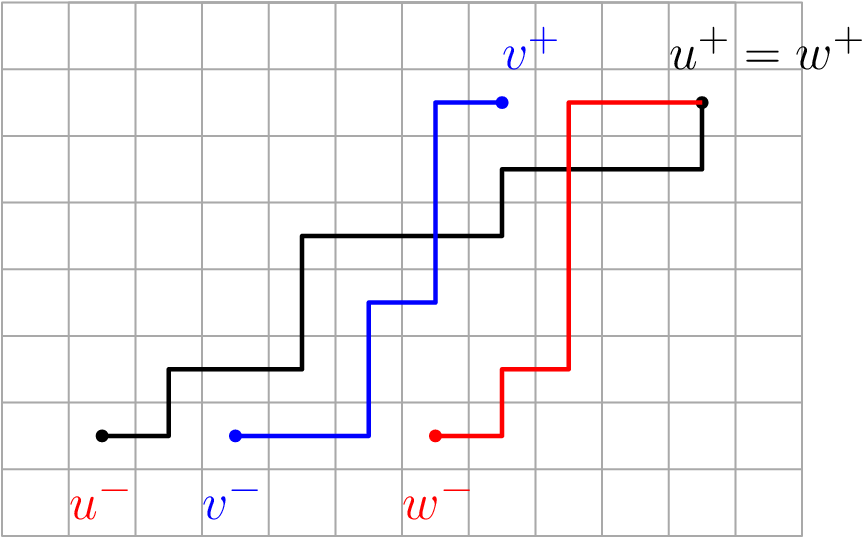}
	\caption{When $M$ is an environment of geometric or exponential random variables, we have $(Z_M(u), Z_M(v)) \eqd (Z_M(u), Z_M(w))$.}
	\label{fig:simple}
\end{figure}
This paper is devoted to understanding invariances such as (IV) in the context of lattice last passage percolation. We will also address directed polymer models.  Unsurprisingly, (IV) does not hold for general weight distributions (see Example \ref{E:nonintegrable}). This suggests that we should restrict our focus to integrable models, and that integrability should play a fundamental role.

Last passage percolation is integrable when the weight distribution is either exponential or geometric. When this is the case, the classical Robinson-Schensted-Knuth correspondence (RSK) connects certain processes of multi-point last passage values with pairs of random Young tableaux via Greene's theorem \citep{greene1974extension}. This allows for explicit formulas for last passage value probabilities, and yields the following remarkable independence structure for these models.

For a point $u = (u^-; u^+) = (u^-_1, u^-_2; u^+_1, u^+_2) \in \Zd$, define the vector
\begin{align*}
u^k &= \big((u^-; u^+ - (k - 1, 0)), (u^- + (1, 0); u^+ - (k - 2, 0)), \dots, (u^- + (k-1, 0); u^+) \big),
\end{align*}
and let 
$$
D(u) = \{u^k : k \in \{1, \dots, \min(u^+_1 - u^-_1 + 1, u^+_2 - u^-_2 + 1)\} \}.
$$
We think of $Z_M|_{D(u)}$ as the collection of all multi-point last passage values going `diagonally' between the lower left and top right corners of $u$  (here the notation $Z_M|_{X}$ is the restriction of the function $Z_M$ to the set $X$). For any sets $V = \{v_1, \dots, v_n\} \sset \Zd$ where each $v_i = (u^-; u^+_1, v^+_{i, 2})$ for some $v^+_{i, 2} \in [u^-_2, u^+_2]$ and $W = \{w_1, \dots, w_k\}$ where each $w_i = (u^-; w^+_{i, 1}, u^+_2)$ for some $w^+_{i, 1} \in [u^-_1, u^+_1]$, the last passage values $Z_M|_V$ and $Z_M|_W$ are conditionally independent given $Z_M|_{D(u)}$ . In words, this says that  the collection of last passage values going from $u^-$ to the right edge of $u$ and the collection of last passage values going from from $u^-$ to the top edge of $u$ are conditionally independent given the collection of multi-point last passage values that cross $u$ diagonally.

 In the case when the weights are geometric, this amounts to the fact that when the RSK correspondence is applied to an $m \times n$ matrix of independent geometric variables $M_{i, j}$ on $\{0, 1, \dots\}$ with parameter $\al_i \be_j$ for vectors $\al \in \R^m, \be \in \R^n$, then the joint law of the two output Young tableaux is proportional to a product of Schur functions 
 \begin{equation}
 \label{E:Schur}
s_\la(\al) s_\la(\be).
 \end{equation} 
 See \cite{sagan2013symmetric} for general background about the combinatorics of RSK and Schur functions, \cite{borodin2016lectures} for an introduction to the application of RSK and Schur functions in probability theory, or Section 3 of \cite{o2003conditioned} for a direct derivation of the above fact.
 Here $\la$ is the shared shape of the Young tableaux; it corresponds to our $Z_M|_{D(u)}$. The conditional independence comes from the product structure.

This conditional independence extends to sets of vectors $V, W$ where each $\bv = (v_1, \dots, v_k) \in V$ has entries $v_i$ of the form $(u^-_1, v^-_{i,2} ; u^+_1, v^+_{i, 2})$ with $[v^-_{i,2}, v^+_{i, 2}] \sset [u^-_{2}, u^+_{2}]$ and each $\bw \in W$ has entries $w_i$ of the form $(w^-_{i,1}, u^-_2 ; w^+_{i,1}, u^+_2)$ with $[w^-_{i,1}, w^+_{i, 1}] \sset [u^-_{1}, u^+_{1}]$. That is, the collections of multi-point last passage values going from left to right across $u$ and from bottom to top across $u$ are conditionally independent given the collection of multi-point last passage values that cross $u$ diagonally.

In the context of Brownian last passage percolation, this is a consequence of a recent result of the author, Ortmann, and Vir\'ag \citep{DOV} regarding preservation of certain last passage values under RSK. It turns out that highly similar results had appeared independently in two much older papers. Biane, Bougerol, and O'Connell \cite{biane2005littelmann} proved this in the case of equal start and end points (see their Lemma 4.8) and Noumi and Yamada \citep{noumi2002tropical} studied closely related phenomena in the discrete and positive-temperature settings. 
 Using the ideas from \citep{noumi2002tropical}, we provide another proof of this preservation result in Theorem \ref{T:encoding-lpp}.

This extended conditional independence implies the existence of certain Markov chains of multi-point last passage values, giving the following strategy for proving that $(X, Y)$ and $(X, Z)$ are equal in distribution, where $X, Y,$ and $Z$ are vectors of (potentially multi-point) last passage values.
\begin{enumerate}
	\item Find Markov chains $(X_0, \dots, X_n)$ and $(Y_0, \dots, Y_n)$ with $(X_0, X_n) = (X, Y)$ and $(Y_0, Y_n) = (X, Z)$ (by using the conditional independence property coming from RSK and the independence of different entries in $M$).
	\item Check that the Markov chains have the same transition probabilities (e.g. by using the basic invariances (I)-(III)).
\end{enumerate}

We give an example of how this strategy yields the symmetry in Figure \ref{fig:simple}. The proof we present here is essentially complete, up to showing the conditional indepedendence statements that will follow from RSK (see Theorem \ref{T:encoding-lpp} for details).

\begin{example}
	\label{Ex:simple}
	Let $M$ be an environment of exponential or geometric random variables. Let $u, v, w \in \Zd$ be such that 
	\begin{itemize}[nosep]
		\item $w = v + (c, 0; c, 0)$ for some $c \in \N$,
		\item $u^-_2 = v^-_2$ and $u^+_2 = v^+_2$,
		\item  $[v_1^-, v_1^+] \cup [w_1^-, w_1^+] \sset [u^-_1, u^+_1]$.
	\end{itemize}
	Then $(Z_M(u), Z_M(v)) \eqd (Z_M(u); Z_M(w))$.
\end{example} 

\begin{proof}[Proof sketch]
	Let $(v^-; w^+) = [v^-_1, w^+_1] \X [v^-_2, w^+_2] = [v^-_1, w^+_1] \X [u^-_2, u^+_2]$ denote the box whose southwest and northeast corners are $v^-$ and $w^+$, respectively.
	The last passage value $Z_M(u)$ is a function of the random variables $\{M_x : x \notin (v^-; w^+)\}$ and the last passage values going from left-to-right across $(v^-; w^+)$:
	$$
	Z_M|_X, \qquad \text{ where } \qquad X = \{(v_1^-, x; w^+_1, y) : x \le y \in [u^-_2, u^+_2] \}.
	$$ 
	Moreover, since $Z_M(v), Z_M(w)$ are both last passage values going from bottom-to-top across $(v^-; w^+)$, by the independence properties coming from RSK, $Z_M|_X$ and $(Z_M(v), Z_M(w))$ are conditionally independent given $Z_M|_{D(v^-; w^+)}$. Therefore
	$$(Z_M(u), Z_M|_{D(v^-; w^+)}, Z_M(v)) \qquad \mathand \qquad (Z_M(u), Z_M|_{D(v^-; w^+)}, Z_M(w))
	$$
	are Markov chains.
	To check that these Markov chains have the same transition probabilities, we just need to show that
	\begin{equation}
	\label{E:tocheck}
	(Z_M|_{D(v^-; w^+)}, Z_M(v)) \eqd (Z_M|_{D(v^-; w^+)}, Z_M(w)).
	\end{equation}
	To see this, observe that under the $180$-degree rotation of the plane that switches the two corners of $(v^-; w^+)$, that $(v, D(v^-; w^+)) \mapsto (w, D(v^-; w^+))$. Last passage percolation in any i.i.d.\ environment is invariant under this rotation (it is a product of the reflections $R_1$ and $R_2$ and a translation), yielding \eqref{E:tocheck}.
\end{proof}

A slight extension of the above argument quickly gives the symmetry (IV). Moreover, exploring these ideas further yields a whole selection of interesting invariance statements for geometric and exponential last passage percolation. 

This proof framework also works in the context of inhomogeneous geometric and exponential last passage percolation, and for the inhomogeneous log-gamma polymer. For the log-gamma polymer, the geometric RSK correspondence (see \citep{kirillov2001introduction, noumi2002tropical}) replaces RSK as the key combinatorial input. The analogue of the preservation result from \cite{DOV} in the geometric RSK context is a recent result of Corwin \citep{corwin2020invariance}. Our proof of  preservation also works for geometric RSK.

The invariance statements we prove pass through to the large collection of models that arise as limits of geometric and exponential last passage percolation and the log-gamma polymer: taseps from different initial conditions coupled with the same noise; Brownian last passage percolation; Poisson last passage percolation in the plane and the polynuclear growth model; Poisson last passage percolation across lines; the directed landscape, the Airy sheet, and the KPZ fixed point; the O'Connell-Yor Brownian polymer; and the continuum directed random polymer, the KPZ equation, and the multiplicative stochastic heat equation. 
Nonetheless, there are still many integrable models where our results do not apply (see Section \ref{S:conclude} for some discussion about such models).

The invariance statements for geometric last passage percolation also give rise to new RSK-like bijections and yield other combinatorial consequences.

\subsection{Some notation and definitions}
\label{SS:notation}

\begin{figure}[htb]
	\centering
	\begin{subfigure}[h]{0.4\textwidth}
		\includegraphics[width=\textwidth]{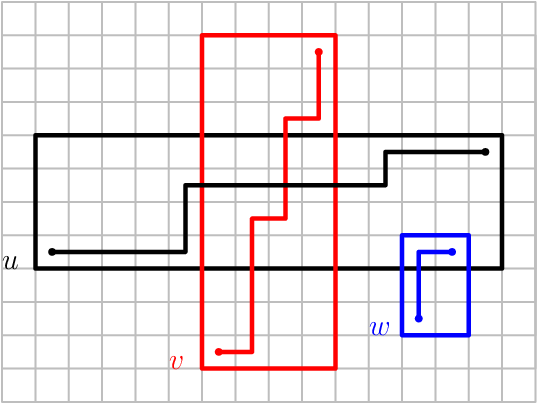}
		\caption{}
	\end{subfigure}
	\qquad
	\qquad
	\begin{subfigure}[h]{0.4\textwidth}
		\includegraphics[width=\textwidth]{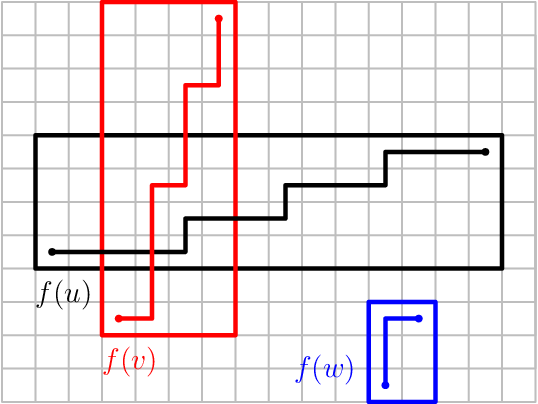}
		\caption{}
	\end{subfigure}
	\caption{In (a), three points $u, v, w \in \Zd$ are represented both as boxes in the plane and by sample paths from $u^-$ to $u^+$, $v^-$ to $v^+$ and $w^-$ to $w^+$. The pair $(u, v)$ crosses horizontally and the pair $(v, u)$ crosses vertically. The boxes $v$ and $w$ are disjoint. The map $f:\{u, v, w\} \to \{f(u), f(v), f(w)\}$ preserves horizontal and vertical crossings, but does not preserve disjointness, since $f(u)$ and $f(w)$ are disjoint but $u$ and $w$ are not disjoint.}
	\label{fig:notation}
\end{figure}

We will identify every point 
$$
u = (u^-, u^+) = (u^-_1, u^-_2; u^+_1, u^+_2) \in \Zd
$$
with the box $[u^-_1, u^-_2] \X [u^+_1, u^+_2]$. This identification allows us to apply standard set operations (e.g. intersection, unions, set difference) to points in $\Zd$. In the figures in the next section, we will also use boxes and paths to help represent distributional equalities for last passage times. 

For points $x, y \in \Z^2$ we say that $x \searrow y$ if $x_1 \le y_1, x_2 \ge y_2$. We say that an ordered pair of boxes $(u,v)$ \textbf{crosses horizontally} if $u^- \searrow v^-$ and $v^+ \searrow u^+$. Visually, a pair $(u, v)$ crosses horizontally if paths between the corners of $u$ cross the box $v$ from left to right. Similarly, an ordered pair of boxes $(u,v)$ \textbf{crosses vertically} if $v^- \searrow u^-$ and $u^+ \searrow v^+$. Visually, a pair $(u, v)$ crosses vertically if paths between the corners of $u$ cross the box $v$ from bottom to top. We say a pair of boxes crosses if the pair crosses either horizontally or vertically.
When thinking about these definitions and their relationships, we encourage the reader to keep in mind two things:
\begin{itemize}
	\item A pair $(u, v)$ crosses horizontally if and only if $(v, u)$ crosses vertically.
	\item If pairs $(u, v)$ and $(v, w)$ cross horizontally, then so does the pair $(u, w)$.
\end{itemize} 
 We say that a pair of boxes $(u, v)$ is disjoint if $u \cap v = \emptyset$.
A bijection $f$ between subsets of $\Zd$ \textbf{preserves horizontal crossings} if
$$
(u, v) \text{ crosses horizontally } \quad \iff \quad (f(u), f(v)) \text{ crosses horizontally}.
$$
Note that any bijection that preserves horizontal crossings also preserves vertical crossings; because of this we will not consider the latter concept. Similarly $f$ \textbf{preserves disjointness} if 
$$
u \cap v = \emptyset \quad \iff \quad f(u) \cap f(v) = \emptyset.
$$. 
See Figure \ref{fig:notation} for an illustration of these definitions. 

Finally, we extend the translation maps $T_c$ and the reflections $R_1, R_2$ to $\Zd$ in the most natural way: 
$$
T_c u = u + (c, c), \quad R_1 u = (R_1 u^-; R_1 u^+), \quad R_2 u = (R_2 u^+, R_2 u^-).
$$
\subsection{Invariance of last passage values}
\label{SS:symmetries-lpp}
We are now ready to present invariance theorems for exponential and geometric last passage percolation. As the theorems themselves are quite general, we encourage the reader to consider them alongside the simpler examples presented in the figures.

Our first theorem was conjectured for a variety of models by Borodin, Gorin, and Wheeler \citep{borodin2019shift} (see Conjecture 1.6 from that paper and the related discussion).

\begin{figure}[htb]
	\centering
	\begin{subfigure}[h]{0.45\textwidth}
		\includegraphics[width=\textwidth]{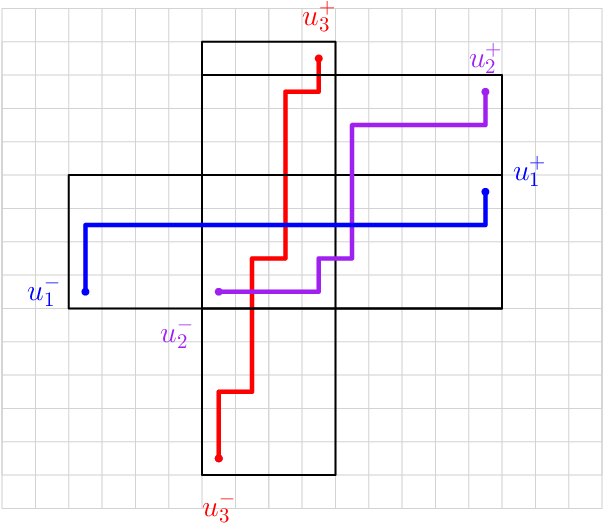}
		\caption{}
	\end{subfigure}
	\qquad
	\begin{subfigure}[h]{0.45\textwidth}
		\includegraphics[width=\textwidth]{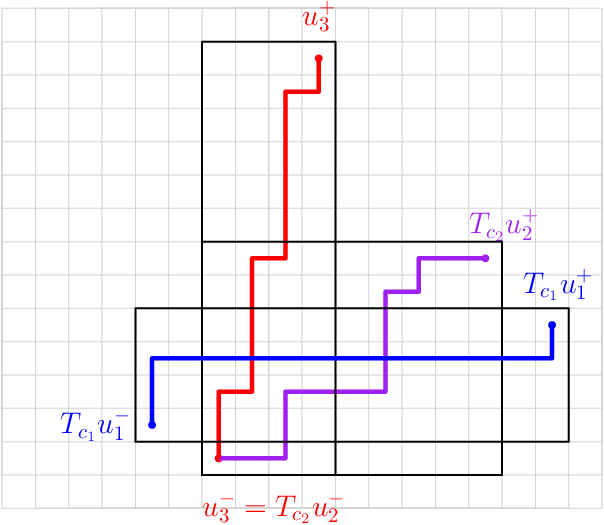}
		\caption{}
	\end{subfigure}
	\caption{An example of Theorem \ref{T:conj-bgw}.
		 The distribution of $(Z_M(u_1), Z_M(u_2), Z_M(u_3))$ in (a) is the same as the distribution of $(Z_M(T_{c_1} u_1), Z_M(T_{c_2} u_2), Z_M(u_3))$ in (b).  Using the language of Theorem \ref{T:conj-bgw}, we can take $U_i = \{u_i\}$ for $i = 1, 2, 3$.}
	\label{fig:tower-sym}
\end{figure}

\begin{theorem}
	\label{T:conj-bgw}
	Let $M$ be an environment of i.i.d. exponential or geometric random variables. Let $U_1, \dots, U_k$ be subsets of $\Zd$ such that for any $i \in \{1, \dots, k-1\}$ and any $u_i \in U_i$ and $u_{i+1} \in U_{i+1}$, the pair $(u_i, u_{i+1})$ crosses horizontally.
	
	Let $c_1, \dots, c_k \in \Z^2$, and let $f:\bigcup U_i \to \bigcup T_{c_i} U_i$ be the function with $f|_{U_i} = T_{c_i}$ for all $i$. That is, $f$ translates each of the sets of boxes $U_i$ by a different amount. Then as long as $f$ preserves horizontal crossings, the joint distribution of last passage values in $\bigcup U_i$ is also preserved by $f$. More precisely, we have
	\begin{equation}
	\label{E:orig-f}
	Z_M|_{\bigcup U_i} \eqd Z_M|_{\bigcup T_{c_i} U_i} \circ f.
	\end{equation}
\end{theorem}

When imagining the scenarios in which Theorem \ref{T:conj-bgw} applies, think of the sets $U_i$ as consisting of boxes which progressively get thinner and taller as $i$ increases.
See Figure \ref{fig:tower-sym} for a simple example of Theorem \ref{T:conj-bgw}.

Property (IV) is a special case of Theorem \ref{T:conj-bgw}. Theorem \ref{T:conj-bgw} is in turn a consequence of our most general theorem for i.i.d.\ geometric and exponential last passage percolation, Theorem \ref{T:iid-trans}. That theorem shows that \eqref{E:orig-f} holds for all $f$ in a certain class of bijections $\scrF$ between subsets of $\Zd$ (with $\bigcup U_i$ and $\bigcup T_{c_i} U_i$ replaced by the domain and codomain of $f$). 
We give two more examples of the sort of invariances that fall out of Theorem \ref{T:iid-trans}. See Figure \ref{fig:perm} and Figure \ref{fig:shift} for illustrations of these theorems.

\begin{theorem}
	\label{T:example}
	Let $M$ be an environment of i.i.d. exponential or geometric random variables. Let $U = U_1 \cup U_2 \dots \cup U_k, V = V_1 \cup \dots \cup V_m \sset \Zd$ be such that 
	\begin{itemize}
		\item Every pair of boxes $(u, v) \in U \X V$ crosses horizontally.
		\item For boxes $u \in U_j$ and $u' \in U_i$ with $i \ne j$, we have $u \cap u' = \emptyset$. Visually, we can imagine the $U_i$ as forming separate disjoint `components' of $U$. 
		\item Similarly, for boxes $v \in V_j$ and $v' \in V_i$ with $i \ne j$, we have $v \cap v' = \emptyset$.
	\end{itemize}
	Let $f:U \cup V \to \bigcup T_{c_i} U_i \cup \bigcup T_{d_j} V_j$ be a function with $f|_{U_i} = T_{c_i}$ and $f|_{V_j} = T_{d_j}$ for some $c_i, d_j \in \Z^2$, and suppose that $f$ preserves horizontal crossings and disjointness. That is, $f$ can be any function that separately translates each of the components of $U$ and each of the components of $V$ in such a way that $f$ preserves the horizontal crossing structure between $U$ and $V$, and the disjointness of different components of these sets.
	Then the distribution of last passage values is also preserved by $f$:
	$$
	Z_M|_{U \cup V} \eqd Z_M|_{\bigcup T_{c_i} U_i \cup \bigcup T_{d_j} V_j} \circ f.
	$$
\end{theorem}

When imagining the kind of scenarios to which Theorem \ref{T:example} applies, think of the boxes in $U$ as being short and wide, and the boxes in $V$ as being thin and tall. We can imagine that these two sets of boxes cross each other in a  kind of `lattice structure'. Last passage values are preserved under any rearrangement of the components of $U$ and $V$ that preserves this lattice structure. See Figure \ref{fig:perm} for a simple example.

\begin{figure}[htb]
	\centering
	\begin{subfigure}[h]{0.4\textwidth}
		\includegraphics[width=\textwidth]{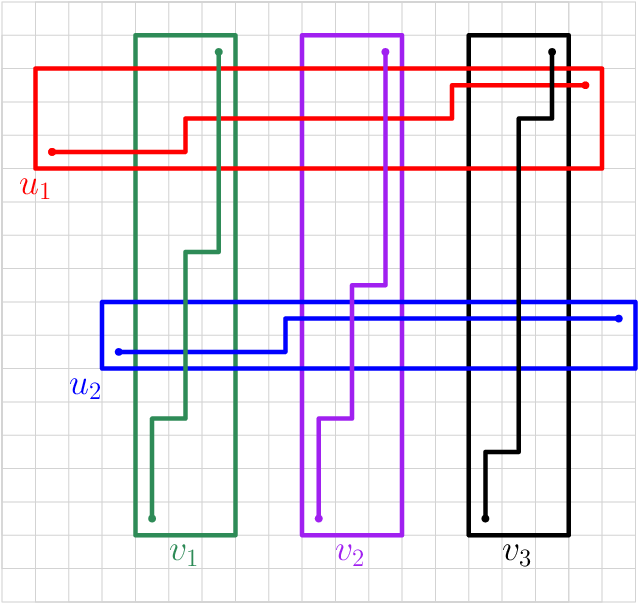}
		\caption{}
	\end{subfigure}
	\qquad
	\qquad
	\begin{subfigure}[h]{0.4\textwidth}
		\includegraphics[width=\textwidth]{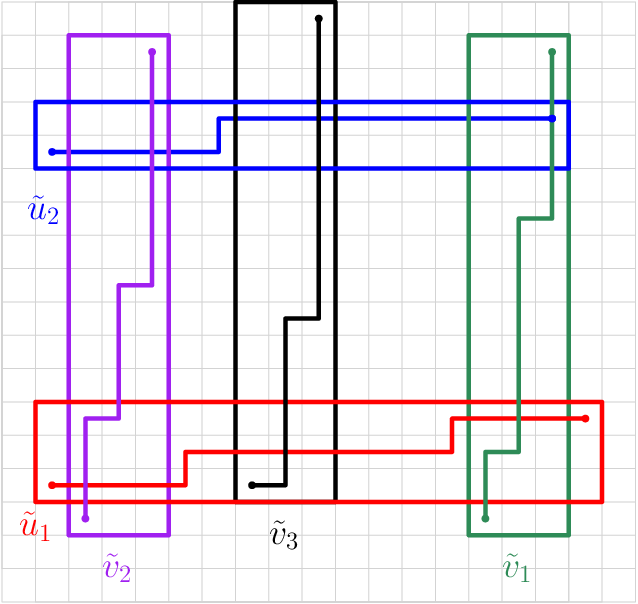}
		\caption{}
	\end{subfigure}
	\caption{An example of Theorem \ref{T:example}. In the language of that theorem, we take $U_i = \{u_i\}, V_i = \{v_i\}$ and $f(u_i) = \tilde u_i, f(v_i) = \tilde v_i$ for all $i$. At the level of individual boxes, $f$ is a translation. Moreover, $f$ preserves the disjointness of pairs $u_i, u_j$ and $v_i, v_j$ for $i \ne j$, and the fact that every pair $(u_i, v_j)$ crosses horizontally. Therefore $f$ preserves last passage values: $(Z_M(u_1), Z_M(u_2), Z_M(v_1), Z_M(v_2), Z_M(v_3)) \eqd (Z_M(\tilde u_1), Z_M(\tilde u_2), Z_M(\tilde v_1), Z_M(\tilde v_2), Z_M(\tilde v_3))$.}
	\label{fig:perm}
\end{figure}

\begin{theorem}
	\label{T:puzzle-pieces}
	Let $M$ be an environment of i.i.d.\ exponential or geometric random variables. Let $U, W \sset \Zd$. Suppose that we can partition $U = U_1 \cup U_2$ and $W = W_1 \cup W_2$ such that 
	\begin{itemize}
		\item Every pair of boxes $(u, w) \in U_1 \X W_1$ crosses horizontally.
		\item Every pair of boxes $(u, w) \in U_2 \X W_2$ crosses vertically.
		\item Every pair of boxes $(u, w)$ in either $U_1 \X W_2$ or $U_2 \X W_1$ is disjoint.
	\end{itemize}
 Now let $c \in \{(\pm 1, 0), (0, \pm 1)\}$, and define $f:U \cup W \to T_c U \cup W$ by $f|_U = T_c, f|_W = \id$. Suppose that $f$ preserves horizontal crossings and disjointness. Then
	$$
	Z_M|_{U \cup W} \eqd Z_M|_{T_cU \cup W} \circ f.
	$$
\end{theorem}

Theorem \ref{T:puzzle-pieces} is again best illustrated through a simple example, see Figure \ref{fig:shift}. One way of imagining Theorem \ref{T:puzzle-pieces} and the example above is that the sets of $U$- and $W$-boxes are two rigid interlocking puzzle pieces. We can slide these two pieces around in the plane without changing the joint distribution of last passage values, so long as we do not break the crossing or disjointness structure. 

\begin{figure}[tb]
	\centering
	\begin{subfigure}[h]{0.4\textwidth}
		\includegraphics[width=\textwidth]{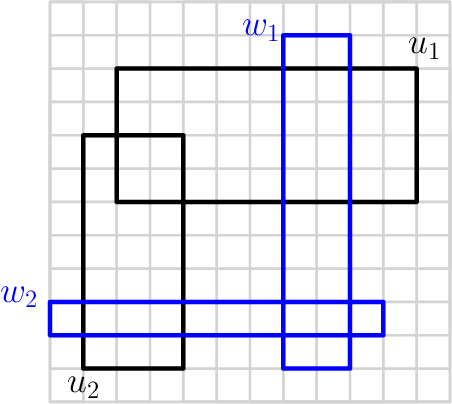}
		\caption{}
	\end{subfigure}
	\qquad
	\qquad
	\begin{subfigure}[h]{0.4\textwidth}
		\includegraphics[width=\textwidth]{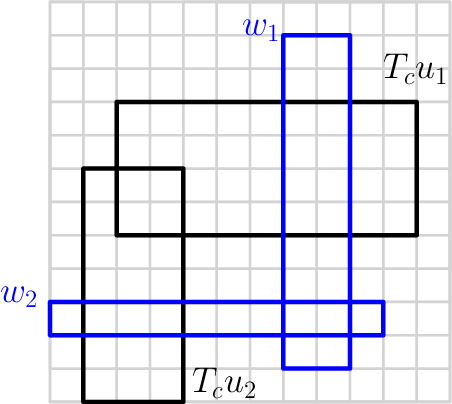}
		\caption{}
	\end{subfigure}
	
	\caption{An example of Theorem \ref{T:puzzle-pieces}. In that theorem, we take $U_i = \{u_i\}, W_i = \{w_i\}$ and $c = (0, - 1)$. From the picture, we can easily check the crossing and disjointness conditions of that theorem, and so $(Z_M(u_1), Z_M(u_2), Z_M(w_1), Z_M(w_2)) \eqd (Z_M(T_c u_1), Z_M(T_c u_2), Z_M(w_1), Z_M(w_2))$.}
	\label{fig:shift}
\end{figure}

As mentioned previously, the proof framework also works for integrable inhomogeneous last passage models. Inhomogenous integrable last passage models are indexed by two biinfinite sequences $\al, \be$. The corresponding environment $M_{\al, \be}$ consists either of independent geometric random variables where $M_{(i, j)}$ has parameter $\al_i \be_j$ (i.e. $\p(M_{(i, j)} = n) = (1- \al_i \be_j) (\al_i \be_j)^n$ for $n = 0, 1, \dots$), or of exponential random variables $M_{(i, j)}$ with mean $(\al_i + \be_j)^{-1}$. The analogue of Theorem \ref{T:conj-bgw} in this case is the following.

\begin{theorem}
	\label{T:inhom-ex}
	Let $f:\bigcup U_i \to \bigcup T_{c_i} U_i$ be a function satisfying the conditions of Theorem \ref{T:conj-bgw}, and let $\al, \be$ be biinfinite sequences indexing an environment of exponential or geometric random variables $M_{\al, \be}$. Then there exist rearrangements $\al', \be'$ of $\al, \be$ such that
	$$
	Z_{M_{\al, \be}}|_{\bigcup U_i} \eqd Z_{M_{\al', \be'}}|_{T_{c_i} U_i} \circ f.
	$$
\end{theorem}
The precise restriction on the types of allowable rearrangements $\al', \be'$ in Theorem \ref{T:inhom-ex} is a generalization of the following well-known fact. For a box $u$ and two environments of independent geometric or exponential random variables indexed by $(\al, \be)$ and $(\al', \be)$, we have
$
	Z_{M_{\al, \be}}(u) \eqd Z_{M_{\al', \be'}}(u)
$
if and only if
$$
(\al'_{u^-_1}, \dots, \al'_{u^+_1}) = \sig (\al_{u^-_1}, \dots, \al_{u^+_1}), \quad \mathand \quad (\beta'_{u^-_2}, \dots, \beta'_{u^+_2}) = \tau (\beta_{u^-_2}, \dots, \beta_{u^+_2})
$$
for some permutations $\sig$ and $\tau$. In the case of geometric last passage percolation, this is again a consequence of \eqref{E:Schur}, and the fact that Schur functions are symmetric.
For $(\al', \be')$ to be an allowable rearrangement in Theorem \ref{T:inhom-ex} we require this sort of condition to hold for many boxes. More precisely, we require the following.
\begin{itemize}
	\item For each $i$, let $I(i)^-_1$ be the smallest interval
	containing $\{u^-_1 : u \in U_i\}$. Similarly define $I(i)^-_2, I(i)^+_1,$ and $I(i)^+_2$. Then for any $i$, and any $a \in I(i)^-_1, b \in I(i)^+_1, c \in I(i)^-_2, d \in I(i)^+_2$, there exist permutations $\sig$ and $\tau$ such that
	$$
	(\al'_{a + c_{i, 1}}, \dots, \al'_{b + c_{i, 1}}) = \sig (\al_a, \dots, \al_b), \quad \mathand \quad (\beta'_{c + c_{i, 2}}, \dots, \beta'_{d + c_{i, 2}}) = \tau (\beta_{c}, \dots, \beta_{d}).
	$$
	\end{itemize}

Finally, in both the inhomogeneous and homogeneous cases the last passage invariance statements extend naturally to certain collections of multi-point last passage values. See Section \ref{S:main} for precise statements.

\subsection{Invariance of related objects}
\label{SS:symmetries}
The invariance statements from Section \ref{SS:symmetries-lpp} are strong enough to yield interesting symmetries of last passage percolation with initial conditions, last passage path locations, disjointness probabilities, and other objects. Here we give a few sample results to illustrate this (note that stronger statements are easily possible, if desired). Most of these results follow straightforwardly from Theorem \ref{T:iid-trans} and its special cases.

For an environment $M$, sets $D_1, D_2 \sset \Z^2$ with $D_1\X D_2 \cap \Zd \ne \emptyset$, and functions $f:D_1 \to \R, g:D_2 \to \R$, define \textbf{last passage value with boundary conditions} $f$ and $g$ by
$$
Z_M(f, g) = \max_{(u^-; u^+) \in (D_1\X D_2) \cap \Zd} \; \max_{\pi:u^- \to u^+} \big( f(u^-) + M(\pi) +g(u^+) \big).
$$
Theorem \ref{T:conj-bgw} yields the following corollary.

\begin{corollary} [Initial conditions]
	\label{C:conj-bgw-init} 
	Let $M$ be an environment of i.i.d. exponential or geometric random variables. Let $U, V, W \sset \Zd$ and $c \in \Z^2$ and suppose that 
	\begin{itemize}
		\item Every pair of boxes $(u, v)$ in either $U \X V$ or $V \X W$ crosses horizontally.
		\item The same is true after $V$ is shifted by $c$. That is, every pair of boxes $(u, v)$ in either $U \X T_cV$ or $T_cV \X W$ crosses horizontally.
	\end{itemize}
	Now let $V^- = \{v^- : v \in V\}$ and $V^+ = \{v^+ : v \in V\}$ be the set of bottom left corners and top right corners of boxes in $V$. Let $f:V^- \to \R$ and $g:V^+ \to \R$ be any boundary conditions.
	Then the joint distribution of the last passage values for $U$ and $W$, and the last passage value with boundary conditions $f$ to $g$, is unchanged under the shift of $V^-$ and $V^+$ by $c$. More precisely,
	$$
	(Z_M|_{U \cup W}, Z_M(f, g)) \eqd (Z_M|_{U \cup W}, Z_M(f \circ T_{-c}, g \circ T_{-c})).
	$$
\end{corollary}

Next, we state a corollary of Theorem \ref{T:conj-bgw} that concerns the geometry of last passage geodesics. For $u  \in \Zd$, we say that a $u$-path $\pi$ is a \textbf{$u$-geodesic} if $M(\pi) = Z_M(u)$. A standard path-crossing argument shows that there exists a unique $u$-geodesic $\pi_M(u) \in \scrP_M(u)$ such that for any $\pi' \in \scrP_M(u)$ and any $v \in \pi'$, there exists a $v' \in \pi$ such that $v \searrow v'$. The path $\pi_M(u)$ is called the \textbf{leftmost $u$-geodesic}.

\begin{figure}[tb]
	\centering
	\begin{subfigure}[h]{0.4\textwidth}
		\includegraphics[width=\textwidth]{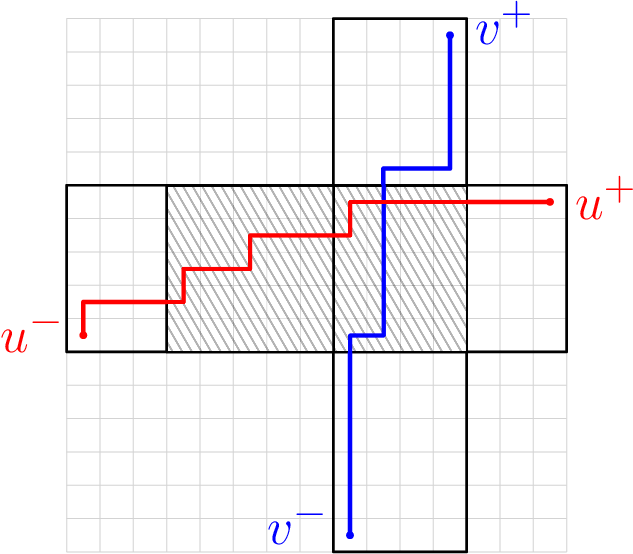}
		\caption{}
	\end{subfigure}
	\qquad
	\qquad
	\begin{subfigure}[h]{0.4\textwidth}
		\includegraphics[width=\textwidth]{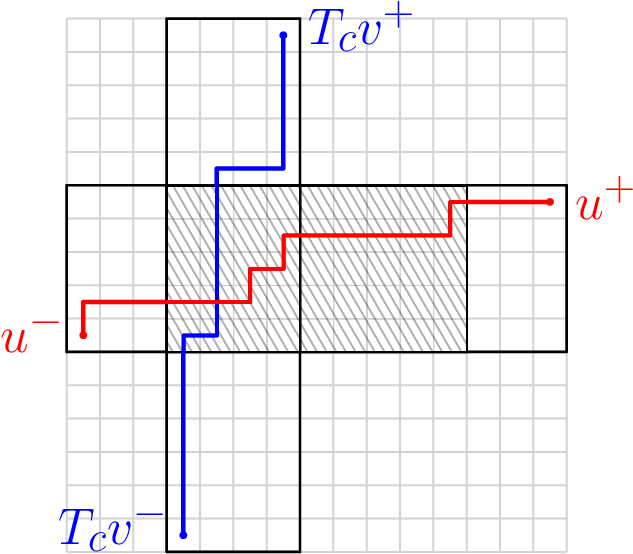}
		\caption{}
	\end{subfigure}
	\caption{An example of Corollary \ref{C:geometry-cor}. To apply that theorem, we set $U = \{u\}$ and $V = \{v\}$. The intersection $u \cap v$ can be taken as the box $x$, and the box $w$ can be taken as the smallest box containing both $u \cap v$ and $u \cap T_c v$; this is the shaded region.
		The joint distribution of the portion of the paths $\pi_M(u)$ and $\pi_M(v)$ that lie outside of the shaded region $w$ in (a) is the same as the joint distribution of the portion of the paths $\pi_M(u)$ and $\pi_M(T_c v)$ that lie outside of the shaded region in (b).}	
	\label{fig:geod}
\end{figure}

\begin{corollary}[Geodesic structure]
	\label{C:geometry-cor}
	Let $M$ be an environment of i.i.d.\ exponential or geometric random variables. Let $x, w \in \Zd, U, V \sset \Zd$, and $c \in \Z^2$ be such that 
	\begin{itemize}
		\item For every box $u \in U$, the pair $(u, w)$ crosses horizontally.
		\item The pairs $(w, x)$ and $(w, T_c x)$ cross horizontally.
		\item For every box $v \in V$, the pair $(x, v)$ crosses horizontally.
	\end{itemize} Then we have
	\begin{equation}
	\label{E:piMu}
	[(\pi_M(u) \smin w : u \in U), (\pi_M(v) \smin x : v \in V)] \eqd [(\pi_M(u) \smin w : u \in U), (\pi_M(T_cv) \smin T_c x : v \in V)].
	\end{equation}
\end{corollary}

For the equality in distribution in Corollary \ref{C:geometry-cor}, $\pi_M(u) \smin w$ is well-defined when we think of both $\pi_M(u), w$ as subsets of $\Z^2$. Corollary \ref{C:geometry-cor} also implies stationarity properties for branch points and coalescence points for pairs of leftmost geodesics in $U$ and $V$. See Figure \ref{fig:geod} for an illustration of this corollary.


We can also use Theorem \ref{T:iid-trans} to prove stationarity results about disjointness events and paths restricted to stay in particular regions. We refer the reader to Corollaries \ref{C:restricted-polymer} and \ref{C:disjointness} for precise sample statements.

\subsection{Some combinatorial consequences}
\label{SS:combinatorial}
We can use Theorem \ref{T:iid-trans} to obtain new versions of the RSK bijection and other interesting combinatorial consequences. To state these consequences, we first recall some definitions.

A \textbf{partition} is a weakly decreasing sequence $\la = (\la_1, \la_2, \dots \la_k)$ of positive integers. The size of the partition is $|\la| = \sum_{i=1}^k \la_i$. We will also identify (possibly finite) weakly decreasing sequences of the form $(\la_1, \dots, \la_k, 0, 0,\dots)$ with the partition $\la$. To any partition $\la$, the \textbf{Young diagram} associated to $\la$ is the set of squares $\{(i, j) \in \Z^2 : 1 \le i \le \la_j \}$, see Figure \ref{fig:poly}. By associating a set to every partition, we have a naturally defined notion of containment, $\la \sset \mu$, and we can define the difference $\la \smin \mu$ as a subset of $\Z^2$. For two partitions $\mu \sset \la$, we say that $\la \smin \mu$ is a \textbf{horizontal strip} if for all $i$, it contains at most one square in the column $\{i\} \X \Z$.

A \textbf{semistandard Young tableau} of shape $\la$ is a sequence of partitions $(\emptyset = \la_0, \la_1, \dots, \la_k = \la)$ where for all $i$, $\la_{i-1} \sset \la_i$ and $\la_i \smin \la_{i-1}$ is a horizontal strip. Equivalently, a semistandard Young tableau of shape $\la$ is a filling of the corresponding Young diagram with positive integers such that the entries are strictly increasing along columns and weakly increasing along rows. 

For a point $u \in \Zd$ and $k \le \min(u^+_1 - u^-_1 + 1, u^+_2 - u^-_2 + 1)$ define
$$
Z^{\Delta k}_M(u) = Z_M(u^k) - Z_M(u^{k-1}),
$$
where $Z_M(u^0) := 0$. The quantities $Z^{\Delta k}_M(u)$ are nonincreasing in $k$ when $M$ has nonnegative entries. In particular, 
$$
Z^\Delta_M(u) :=(Z^{\Delta 1}_M(u), Z^{\Delta 2}_M(u), \dots)
$$
is a partition whenever $M$ has nonnegative integer entries. By Greene's theorem \citep{greene1974extension}, the RSK bijection relates matrices $M$ with nonnegative entries to pairs of semistandard Young tableaux by recording certain partition sequences of the form $Z^\Delta_M(u)$. Our framework yields other bijections of this form.

Let $\scrX_{n, m}$ be the set of $n \X m$ matrices $M$ with nonnegative integer entries. We think of a matrix $M \in \scrX_{n, m}$ as a weight environment  on the box $\{1, \dots, n\} \X \{1, \dots, m\}$ so that we can study last passage percolation using the weights given by $M$. Let $\Om_{n, \la}$ be the space of semi-standard Young tableaux of length $n$ and shape $\la$: $\emptyset = \la_0 \sset \la_1 \sset \dots \sset \la_n = \la$. Now consider sequences of intervals 
$$
I = I_1 \sset \dots \sset I_n = \{1, \dots, n\} \quad \mathand \quad J = J_1 \sset \dots \sset J_n = \{1, \dots, m\}
$$
where $|J_i| = i, |I_i| = i$. Let $\bar I_i = I_i \X [1, m], \bar J_i = [1, n] \X J_i \in \Zd$. 
Define maps $\Phi_I, \Psi_J$ mapping $\scrX_{n, m}$ to the space of semistandard Young tableaux by
\begin{align*}
\Phi_I(M) &= (\emptyset, Z^\Delta_M(\bar I_1), Z^\Delta_M(\bar I_2), \dots, Z^\Delta_M(\bar I_n)) \quad \mathand \\
\quad \Psi_J(M) &= (\emptyset, Z^\Delta_M(\bar J_1), Z^\Delta_M(\bar J_2), \dots, Z^\Delta_M(\bar J_n)).
\end{align*}

\begin{theorem}[The scrambled RSK bijection]
	\label{T:rsk-scramble}
	Fix $n, m \in \N$. Then for any sequences $I, J$ as above, the map 
	$$
	M \mapsto (\Phi_I(M), \Psi_J(M))
	$$
	is a bijection from $\scrX_{n, m} \to \bigcup_\la \Om_{n, \la} \X \Om_{m, \la}$, where the union is over all partitions $\la$. Moreover, for all $i \in [1, n], j \in [1, m]$, we have
	$$
	\sum_{\ell=1}^m M_{i, \ell} = |\Phi_I(M)_{\sig(i)} \smin \Phi_I(M)_{\sig(i) - 1} |  \quad \mathand \quad \sum_{\ell=1}^n M_{\ell, j} = |\Psi_J(M)_{\tau(j)} \smin \Psi_J(M)_{\tau(j) - 1}|,
	$$
	where $\sig(i), \tau(j)$ are the smallest indices such that $i \in I_{\sig(i)}, j \in J_{\tau(j)}$. 
\end{theorem}

The case when $I_i = [1, i]$ and $J_i = [1, j]$ in Theorem \ref{T:rsk-scramble} is the usual RSK bijection. The functions $\Phi_I(M)$ and $\Psi_J(M)$ are the two Young tableaux, and $Z_M^\Delta(\{1, \dots, n\} \X \{1, \dots, m\})$ is their common shape.
If we set $I_i = [n- i + 1, n]$ or $J_i = [m-i + 1, m]$, then the resulting maps $\Phi_I$ and $\Phi_J$ are compositions of the usual RSK bijection with a Sch\"utzenberger involution. Setting $I_i = [1, i]$ and $J_i = [m-i + 1, m]$ yields the Burge (or Hillman-Grassl) correspondence \cite{burge1974four} (see Krattenthaler \cite{krattenthaler2005growth} for the connection between this correspondence and increasing subsequences).
If we set $n = m$ and restrict to the set of permutation matrices, then we recover scrambled versions of the Robinson-Schensted bijection between permutations and pairs of standard Young tableaux. 

The scrambled RSK bijections have previously been described in different language (and with entirely orthogonal proofs) by Garver, Patrias, and Thomas \cite{garver}. In that paper, the authors developed bijections between sets of quiver representations and reverse plane partitions. Section 6 of \cite{garver} shows how these bijections for particular type $A$ Dynkin quivers yield the RSK and Burge correspondences. More generally, there is an exact correspondence between scrambled RSK bijections on an $n \X m$ box and these bijections for Dynkin quivers of type $A_{n + m - 1}$ with miniscule vertex $n$. Very roughly, the direction of each arrow in the Dynkin quiver encodes whether each of the intervals $I_i, J_i$ should grow to the left or to the right.

The classical RSK bijection has a generalization from  rectangular arrays of nonnegative integers to fillings of Young diagrams with nonnegative integers, see \citep{krattenthaler2005growth}. Our framework allows us to generalize the RSK bijection to an even larger class of shapes, known as moon polyominoes.
A \textbf{moon polyomino} is a finite subset $S \sset \Z^2$ with the following properties (see Figure \ref{fig:poly}):
\begin{itemize}
	\item (Convexity) For every $i \in \Z$, let $S^i = \{ x \in \Z: (x, i) \in S\}$ and let $S_i =  \{ x \in \Z: (i, x) \in S\}$. Then for all $i \in \Z$, the sets $S^i$ and $S_i$ are (possibly empty) intervals. 
	\item (Intersection-free) For all $i, j \in \Z$, either $S^i \sset S^j$ or $S^j \sset S^i$. Similarly, either $S_i \sset S_j$ or $S_j \sset S_i$.
\end{itemize}

\begin{figure}[tb]
	\centering
	\begin{subfigure}[h]{0.15\textwidth}
		\includegraphics[width=\textwidth]{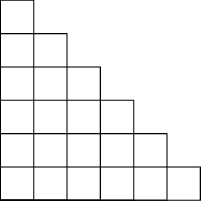}
		\caption{}
	\end{subfigure}
	\qquad \qquad 
	\begin{subfigure}[h]{0.15\textwidth}
		\includegraphics[width=\textwidth]{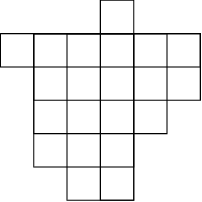}
		\caption{}
	\end{subfigure}
	\caption{The Young diagram of a partition and a moon polyomino with the same row lengths.}
	\label{fig:poly}
\end{figure}

To set up the RSK bijection for moon polyominoes, we will use the following property (see Lemma \ref{L:moon}). For any moon polymino $S$, there is a sequence $(u_0 = \emptyset, u_1, \dots, u_k, u_{k+1} = \emptyset) \sset \Zd$ such that for all $i$, $u_{i+1}$ is obtained from $u_i$ by either adding a row or subtracting a column, and
$$
\bigcup_{i=1}^k u_i = S.
$$
We call the sequence $(u_0, u_1, \dots, u_k, u_{k+1})$ a \textbf{box exhaustion} of $S$. We let $\scrX_S$ be the set of functions $M:S \to \Z_{\ge 0}$.

\begin{theorem}
	\label{T:rsk-moon-poly}
	Let $S$ be any moon polyomino, and let $U = (u_0, u_1, \dots, u_{k+1})$ be a box exhaustion of $S$. 
	For $M \in \scrX_S$, define the map
	$$
	\Phi_U(M) = (\emptyset, Z^\Delta_M(u_1), \dots, Z^\Delta_M(u_k), \emptyset).
	$$
	Then $\Phi_U$ is a bijection between $\scrX_S$ and the set of partition sequences $(\emptyset = \la_0, \la_1, \dots, \la_k, \emptyset = \la_{k+1})$ such that $\la_i \sset \la_{i+1}$ and $\la_{i+1} \smin \la_i$ is a horizontal strip whenever $u_i \sset u_{i+1}$, and $\la_i \supset \la_{i+1}$ and $\la_{i} \smin \la_{i+1}$ is a horizontal strip whenever $u_i \supset u_{i+1}$. Moreover, for every $i = 0, \dots, k$,
	$$
	|\Phi_U(M)_i \;\symdif \;\Phi_U(M)_{i+1}| = \sum_{x \in u_i \symdif u_{i+1}} M_x.
	$$ 
	In the above display, we use the notation $\symdif$ to denote the symmetric difference of both partitions (thought of as subsets of $\Z^2$) and of boxes.
\end{theorem}

Theorem \ref{T:rsk-scramble} is a special case of Theorem \ref{T:rsk-moon-poly}. Each pair $(I, J)$ in that theorem represents a different box exhaustion $U$ for the $n \X m$ rectangle.

Theorem \ref{T:rsk-moon-poly} implies interesting enumerative consequences about fillings of moon polyominoes. In particular, it gives a bijection between nonnegative integer fillings of moon polyominoes with the same set of row lengths. This bijection preserves all last passage values that can be read off from a component of $\Phi_U(M)$. A bijection preserving a certain subset of these last passage values was previously discovered by Rubey (see \citep{rubey2011increasing}, Theorem 5.3), building on work of Jonsson \cite{jonsson2005generalized} and Krattenthaler \citep{krattenthaler2005growth}. It would be interesting to compare the purely combinatorial methods of that paper to the more probabilistic methods used here. Note that moon polyominoes have also been studied under the name L-convex polyominoes (e.g. see \citep{castiglione2003reconstruction}).

\subsection{Directed polymers}
\label{SS:directed-limit}
The framework used to prove invariance results of last passage percolation also works in the positive temperature setting of directed polymers, where the operations $(\max, +)$ are replaced by $(+, \X)$. 

Let $M = \{M_x : x \in \Z^2\}$ be an array of nonnegative random variables, and for $u \in \Zd$, define the \textbf{polymer partition function} 
$$
Z_M(u) = \sum_\pi \prod_{x \in \pi} M_x,
$$
where the sum is over all $u$-paths $\pi$. Just as in the last passage case, we can also define multi-point partition functions. When $M$ consists of i.i.d inverse-gamma random variables, then this is the log-gamma polymer first introduced by Sepp\"al\"ainen \citep{seppalainen2012scaling}. Throughout the paper, we use the same notation for last passage values and polymer partition functions as the two cases will be treated together. All our main theorems hold for the log-gamma polymer.

\begin{theorem}
	\label{T:polymers}
	Theorems \ref{T:conj-bgw}, \ref{T:example}, \ref{T:puzzle-pieces}, and Corollary \ref{C:conj-bgw-init} hold when $M$ is an environment of i.i.d. inverse-gamma random variables and $Z_M$ is the polymer partition function. Theorem \ref{T:inhom-ex} holds when $M$ is an environment of independent random variables where each $M_{(i, j)}$ has inverse-gamma distribution with parameter $\al_i + \be_j$ and $\al_i + \be_j > 0$ for all $i, j$. That is,
	$$
	\p(M_{(i, j)} \in dx) = \frac{1}{\Ga(\al_i + \be_j)} x^{-(1 + \al_i + \be_j)}e^{-1/x} dx
	$$
	for $x \in (0, \infty)$.
\end{theorem}

Next, we write down an analogue of Corollary \ref{C:geometry-cor} in the polymer case.
The analogue of the last passage geodesic in the directed polymer setting is the \textbf{quenched polymer measure} $Q_M(u)$. For $u \in \Zd$, this is a random probability measure on $u$-paths, where
$$
Q_M(u)(\pi) = \frac{1}{Z_M(u)} \prod_{v \in \pi} M_v.
$$
For this corollary, for $S \sset \Z^2$ we let $(P_S)_* Q_M(u)$ denote the pushforward of the measure $Q_M(u)$ under the projection $\pi \mapsto \pi \cap S$. 

\begin{corollary}[Quenched measure invariance]
	\label{C:quenched-cor}
	Let $M$ be an environment of i.i.d.\ inverse-gamma random variables, and let $Q_M$ be the corresponding quenched polymer measure. Let $x, w \in \Zd; \;U, V \sset \Zd;$ and $c \in \Z^2$ satisfy the same conditions as in Corollary \ref{C:geometry-cor}. Then
	\begin{align*}
	&\big[((P_{u \smin w})_* Q_M(u) : u \in U), ((P_{v \smin x})_* Q_M(v) : v \in V)\big] \\
	\eqd &\big[((P_{u \smin w})_* Q_M(u) : u \in U), ((P_{T_c v \smin T_c x})_* Q_M(v) : v \in V)\big].
	\end{align*}
\end{corollary}

The correspondences in Section \ref{SS:combinatorial} also have analogues in the polymer setting which we do not explore in this paper. Just as with the geometric RSK correspondence, we expect that these correspondences are not only bijective, but also volume-preserving in log-log variables (see \cite{o2014geometric}). For the geometric analogue of the Burge correspondence, this volume-preservation property was recently shown by Bisi, O'Connell, and Zygouras \cite{bisi2020geometric}. Bisi et al. \cite{bisi2020geometric} also obtained a distributional identity between the partition function of the symmetric and persymmetric log-gamma polymer which is closely related to some of the invariance statements we study here. (Note that our invariance statements do not imply this identity since they do not apply to symmetric or persymmetric polymer environments).

\subsection{Limiting models}
\label{SS:limits}

There are many models that fall out through bijections or limits of geometric and exponential last passage percolation and the log-gamma polymer. 

For example, exponential last passage percolation can be coupled to multiple taseps evolving with the same noise in a standard way. Brownian last passage percolation, planar Poisson last passage percolation and Poisson lines last passage percolation are all limits of geometric last passage percolation via straightforward limiting procedures (see \citep{dauvergne2019uniform}, Section 6 for descriptions). 

The Airy sheet, the directed landscape, and the KPZ fixed point are (conjecturally universal) scaling limits. The Airy sheet and the directed landscape have been proven to arise as scaling limits of Brownian last passage percolation  \citep{DOV} and the KPZ fixed point has been proven to arise as the scaling limit of tasep  \citep{matetski2016kpz}. 

The O'Connell-Yor Brownian polymer (see \citep{o2001brownian}) arises as the limit of the log-gamma polymer via the same procedure that takes geometric last passage percolation to Brownian last passage percolation (see \citep{borodin2019shift}, Section 7.3 for a description). The continuum directed random polymer/multiplicative stochastic heat equation/KPZ equation arises in a certain limit of the log-gamma polymer. See \citep{alberts2014intermediate, alberts2014continuum} for background and convergence of single-point partition functions and \citep{o2016multi, corwin2017intermediate} for a description and convergence of certain multi-point partition functions.

Our results also apply to all of these models. Appropriate statements can be found by examining how certain invariances are transformed under limiting procedures or certain mappings. 

We only include two of the more striking results here for illustrative purposes. For these corollaries, $\scrS:\R^2 \to \R$ is the Airy sheet, a random continuous function that arises as a scaling limit of the last passage percolation $Z_M(\cdot, 1; \cdot, n)$ as $n \to \infty$. See Section \ref{SS:lim-proof} for a more precise setup. For this next result, for sets $S, T \sset \R^2$, we say that $S \searrow T$ if $s_1 \le t_1, s_2 \ge t_2$ for all $s \in S, t \in T$. 
We also define translations $T_r(x, y) = (x + r, y+r)$.

\begin{corollary}
	\label{C:airy-sheet}
	Let $S_1, \dots, S_k$ be Borel measurable subsets of $\R^2$. Let $r_1, \dots, r_k \in \R,$ and let  $f:\bigcup S_i \to \bigcup T_{r_i} S_i$ be the map translating each $S_i$ to $T_{r_i} S_i$. Suppose that for all $i =1, \dots, k-1$, we have
	$
	S_i \searrow S_{i+1}$ and  $T_{r_i} S_i \searrow T_{r_{i+1}} S_{i+1}.
	$
Then as random continuous functions,
	$$
	\scrS|_{\bigcup S_i} \eqd \scrS|_{\bigcup T_{r_i} S_i} \circ f.
	$$
\end{corollary}

This next corollary only actually requires the distributional equality (IV) for Brownian last passage percolation (which had appeared in \citep{borodin2019shift}). We include it anyways for its aesthetic appeal.

\begin{corollary}
	\label{C:airy-processes}
	Let $g:\R \to \R$ be any nonincreasing function. Let 
	$$
	\Ga(g) = \{(x, y) \in \R^2: \lim_{t \to x^-} g(t) \ge y \ge \lim_{t \to x^+} g(t) \}
	$$
	be the graph of $g$, connected by vertical lines at all discontinuities. Let $m = (m_1, m_2): \R \to \R^2$ be the unique function satisfying
	\begin{itemize}[nosep]
		\item $m(\R) = \Ga(g)$,
		\item $m_1(0) = m_2(0)$ and $m_1(t) > m_2(t)$ for $t > 0$,
		\item For $x, y \in \R$, we have $||m(x) - m(y)||_1 = |x - y|$. Here $||u||_1$ is the $L^1$ norm on $\R^2$.
	\end{itemize}
	Then $\scrS \circ m:\R\to \R$ has the same distribution as $\scrS(0, \cdot): \R \to \R$. That is, $\scrS \circ m$ is a parabolically shifted Airy process. 
\end{corollary}

\begin{figure}[tb]
	\centering
	\begin{subfigure}[h]{0.4\textwidth}
		\includegraphics[width=\textwidth]{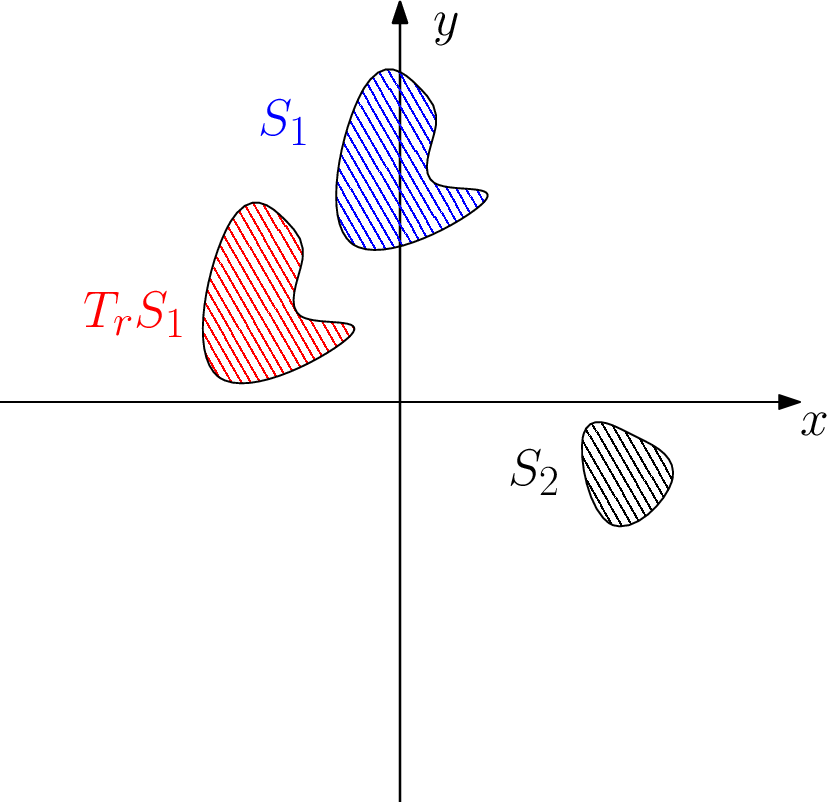}
		\caption{}
	\end{subfigure}
	\qquad
	\qquad
	\begin{subfigure}[h]{0.4\textwidth}
		\includegraphics[width=\textwidth]{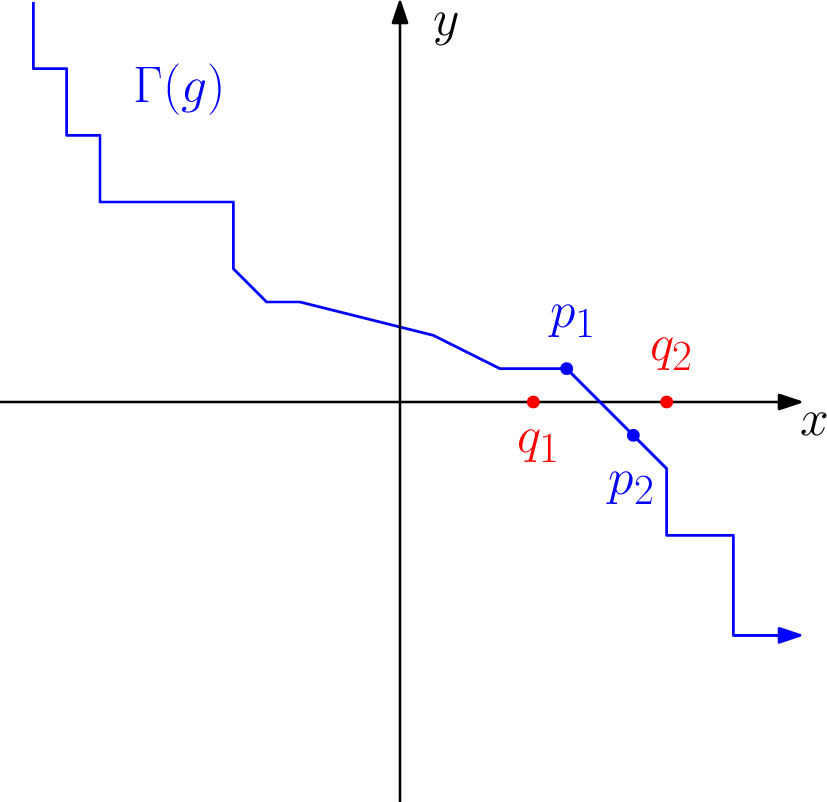}
		\caption{}
	\end{subfigure}
	
	\caption{Corollary \ref{C:airy-sheet} is illustrated in (a). The joint distribution of the Airy sheet $\scrS$ on $S_1$ and $S_2$ is the same as the joint distribution of $\scrS$ on $T_r S_1$ and $S_2$. Corollary \ref{C:airy-processes} is illustrated in (b). After an appropriate parametrization $m$ of $\Ga(g)$, the process $\scrS \circ m$ has the same distribution as the two axial parabolic Airy processes $\scrS(0, \cdot)$ and $\scrS(\cdot, 0)$. The correct parametrization is one that preserves $L^1$-distance. For example, this implies that $(\scrS(p_1), \scrS(p_2)) \eqd (\scrS(q_1), \scrS(q_2))$.}
	\label{fig:airy}
\end{figure}

See Figure \ref{fig:airy} for examples of Corollaries \ref{C:airy-sheet} and \ref{C:airy-processes}.
The proof of Corollary \ref{C:airy-processes} (via taking a limit from geometric/exponential last passage percolation) also yields new results about certain last passage processes converging to the Airy process. These results extend to new convergence results for the entire Airy line ensemble. See Theorem \ref{T:airy-symm} and surrounding discussion for details.

\subsection{Related work}

Shortly after the first version of this paper was posted, Galashin \cite{galashin2020symmetries} found a new flip symmetry of the coloured stochastic six vertex model. This flip symmetry leads to a large number of other symmetries for the stochastic six vertex model and its degenerations. In the context of models that are degenerations or limits of both the stochastic six vertex model and the log-gamma polymer (e.g. the KPZ equation), this flip symmetry can be used to give alternate proofs of some of our results. For example, the aformentioned Conjecture 1.6 in \cite{borodin2019shift} can be deduced from both Theorem 1.9 in \cite{galashin2020symmetries} and our Theorem \ref{T:conj-bgw}. Galashin's flip symmetry can be viewed as a special, but fundamental, case of Theorem \ref{T:conj-bgw}. 

Interestingly, our proofs and the proofs in \cite{galashin2020symmetries} are quite different. Galashin's methods are mostly algebraic, relying on a new connection between the stochastic six vertex model and the Hecke algebra of the symmetric group.

\subsection{Organization of the paper}

In Section \ref{S:prelim}, we give all the necessary background on the main models that fall into framework of this paper. In that section we isolate the specific model properties coming from the RSK and geometric RSK correspondences that allow for hidden invariance. Note that while these correspondences play a prominent background role in the paper, we never need to describe them explicitly. 

Section \ref{S:framework} proves basic conditional independence statements that will allow us to implement the proof strategy outlined at the start of the introduction. In Section \ref{S:main}, we use the building blocks of Section \ref{S:framework} to prove our main theorems.
Section \ref{SS:consequences} contains proofs of the consequences in Section \ref{SS:symmetries}, \ref{SS:combinatorial}, \ref{SS:directed-limit}, and \ref{SS:limits}. Section \ref{S:conclude} contains concluding remarks and a short appendix gives examples of situations where invariance fails.

\section{Decoupled polymer models}
\label{S:prelim}

In this section, we introduce notation and recall the basic properties of geometric and exponential last passage percolation and the log-gamma polymer that we will need moving forward.

\subsection{Notation and Definitions}
For a collection of points $\bu = (u_1, \dots, u_k) \in \Zd$, let 
$
\scrD(\bu)
$
be the set of all $k$-tuples $\pi = (\pi_1, \dots, \pi_k)$ of disjoint paths, where each $\pi_i$ is a $u_i$-path. For $U \sset \Zd$, let 
$$
\scrE(U) = \lf\{\bu \in \bigcup_{k \in \N} U^k : \scrD(\bu) \ne \emptyset \rg\},
$$
and let $\scrE := \scrE(\Zd)$. We refer to a point $\bu \in \scrE$ as an endpoint.
The following definition unifies the notions of last passage percolation and directed polymers.

\begin{definition}
	\label{D:polymer}
	Let $(R, \oplus, \otimes)$ be either the algebra of addition and multiplication $(+, \X)$ over the positive real numbers $R = \R_{> 0}$, or else the max-plus algebra $(\max, +)$ over $R = \R$.
	Let $M = \{M_x : x \in \Z^2 \}$ be a collection of (possibly random) vertex weights in $R$. For a $k$-tuple of disjoint paths $\pi$, define its \textbf{weight}
	$$
	M(\pi) = \bigotimes_{x \in \pi} M_x.
	$$
	Define the \textbf{partition function} $Z_M:\scrE \to R$ by
	$$
	Z_M(\bu) = \bigoplus_{\pi \in \scrD(\bu)} M(\pi).
	$$
	We refer to $M$ as an \textbf{environment}, and $Z_M$ as the \textbf{polymer model} defined by $M$ (the choice of the algebra $(R, \oplus, \otimes)$ is implicit here). 
\end{definition}

The general proof framework in the bulk of the paper does not require any distinction between the two algebras. Regardless of which algebra we are working in, we will use $1$ and $0$ to denote the multiplicative $(\otimes)$ and additive $(\oplus)$, identities. 

Certain collections of partition functions will be especially relevant to us. As in the introduction, for $u \in \Zd$, define
\begin{align}
\label{E:uk}
u^k &= \big((u^-; u^+ - (k - 1, 0)), (u^- + (1, 0); u^+ - (k - 2, 0)), \dots, (u^- + (k-1, 0); u^+) \big).
\end{align}
The partition function $Z_M(u^k)$ should be thought of as the partition function for $k$ disjoint $u$-paths. It is defined whenever $k \le \min(u^+_1 - u^-_1 + 1, u^+_2 - u^-_2 + 1)$. 
We need to shift the start and end points of the paths in order to prevent overlap. Note that shifting the start points vertically instead of horizontally will result in the same partition function. In other words, 
\begin{equation}
\label{E:d-hat-d}
Z_M(u^k) = Z_M(\hat u^k),
\end{equation}
where
\begin{align*}
\hat u ^k = \big((u^-; u^+ - (0, k - 1)), (u^- + (0, 1);  u^+ - (0, k - 2)), \dots, (u^- + (0, k - 1), u^+) \big).
\end{align*} 
The correspondence \eqref{E:d-hat-d} extends as follows. For $\bu = (u_1^{k_1}, \dots, u_\ell^{k_\ell}) \in \scrE$ and any vector $\bv = (v_1, \dots, v_\ell) \in \scrE$ where $v_i \in \{u_i^{k_i}, \hat u_i^{k_i}\}$ for all $i$, we have
\begin{equation}
\label{E:d-hat-d'}
Z_M(\bu) = Z_M(\bv).
\end{equation}
We introduce special notation for certain sets of endpoints. See Figure \ref{fig:DVV} and the text below for explanations of these definitions.
\begin{itemize}[nosep]
	\item $D(u) := \{u^k : k \in \{1, \dots, \min(u_{2, 1} - u_{1, 1} + 1, u_{2, 2} - u_{1, 2} + 1)\}\}$,
	\item $\hat D(u) := \{\hat u^k : k \in \{1, \dots, \min(u_{2, 1} - u_{1, 1} + 1, u_{2, 2} - u_{1, 2} + 1)\}\}$
	\item $H(u) := \bigcup \{ D(u^-, u^+ - (0, i)) : i \in \{0, \dots, u^+_2 - u^-_2\}\}$.
	\item $V(u) := \bigcup \{ D(u^-, u^+ - (i, 0)) : i \in \{0, \dots, u^+_1 - u^-_1\}\}.$
	\item $\bar H(u) := \scrE \{v \in \Zd : v \sset u, v_1^- = u^-_1, v_1^+ = u^1_+\}.$
	\item $\bar V(u) := \scrE \{v \in \Zd : v \sset u,  v_2^- = u^-_2, v_2^+ = u^2_+\}.$
\end{itemize}

\begin{figure}[tb]
	\centering
	\includegraphics[width=0.4\textwidth]{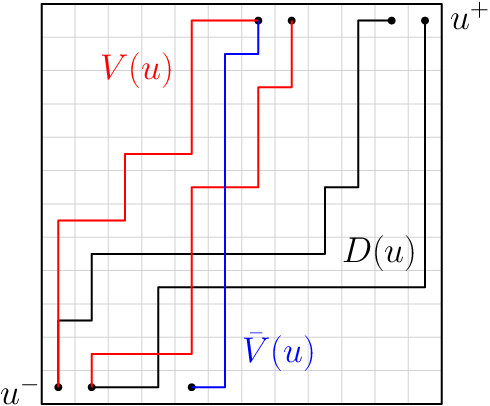}
	\caption{An illustration of the kind of paths allowed in the definitions of $D(u), V(u)$, and $\bar V(u)$. Paths corresponding to endpoints in $D(u)$ move between clusters of points in two corners of $u$, paths in $V(u)$ move from a cluster of points at $u^-$ to a cluster of points on the top edge of $u$, and paths in $\bar V(u)$ move vertically across $u$.}
	\label{fig:DVV}
\end{figure}

The sets $D(u)$ and $\hat D(u)$ should be thought of as the set of endpoints whose paths cross $u$ diagonally. Even though $Z_M|_{D(u)}$ and $Z_M|_{\hat D(u)}$ carry the same information, we keep both notations to help clarify when \eqref{E:d-hat-d} is used. The sets $H(u)$ and $V(u)$ should be thought of as endpoints whose paths cross $u$ horizontally and vertically, respectively, starting at the point $u^-$. The functions $Z_M|_{H(u)}$ and $Z_M|_{V(u)}$ can be interpreted as the two output tableaux coming from the RSK/geometric RSK correspondences applied to the environment $M$ on the box $u$. The sets $\bar H(u)$ and $\bar V(u)$ comprise all endpoints contained in the box $u$ whose paths move horizontally and vertically across $u$. 

Moving forward, we will also use the following notation for box crossings and disjointness:
\begin{equation*}
\begin{split}
\scrH &= \{(u, v) \in \Zd \X \Zd: (u, v) \text{ crosses horizontally}\}, \\
\quad \scrV &= \{(u, v) \in \Zd \X \Zd: (u, v) \text{ crosses vertically}\}, \\ \quad \mathand \quad
\scrN &= \{(u, v) \in \Zd \X \Zd : u \cap v = \emptyset \}.
\end{split}
\end{equation*}

\subsection{Decoupled models}
We now state the main theorem we will use regarding integrable last passage and polymers models.

\begin{theorem}
	\label{T:explicit}
	Let $M, Z_M$ be one of the following three environments (with associated polymer model):
	\begin{enumerate}[label=(\roman*)]
		\item Fix biinfinite sequences $\al, \be$ of positive numbers with $\al_i\be_j < 1$ for all $i, j$. Let $M$ be an array of independent geometric random variables where
		$
		\p(M_{(i, j)} = n) = (1 - \al_i \be_j)(\al_i \be_j)^n,
		$
		for $n =0, 1, \dots,$ and define $Z_M$ with the $(\max, +)$-algebra.
		\item Fix biinfinite sequences $\al, \be$ of real numbers with $\al_i + \be_j > 0$ for all $i, j$. Let $M$ be an array of independent exponential random variables with $\expt M_{(i, j)} = (\al_i + \be_j)^{-1}$,
		and define $Z_M$ with the $(\max, +)$-algebra.
		\item Fix two biinfinite sequences $\al, \be$ of real numbers with $\al_i + \be_j > 0$ for all $i, j$. Let $M$ be an array of independent inverse-gamma random variables with parameter $\al_i + \be_j$.
	\end{enumerate}
	Then for every box $u \in \Zd$, the vectors $Z_M|_{H(u)}$ and $Z_M|_{V(u)}$ are conditionally independent given $Z_M|_{D(u)}$.
\end{theorem}

Theorem \ref{T:explicit} is well-known. It can be interpreting as saying that the three models have a particular `decoupling' property.
 In the first two cases, it follows from the connection between these models and the RSK bijection. The key fact coming from RSK is that it is a bijection from matrices onto all pairs of Young tableaux with a shared shape; the fact that the tableaux only need to have a shared shape yields the decoupling property. The result in these cases was first observed by Johansson (see \citep{johansson2003discrete} and \citep{johansson2000shape}). See also Section 3 in O'Connell \cite{o2003conditioned} for a derivation of these properties that more closely uses the language of our paper. In the third case, the geometric RSK correspondence plays the role that RSK plays in the first two cases. Here the result was shown by Corwin, O'Connell, Sepp\"al\"ainen, and Zygouras (see Corollary 3.3. in \citep{o2014geometric} and also equations (1.2) and (1.3) in \citep{corwin2014tropical}). 

In the context of Theorem \ref{T:explicit}, if we permute elements of the sequences $\al, \be$, then certain last passage statistics are known to be preserved. This is summarized by the following theorem. For biinfinite sequences $\al, \be$ and $u \in \Zd$, define the vectors
$$
\al_u = (\al_i : (i, j) \in u \text{ for some } j) \quad \mathand \quad \hat \be_u = (\be_j : (i, j) \in u \text{ for some } i).
$$
\begin{theorem}
	\label{T:decoupled-family}
	Let $M_{\al, \be}, Z_{M_{\al, \be}}, M_{\al', \be'}, Z_{M'_{\al', \be'}}$ be two polymer models of the same type (i), (ii), or (iii) in Theorem \ref{T:explicit}, indexed by different sequences $(\al, \be)$ and $(\al', \be')$.  Then for any box $u$, we have 
	$$
	Z_{M_{\al, \be}}|_{H(u)} \eqd Z_{M_{\al', \be'}}|_{H(u)}
	$$
	whenever $\hat \be_u = \hat \be_u'$ and $\al_u$ is a permutation of $\al_u'$.
	That is, the joint distribution of all last passage statistics from $u^-$ to the \textbf{right} side of $u$ remains unchanged under permutations of the column parameters $\al_i$. Similarly, we have 
	$$
	Z_{M_{\al, \be}}|_{V(u)} \eqd Z_{M_{\al', \be'}}|_{V(u)}
	$$
	whenever $\al_u = \al_u'$ and $\hat \be_u$ is a permutation of $\hat \be_u'$.
	That is, the joint distribution of all last passage statistics from $u^-$ to the \textbf{top} side of $u$ remains unchanged under permutations of the row parameters $\be_i$.
\end{theorem}

Theorem \ref{T:decoupled-family} follows from symmetries of the RSK and geometric RSK correspondences. Again, it is well known and was observed in the first two cases in \citep{johansson2003discrete} and in the third case in \citep{corwin2014tropical}. In the case of geometric last passage percolation, this follows from \eqref{E:Schur} above and the fact that Schur functions are symmetric in their variables. The exponential last passage percolation statement follows by taking the standard geometric-to-exponential limit. For the case of the log-gamma polymer, this follows from the formula (1.2) in \cite{corwin2014tropical} and the fact that the $GL(N, \R)$-Whittaker functions $\Phi_\theta$ used in that formula are symmetric in permutations of the index vector $\theta$ (e.g. see Section 2 of \cite{o2014geometric} for the definition of these functions and their basic properties). 

 In both cases, the fact that underlies Theorem \ref{T:decoupled-family} is that the law of $Z_{M_{\al, \be}}|_{H(u)}$ is given by a product of functions $s_\cdot(\al) s_\cdot(\be)$ whose values depend only on the order of the vectors $\al$ and $\be$. In the case of geometric last passage percolation, these are normalized Schur symmetric functions, e.g. see equation (1.1) in \cite{corwin2014tropical}, and in the case of the log-gamma polymer these are essentially Whittaker functions, see equation (1.2) in \cite{corwin2014tropical}.

\section{Stronger decoupling statements}
\label{S:framework}
In this section, we prove extensions of the conditional independence property of Theorem \ref{T:explicit}. These stronger conditional independence results will allow us to implement the proof strategy outlined at the beginning of Section \ref{S:intro}. We start with two deterministic facts about polymer models.

\begin{prop}
	\label{P:poly-comp}
	Let $\bu = (u_1, \dots, u_k) \in \scrE$, and suppose that for some $v \in \Zd$, that $(u_i, v) \in \scrH \cup \scrN$ for all $i$. Then $Z_M(\bu)$ is a function of the random variables $\{M_x : x \in (u_1 \cup \dots \cup u_k) \smin v \}$ and the function $Z_M|_{\bar H(v)}$. Similarly, if $(u_i, v) \in \scrV \cup \scrN$ for all $i$, then $Z_M(\bu)$ is a function of $\{M_x : x \in (u_1 \cup \dots \cup u_k) \smin v \}$ and $Z_M|_{\bar V(v)}$.
\end{prop}

The functions in Proposition \ref{P:poly-comp} are explicit, see \eqref{E:poly-comp} below.

\begin{proof}
	We only handle the case when $(u_i, v) \in \scrH \cup \scrN$, as the other case is symmetric. For any collection of disjoint $u_i$-paths $\pi_1, \dots, \pi_k$, let $\pi_v = \lf(\bigcup_i \pi_i \rg) \cap v$ and $\pi_v^c = \lf(\bigcup_i \pi_i \rg) \smin v$. Since $(u_i, v) \in \scrH \cup \scrN$ for all $i$, there exists a unique endpoint $(w_1, \dots, w_\ell)$ depending only on $\pi_v^c$ with $w_i \sset v, (w_i, v) \in \scrH$ such that $\pi_v$ is a union of disjoint $w_i$-paths. Call this endpoint $\bw(\pi_v^c)$. Any $\tau \in \scrD(\bw(\pi^c_v))$ can be concatenated with $\pi_v^c$ to give a $k$-tuple $\sig \in \scrD(\bu)$. Putting this together, we get that
	\begin{equation}
	\label{E:poly-comp}
	Z_M(\bu) = \bigoplus_{\pi^v_c} M(\pi^v_c) \otimes Z_M(\bw(\pi_v^c)) , \qquad \text{ where } \quad M(\pi^v_c) = \bigotimes_{x \in \pi^v_c} M(x).
	\end{equation}
	Here the product is over all possible choices of $\pi^v_c$. Finally, $M(\pi^v_c)$ is always a function of $\{M_x : x \in (u_1 \cup \dots \cup u_k) \smin v \}$ and $Z_M(\bw(\pi_v^c))$ is a function of $Z_M|_{\bar H(v)}$.
\end{proof}

The next fact explains the suggestive notation $H(u), \bar H(u), V(u), \bar V(u)$ introduced in Section \ref{S:prelim}.

\begin{prop}
	\label{P:measurable-fn} Let $M, Z_M$ be any environment and associated polymer model.
Then for any box $u \in \Zd$, we can express $Z_M|_{\bar H(u)}$ as a function of 
$Z_M|_{H(u)}$ and $Z_M|_{\bar V(u)}$ as a function of 
$Z_M|_{V(u)}$.
\end{prop}

Proposition \ref{P:measurable-fn} follows from a stronger statement which constructs $Z_M|_{\bar H(u)}$ and $Z_M|_{\bar V(u)}$ explicitly from 
$Z_M|_{H(u)}$ and $Z_M|_{V(u)}$. The explicit construction is not necessary for the remainder of the paper.

For $u = (u^-, u^+) \in \Zd$, define the set
$$
T_u = \{(i, j) \in u : i + j \ge u_1^- + u_2^+ \}.
$$
Let $P:[u^-_1, u^+_1] \X \{u^-_2\} \to T_u$ be the map taking a point $(i, j)$ to the point $(i, j') \in T_u$ that minimizes the distances $|j - j'|$. We can think of $P$ as projecting the bottom boundary of $u$ onto the bottom boundary of $T_u$. We will also use the notation $P(\bv)$ for the map induced by $P$ on a point 
$
\bv \in \bar V(u).
$
Finally, for $S \sset \Z^2$, let $F_{R}(S)$ be the set of functions  $M:S \to R$ (in other words, the set of $R$-valued environments defined on $S$).

\begin{theorem}
	\label{T:encoding-lpp}
	For every $u \in \Zd$, there is a unique map $\Phi:F_{R}(u) \to F_{R}(T_u)$ satisfying
	\begin{equation}
	\label{E:ZM}
	Z_M(\bv) = Z_{\Phi(M)}(P(\bv))
	\end{equation}
	for all $\bv \in \bar V(u)$. This map is given by
	$$
	\Phi(M)(i, u_2^+ - j) = \begin{cases}
	\frac{Z_M(u_i^{j+1})}{Z_M(u_{i}^{j})}, \qquad &j + u^-_1 = i \\
	\frac{Z_M(u_i^{j+1})  Z_M(u_{i-1}^{j})}{Z_M(u_{i-1}^{j+1}) Z_M(u_{i}^{j})}, \qquad &j + u^-_1 < i 
	\end{cases}
	$$
	where $u_i := (u^-;(i, u_2^+))$ and $u_i^j$ is the $j$-tuple defined in \eqref{E:uk}. We use the convention that $Z_M(u_i^0) = 1$.
\end{theorem}

In the above theorem statement, we use the standard shorthand notation $ab = a \otimes b$ and $a/b = a \otimes b^{-1}$, where $b^{-1}$ is the $\otimes$-inverse of $b$. The environment $\Phi(M)$ is essentially the recording tableau in either the RSK or geometric RSK correspondence. See Figure \ref{fig:gRSK} for an example of this map. See also equation \eqref{E:u-string} below for an explicit description of $u^j_i$ when $u = (1,1; n, m)$.

As discussed in the introduction, a continuous-time version of Theorem \ref{T:encoding-lpp} for last passage percolation was proven in \citep{DOV}. This proof was adapted to the polymer setting by Corwin \citep{corwin2020invariance}. The paper \cite{corwin2020invariance} also contains an alternate proof of Theorem \ref{T:encoding-lpp} due to Matveev, as well as a third proof that is obtained by combining Theorem 1.7 and the results of Section 2.3 from Noumi and Yamada \citep{noumi2002tropical}. We provide yet another proof of Theorem \ref{T:encoding-lpp} here based on combining different ideas from \cite{noumi2002tropical}; namely those that go into the proof of Theorem 1.7 and are later used in their Section 3.2.

\begin{figure}[tb]
	\centering
	\includegraphics[width=0.5\textwidth]{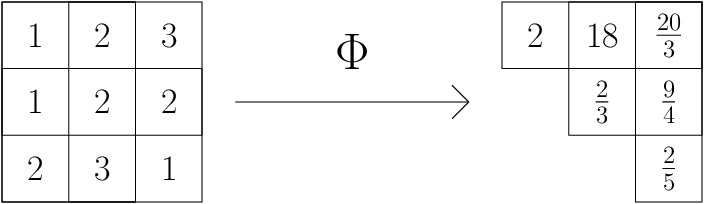}
	\caption{An example of the map $\Phi$ from Theorem \ref{T:encoding-lpp} on a $3 \X 3$ box for the $(+, \X)$-algebra.}
	\label{fig:gRSK}
\end{figure}

\begin{proof}
	To simplify notation in the proof, we assume that $u = (1, 1; n, m)$ for some $n, m \in \N$.
	First, we establish that the function $\Phi$ in the theorem is the unique function that could satisfy $Z_M(\bv) = Z_{\Phi(M)}(P(\bv))$ for all $\bv \in \bar V(u)$. The key observation here is that for any $j$-tuple of the form
	\begin{equation}
	\label{E:u-string}
	u_i^j  = \big((1, 1; i - j + 1, m), (2, 1; i - j + 2, m), \dots, (j, 1; i, m)\big),
	\end{equation}
	where $j \le \min(i, m)$, there is exactly one $j$-tuple of paths in $\scrD(P(u_i^j ))$. Therefore we have a tractable set of $\# T_u$ equations for $\Phi(M)$ given by $Z_M(u_i^j ) = Z_{\Phi(M)}(P(u_i^j ))$. Noting that the equation $Z_M(u_i^j ) = Z_{\Phi(M)}(P(u_i^j ))$ only involves entries $\Phi(M)_x$ with $1 \le x_1 \le i$ and $m- j +1 \le x_2 \le m$, it is straightforward to see that these equations determine $\Phi(M)$.
	
	We now first complete the proof in the case where $(R, \otimes, \oplus)$ is the algebra of multiplication and addition over $\R_{> 0}$. Construct two $n \X n$ matrices $L$ and $\tilde L$ where
	$$
	L_{j, i} = Z_M(j, 1; i, m), \qquad \tilde L_{j, i} = Z_{\Phi(M)}(P(j, 1); i, m)
	$$ 
	where $j \le i$, and $L_{j, i} = \tilde L_{j, i} = 0$ otherwise. By the Lindstr\"om-Gessel-Viennot lemma (\citep{lindstrom1973vector, gessel1985binomial}, see also \citep{karlin1959coincidence}), if $\bu = ((u_i, 1; v_i, m))_{i = 1, \dots, k}$, then 
	$$
	Z_M(\bu) = \det L^{u_1, \dots, u_k}_{v_1, \dots, v_k} , \qquad \mathand \qquad Z_{\Phi(M)}(P(\bu)) = \det \tilde L^{u_1, \dots, u_k}_{v_1, \dots, v_k},
	$$
	where $L^J_I$ denotes the matrix minor coming from rows with indices in $J$ and columns with indices in $I$. This allows us to recognize two things. First, to prove Theorem \ref{T:encoding-lpp}, we just need to show that $L = \tilde L$. Second, for any $i \in [1, n], j \in [1, m \wedge i]$, we have that
	\begin{equation}
		\label{E:dets}
		\det L^{1, \dots, j}_{i - j +1, \dots, i} = Z_M(u_i^j) = Z_{\Phi(M)}(P(u_i^j )) = \det \tilde L^{1, \dots, j}_{i - j +1, \dots, i}.
	\end{equation}
	Finally, the minor determinants in \eqref{E:dets} determine the matrix $L$ (and $\tilde L$). We verify this by inducting on the entries $L_{j, i}$ with row index $j$ and column index $i$, using the partial order $(j, i) \preceq (j', i')$ whenever $j \le j', i \le i'$. 
	
Observe that by setting $j=1$ in \eqref{E:dets} we have that $L_{1, i} = \tilde L_{1, i}$. Also, $L_{j, i} = \tilde L_{j, i} = 0$ for all $j > i$. These two points together can be seen as the base case.

Now consider an arbitrary entry $L_{j, i}$ for some $j \le i$ and suppose that $L_{j', i'} = \tilde L_{j',i'}$ for all $(j', i') \prec (j, i)$. First, if $j \le m + 1$ then by cofactor expansion we have
\begin{equation}
	\label{E:L-eqn}
	\det L^{1, \dots, j}_{i - j +1, \dots, i} = L_{j, i} Z_M(u_{i-1}^{j-1}) + f,
\end{equation}
where the term $f$ only depends on $\{L_{j', i'} : (j', i') \prec (j, i)\}$. Since the entries of $M$ were all positive, $Z_M(u_{i-1}^{j-1}) \ne 0$, so we can rearrange \eqref{E:L-eqn} to get a formula for $L_{j, i}$ in terms of $\{L_{j', i'} : (j', i') \prec (j, i)\}$. As the same formula applies for $\tilde L$, we get that $L_{j, i} = \tilde L_{j, i}$.

Now suppose $i \ge j > m + 1$. We claim that 
$$
\det L^{1, \dots, j}_{1, \dots, j - m, i - m + 1, \dots, i} = \det L^{1, \dots, m}_{i - m + 1, \dots, i}.
$$
Indeed, by \eqref{E:dets} it suffices to show that the left-hand side above equals $Z_M(u^m_i)$. To see this, observe that there is a single $m$-tuple of paths in $\mathcal D(u^m_i)$, which covers all vertices in the box $(1, 1; i, m)$. Therefore 
$$
Z_M(u^m_i) = \prod_{x \in (1, 1; i, m)} M_x.
$$
On the other hand, there is also just a single $j$-tuple of disjoint paths for
$$
\bu := ((1, 1; 1, m), \dots, (j-m, 1; j- m, m), (j- m + 1, 1; i - m + 1, m), \dots, (j, 1; i, m)),
$$
which also covers all vertices in the box $(1, 1; i, m)$. Therefore $Z_M(u^m_i) = Z_M(\bu)$, which equals $\det L^{1, \dots, j}_{1, \dots, j - m, i - m + 1, \dots, i}$ by the Lindstr\"om-Gessel-Viennot lemma. This gives the desired equality. From here on the proof is similar to the case when $j \le m + 1$. By cofactor expansion we can write
$$
\det L^{1, \dots, m}_{i - m + 1, \dots, i} = \det L^{1, \dots, j}_{1, \dots, j - m, i - m + 1, \dots, i} = L_{j,i} \det L^{1, \dots, j-1}_{1, \dots, j - m, i - m + 1, \dots, i-1} + f,
$$
where the term $f$ only depends on $\{L_{j', i'} : (j', i') \prec (j, i)\}$. The Lindstr\"om-Gessel-Viennot lemma and the positivity of all entries in $M$ implies that the determinant $\det L^{1, \dots, j-1}_{1, \dots, j - m, i - m + 1, \dots, i-1}$ is non-zero, allowing us to rearrange the equality above to write $L_{j, i}$ in terms of $f$, $\det L^{1, \dots, j-1}_{1, \dots, j - m, i - m + 1, \dots, i-1}$, and $\det L^{1, \dots, m}_{i - m + 1, \dots, i}$. By identical reasoning, we have the same expression with $L$ replaced by $\tilde L$ everywhere. Finally, by the inductive hypothesis and \eqref{E:dets}, $f$, $\det L^{1, \dots, j-1}_{1, \dots, j - m, i - m + 1, \dots, i-1}$, and $\det L^{1, \dots, m}_{i - m + 1, \dots, i}$ remain unchanged if we replace $L$ with $\tilde L$. Therefore $L_{j, i} = \tilde L_{j, i}$, as desired.

	If $(R, \otimes, \oplus)$ is the max-plus algebra, which lacks an inverse for the operation $\max$, then the Lindstrom-Gessel-Viennot lemma no longer applies. However, since the
	definitions of $\Phi$ and $Z_M$ only involve the operations $\otimes, \oplus,$ and $\otimes^{-1}$, the theorem is really just claiming an equality of two functions $f_1$ and $f_2$ built from those operations with finitely many variables in $R$. Such an equality is true in the $(\max, +)$-algebra if it is true in the $(+, \X)$-algebra (see Proposition 1.9 in \citep{noumi2002tropical}).
\end{proof}

Proposition \ref{P:measurable-fn} falls out as an immediate corollary of Theorem \ref{T:encoding-lpp} for the vertical endpoint set $\bar V(u)$, by noting that the map $\Phi$ in Theorem \ref{T:encoding-lpp} only depends on $Z_M|_{V(u)}$. The claim for the horizontal set $\bar H(u)$ follows by symmetry.

By using the observations of Propositions \ref{P:poly-comp} and \ref{P:measurable-fn}, we can extend the decoupling property from Theorem \ref{T:explicit} to more general partition functions. To state these extensions, we introduce the following notation. 
Let
$$
\sqcup S := \bigcup_{s \in S} s
$$ denote the union (as a subset of $\Z^2$) of all boxes in $S$. We use the $\sqcup$ notation for the middle union above to distinguish the operation from the usual union of sets. Finally, for any function $F:\Zd \to \Zd$ (e.g. $D, \hat D, H, V, \bar H, \bar V)$, let
$$
F(S) = \bigcup_{s \in S} F(s).
$$

\begin{definition}
	\label{D:crossing-triple}
	Let $E, F, G \sset \Zd$. Then $(E, F, G)$ is a \textbf{Markov triple} if there are partitions $E = E_1 \cup E_2, F = F_1 \cup F_2, G = G_1 \cup G_2$ such that the following conditions hold:
	\begin{enumerate}[nosep, label=(\roman*)]
		\item $\sqcup E \cap \sqcup G \sset \sqcup F$.
		\item For $i \ne j$, $E_i \X F_j, F_i \X G_j, G_j \X E_i \sset \scrN$ and $f_1 \cap f_2 = \emptyset$ for any distinct boxes $f_1, f_2 \in \scrF$.
		\item $E_1 \X F_1$ and $F_1 \X G_1$ are contained in $\scrH \cup \scrN$. Similarly, $E_2 \X G_2 \cup G_2 \X F_2 \sset \scrV \cup \scrN$.
	\end{enumerate}
\end{definition}
Finally, for $U \sset \Zd$, define the extended endpoint set
\begin{equation*}
\scrE_D(U) = \{\bu \in \scrE: \bu = (u_1^{k_1}, \dots, u_m^{k_m}) \in \scrE, \text{ where } u_i \in U, k_i, m \in \N\}.
\end{equation*}
The set $\scrE_D(U)$ contains all endpoints consisting of tuples of points in the diagonal set $D(U)$.
\begin{prop}
	\label{P:1-step-markov}
	Let $(E, F, G)$ be a Markov triple and let $M, Z_M$ be as in (i), (ii), or (iii) of Theorem \ref{T:explicit}, i.e. $Z_M$ is either last passage percolation with geometric or exponential weights or the log-gamma polymer.
	 
	Then $Z_M|_{\scrE_D(E)}$ and $Z_M|_{\scrE_D(G)}$ are conditionally independent given $Z_M|_{\scrE_D(F)}$. In other words, $(Z_M|_{\scrE_D(E)}, Z_M|_{\scrE_D(F)}, Z_M|_{\scrE_D(G)})$ is a Markov chain.
\end{prop}

Note that since any pair of distinct boxes in $F$ are disjoint, that $Z_M|_{\scrE_D(F)}$ and $Z_M|_{D(F)}$ are functions of each other, so we can equivalently show $(Z_M|_{\scrE_D(E)}, Z_M|_{D(F)}, Z_M|_{\scrE_D(G)})$ is also a Markov chain.

\begin{proof} 
	Define
	$$
	E_{1, H} = \bigcup_{e \in E_1} \{u \sset e : (u, e) \in \scrH \}, \quad \mathand \quad  E_{2, V} = \bigcup_{e \in E_2} \{u \sset e : (u, e) \in \scrV \}.
	$$
	Observe that whenever 
	$$
	\bu =  (u_1^{k_1}, \dots, u_m^{k_m}) \in \scrE_D(E_1 \cup E_2),
	$$
	then for every $i$, either $u^{k_i}_i \in \scrE(E_{2, V})$ or $\hat u^{k_i}_i \in \scrE(E_{1, H})$. Therefore by the correspondence \eqref{E:d-hat-d}, $Z_M|_{\scrE_D(E)}$ is a function of $Z_M|_{\scrE(E_{1, H} \cup E_{2, V})}$.
	Moreover, by Definition \ref{D:crossing-triple}(ii, iii) and the construction of $E_{1, H}, E_{2, V}$, we have 
	$$
	(E_{1, H} \cup E_{2, V}) \X F_1 \sset \scrH \cup \scrN, \quad  (E_{1, H} \cup E_{2, V}) \X F_2 \sset \scrV \cup \scrN, \qquad \sqcup (E_{1, H} \cup E_{2, V}) = \sqcup E.
	$$
	Therefore by Proposition \ref{P:poly-comp}, $Z_M|_{\scrE(E_{1, H} \cup E_{2, V})}$ is a measurable function of 
	$$
	\{M_x : x \in \sqcup E \smin \sqcup F \}, \quad Z_M|_{\bar H(F_1)}, \quad \mathand \qquad Z_M|_{\bar V(F_2)}.
	$$
	Hence so is $Z_M|_{\scrE_D(E)}$. By Proposition \ref{P:measurable-fn}, this implies that $Z_M|_{\scrE_D(E)}$ is measurable function of  $\{M_x : x \in \sqcup E \smin \sqcup F \}, Z_M|_{H(F_1)}$, and $Z_M|_{V(F_2)}$.
	Similarly, $Z_M|_{\scrE_D(G)}$ is measurable function of  
	$\{M_x : x \in \sqcup G \smin \sqcup F \}, Z_M|_{V(F_1)}$, and $Z_M|_{H(F_2)}$. Therefore since $M$ has independent entries, and $\sqcup E \cap \sqcup G \sset \sqcup F$, to prove the proposition we just need to show that 
	$$
	(Z_M|_{H(F_1)},Z_M|_{V(F_2)}) \quad \mathand \quad (Z_M|_{V(F_1)},Z_M|_{H(F_2)}) 
	$$
	are conditionally independent given $Z_M|_{D(F)}$. This follows immediately from the decoupling property of $M$ and the fact that any pair of distinct boxes in $F$ are disjoint.
\end{proof}

By iterating Proposition \ref{P:1-step-markov}, we can get longer Markov chains. We note one such construction here that will be particularly useful to us. 

\begin{corollary}
	\label{C:2-step-Markov}
	Let $(E, F, G)$ and $(E, F', G)$ be two Markov triples that satisfy Definition \ref{D:crossing-triple} for partitions $E = E_1 \cup E_2, F = F_1 \cup F_2, F' = F_1' \cup F_2', G = G_1 \cup G_2$. Suppose that 
	\begin{itemize}[nosep]
		\item $F_1 \X F_1' \sset \scrH \cup \scrN$ and $F_2 \X F_2' \sset \scrV \cup \scrN$,
		and $F_i \X F_j' \sset \scrN$ for $i \ne j$,
		\item $\sqcup F \cap \sqcup G \sset \sqcup F'$ and $\sqcup E \cap \sqcup F' \sset \sqcup F$. 
	\end{itemize}
	Then $(E \cup F, F', G)$ and $(E, F, F' \cup G)$ are Markov triples and  
	\begin{equation}
	\label{E:4-chain}
	(Z_M|_{\scrE_D(E)}, Z_M|_{\scrE_D(F)}, Z_M|_{\scrE_D(F')}, Z_M|_{\scrE_D(G)})
	\end{equation}
	is a Markov chain.
\end{corollary}

\begin{proof}
	It is straightforward to check the conditions of Definition \ref{D:crossing-triple} for the triples $(E \cup F, F', G)$ and $(E, F, F' \cup G)$. The fact that \eqref{E:4-chain} is a Markov chain then follows from two applications of Proposition \ref{P:1-step-markov}.
\end{proof}

We will refer to a quadruple $(E, F, F', G)$ satisfying the conditions of Corollary \ref{C:2-step-Markov} as a \textbf{Markov quadruple}.

\section{Hidden invariance}
\label{S:main}

In this section, we use the decoupling statements of Section \ref{S:framework} to find new invariances of the models in Theorem \ref{T:explicit}.

\subsection{Homogeneous models}
\label{SS:homog}
We start with the homogeneous case. Throughout this subsection we assume that $Z_M$ is either the log-gamma polymer with i.i.d.\ weights, or else geometric or exponential last passage percolation with i.i.d.\ weights
(e.g. the cases when all $\al_i, \be_i$ are equal in Theorem \ref{T:explicit}).

\begin{definition}
	\label{D:scrF}
	Let $\scrF$ be the set of bijections between subsets of $\Zd$ defined by the following four rules: 
	\begin{enumerate}[nosep, label=(\roman*)]
		\item (Contains translations) For any $c \in \Z^2$, the translation $T_c:\Zd \to \Zd \in \scrF$.
		\item (Contains reflections) The reflections $R_1, R_2:\Zd \to \Zd$ are in $\scrF$.
		\item (Closed under restriction, composition, inverses, and finite exhaustion) If $f:S \to T, g:T \to Q$ are elements of $\scrF$, then $g \circ f:S \to Q$ is in $\scrF$. Also, if $f:S \to T \in \scrF$, then $f^{-1}:T \to S \in \scrF$ and $f|_{Q}:Q \to f(Q)$ is in $\scrF$ for any $Q \sset S$. Finally, if $f:S \to T$ is a bijection between infinite subsets of $\Zd$, then $f \in \scrF$ if $f|_{Q} \in \scrF$ for any finite subset $Q \sset S$. 
		\item (Factoring) Suppose that $(U, V, W), (U', V', W')$ are any Markov triples, and that $f:U \cup V \cup W \to U' \cup V' \cup W'$ is a function with $f(U) = U', f(V) = V', f(W) = W'$. Suppose also that $f|_{U \cup V}, f|_{V \cup W} \in \scrF$. Then $f \in \scrF$.
	\end{enumerate} 
\end{definition}

To state our main invariance theorem, we will extend a function $f:S \to T \in \scrF$ to a function $\bar f:\scrE_D(S) \to \scrE_D(T)$ by letting 
$$
f(u_1^{k_1}, \dots, u_\ell^{k_\ell}) = (f(u_1)^{k_1}, \dots, f(u_\ell)^{k_\ell})
$$
for any $u_1, \dots, u_\ell \in S$. As part of the proof of the next theorem, we will check that any such $\bar f$ is a well-defined bijection.

\begin{theorem}
	\label{T:iid-trans}
Let $M$ be an i.i.d.\ model from Theorem \ref{T:explicit}. Let $f:S \to T \in \scrF$, and let $\bar f$ be its extension to a function from $\scrE_D(S)$ to $\scrE_D(T)$ . Then
	$$
	Z_M|_{\scrE_D(S)} \eqd Z_M|_{\scrE_D(T)} \circ \bar f.
	$$
\end{theorem}

To prove Theorem \ref{T:iid-trans}, we first state and prove a lemma about functions in $\scrF$.

\begin{lemma}
	\label{L:N-preservation}
	Any $f \in \scrF$ preserves $\scrN$ and $\scrH \cup \scrV$. That is, $(u, v) \in \scrN$ if and only if $(f(u), f(v)) \in \scrN$, and $(u, v) \in \scrH \cup \scrV$ if and only if $(f(u), f(v)) \in \scrH \cup \scrV$.
\end{lemma}

\begin{proof}
	Let $\scrF' \sset \scrF$ be the set of functions that preserve $\scrN$ and $\scrH \cup \scrV$. It is easy to see that $\scrF'$ contains translations, reflections, and is closed under restriction, composition, inverses, and finite exhaustion. It remains to show that $\scrF'$ is closed under factoring. Suppose that $f: U \cup V \cup W \to U' \cup V' \cup W'$ is such that 
	$$
	f(U) = U', \quad f(V) = V', \quad f(W) = W', \quad f|_{U \cup V}, f|_{V \cup W} \in \scrF'.
	$$
	The map $f$ preserves $\scrN$ and $\scrH \cup \scrV$ for all pairs in $(U \cup V)^2$ or in $(V \cup W)^2$ by the corresponding properties of the restrictions $f|_{U \cup V}$ and $f|_{V \cup W}$. Now suppose that $(u, w) \in U \X W$. We first check that
	\begin{equation}
	\label{E:fNN}
	\text{ if } (f(u), f(w)) \notin \scrN, \text{ then } (u, w) \notin \scrN.
	\end{equation}
	For this, observe that if $(U, V, W)$ is a Markov triple, then:
	\begin{enumerate}[nosep,label=(\Roman*)]
		\item If $(u, w) \in (U \X W) \cap \scrH$, there exists $v \in V$ such that $(u, v), (v, w) \in \scrH$. Similarly, if $(u, w) \in (U \X W) \cap \scrV$, there exists $v \in V$ such that $(u, v), (v, w) \in \scrV$.
		\item If $(u, v, w) \in U \X V \X W$ are such that $(u, v), (v, w)$ cross, then either $(u, v), (v, w)$ and $(u, w)$ are all in $\scrH$ or they are all in $\scrV$.
	\end{enumerate} 
	Now if $(f(u), f(w)) \notin \scrN$ then since $(f(u), f(w)) \in U' \X W'$, by (ii) and (iii) of Definition \ref{D:crossing-triple}, $(f(u), f(w)) \in \scrH \cup \scrV$. Therefore by (I) above there exists $v \in V'$ such that $(f(u), v), (v, f(w)) \in \scrH \cup \scrV$. Using that $f|_{U \cup V},f|_{V \cup W}$ preserve $\scrH \cup \scrV$, we have that
	$$
	(u, f^{-1}(v)), (f^{-1}(v), w) \in \scrH \cup \scrV.
	$$
	Key observation (II) then implies that $(u, w)$ crosses, proving \eqref{E:fNN}. The converse of \eqref{E:fNN} follows from the same argument applied to $f^{-1}$. Finally, since $f$ is a bijection and $U \X W, U' \X W' \sset \scrN \cup \scrH \cup \scrV$, we have that $(u, w) \in \scrH \cup \scrV$ if and only if $(f(u), f(w)) \in \scrH \cup \scrV$ for $(u, w) \in U \X W$. This completes the proof. 
\end{proof}

\begin{proof}[Proof of Theorem \ref{T:iid-trans}]
	Let $\scrF'$ be the set of bijections $f \in \scrF$ for which $\bar f$ is a well-defined bijection. We first check that $\scrF' = \scrF$ by showing that $\scrF'$ is closed under (i-iv) of Definition \ref{D:scrF}. The only nontrivial thing to check is that $\scrF'$ is closed under factoring. Let $f:S \to T$ with $S = U \cup V \cup W$ and $T = U' \cup V' \cup W'$ be such that $f(U) = U', f(V) = V', f(W) = W'$, and $f|_{U \cup V}, f|_{V \cup W} \in \scrF'$. 
	
	Suppose that $(\bu, \bv) \in \scrE_D(S)$ where $\bu = (u_1^{k_1}, \dots, u_\ell^{k_\ell}) \in \scrE_D(U)$ and $\bv = (v_1^{m_1}, \dots, v_n^{m_n}) \in \scrE_D(V \cup W)$. By Definition \ref{D:crossing-triple}, $U \X (V \cup W) \sset \scrN \cup \scrH \cup \scrV$. In particular, this means that the set of disjoint paths $\scrD((\bu, \bv))$ is empty unless $\{u_1, \dots, u_\ell\} \X \{v_1, \dots, v_n\} \sset \scrN.$ By Lemma \ref{L:N-preservation}, this in turn implies that 
	\begin{equation}
	\label{E:fuu}
	\{f(u_1), \dots, f(u_\ell)\} \X \{f(v_1), \dots, f(v_n)\} \sset \scrN.
	\end{equation}
	Moreover, $f(\bu) \in \scrE_D(U')$ and $f(\bv) \in \scrE_D(V' \cup W')$ since $\bar f$ is well-defined on the restrictions of $f$, and hence $(f(\bu), f(\bv)) \in \scrE_D(T)$ by \eqref{E:fuu}. The bijectivity of $\bar f$ then follows from the bijectivity of $f$.
	
	Now let $\scrG$ be the set of bijections $f$ for which the theorem holds. The set $\scrG$ contains translations by the homogeneity of $M$. It contains reflections again by the homogeneity of $M$ along with the general identity \eqref{E:d-hat-d'} between $u^k$ and $\hat u^k$. It is clear that $\scrG$ is closed under restriction, composition, inverses, and finite exhaustion. To check that $\scrG$ is closed under factoring, suppose $f$ is as in Definition \ref{D:scrF}(iv) with restrictions $f|_{U \cup V}, f|_{V \cup W} \in \scrG$. Note that $U \X V, U \X W, V \X W \sset \scrN \cup \scrH \cup \scrV$ as $(U, V, W)$ is a Markov triple. Therefore the restriction $Z_M|_{\scrE_D(U \cup V \cup W)}$ is a function of $Z_M|_{\scrE_D(U)}, Z_M|_{\scrE_D(V)},$ and $Z_M|_{\scrE_D(W)}$ (and similarly for $U', V', W'$). The joint distribution of $Z_M|_{\scrE_D(U)}, Z_M|_{\scrE_D(V)},$ and $Z_M|_{\scrE_D(W)}$ is preserved by $f$ by Proposition \ref{P:1-step-markov}.
\end{proof}

As written, Theorem \ref{T:iid-trans} is rather abstract and doesn't give a great deal of insight into the nature of the functions in $\scrF$. We spend the rest of this subsection exploring what types of functions $\scrF$ contains. We don't attempt to give a complete classification, but rather hope to give a good account of the kinds of functions $\scrF$ does and does not contain. Beyond those in Lemma \ref{L:N-preservation}, we first note a few more restrictions on the types of functions in $\scrF$.

\begin{lemma}
	\label{L:F-restrictions}
	Suppose that $f:S \to T \in \scrF$. Then 
	\begin{enumerate}[nosep, label=(\roman*)]
		\item For all $s \in S$, $f(s)$ can be obtained from $s$ by translation and the reflections $R_1$ and $R_2$.
		\item If $u, v \in S$ and $(u, v) \notin \scrH \cup \scrV \cup \scrN$, then the restriction $f|_{\{u, v\}}$ is a product of reflections and a translation. In other words, the pair $(u, v)$ is congruent to the pair $(f(u), f(v))$.
		\item Let $(s_1, \dots, s_k)$ be a $k$-tuple of distinct points in $S$ with $(s_i, s_{i+1}) \in \scrH \cup \scrV$ for all $i$. Let $(a_1, \dots, a_{k-1}) \in \{\scrH, \scrV\}^{k-1}$ be the sequence where $(s_i, s_{i+1}) \in a_i$ for all $i$, and $(b_1, \dots, b_{k-1})$ be the corresponding sequence defined from $(f(s_1), \dots, f(s_k))$. Then either $a_i = b_i$ for all $i$, or else $a_i \ne b_i$ for all $i$. 
	\end{enumerate}
\end{lemma}

When $|S| \le 2$, it turns out that any function satisfying the conditions of Lemmas \ref{L:N-preservation} and \ref{L:F-restrictions} is in $\scrF$. When $|S| \ge 3$, this is no longer the case. See Example \ref{E:not-in-F} for a map $f:S \to T$ satisfying the conclusions of Lemmas \ref{L:N-preservation} and \ref{L:F-restrictions} with $|S| = |T| = 3$, but yet $Z_M|_{S} \ne Z_M|_{T} \circ f$ for any bijection $f$ when $M$ is geometric or exponential last passage percolation. 

\begin{proof}
	Each of the three properties above holds when $f$ is a translation or a reflection, and each of these properties is closed under restriction, composition, inverses, and finite exhaustion. We just need to check that each of the properties is closed under factoring for functions in $\scrF$. To this end, let $f:U \cup V \cup W \to U' \cup V' \cup W'$ be a function in $\scrF$ constructed as in Definition \ref{D:scrF}(iv), and suppose that the restrictions $f|_{U \cup V}, f|_{V \cup W}$ satisfy the conclusions of the lemma. Now, if 
	$$
	(u, v) \in (U \cup V \cup W)^2 \smin \scrH \cup \scrV \cup \scrN,
	$$
	then since $(U, V, W)$ is a Markov triple, either $(u, v) \in U^2$ or else $(u, v) \in W^2$. In particular, this implies that $f$ inherits property (ii) from $f|_{U \cup V}$ and $f|_{V \cup W}$. The function $f$ also has property (i) since property (ii) implies property (i).
	It is enough to check property (iii) for chains $(s_1, s_2, s_3)$ of length three. 
	
	\textbf{Case 1:} $s_1, s_3 \in V \cup W, s_2 \in U \cup V$. \qquad  If $s_2 \in V$, then the assertion follows from property (iii) of the restriction $f|_{V \cup W}$. If not, then the pair $((s_1, s_2), (s_2, s_3))$ is either in $\scrH \X \scrV$ or $\scrV \X \scrH$; the other cases are ruled out by Definition \ref{D:crossing-triple}(ii),(iii). The same assertion holds for $((f(s_1), f(s_2)), (f(s_1), f(s_2))$ since $f(U) = U', f(V) = V'$ and $f(W) = W'$. Hence the assertion holds in this case. This argument also covers the case when the roles of $W$ and $U$ are switched.
	
	\textbf{Case 2:} $s_1, s_2 \in V \cup W, s_3 \in U$ with $(s_2, s_3) \in \scrH$. \qquad All remaining cases are symmetric versions of this case. By observation (I) from Lemma \ref{L:N-preservation}, we can find $v \in V$ such that $(s_2, v), (v, s_3) \in \scrH$. Now, assertion (iii) holds for the triple $(s_1, s_2, v)$ since all points lie in $V \cup W$. It holds for $(v, s_2, s_3)$ by case 1 (applied to the case where the roles of $W$ and $U$ are switched). Therefore assertion (iii) also holds for $(s_1, s_2, s_3)$.
\end{proof}

We now give examples of the kind of functions that lie in $\scrF$. We build up starting with two nearly trivial examples.

\begin{lemma}
	\label{L:in-F-triv}
	\begin{enumerate}
		\item (Basic reflection) Let $u \in \Zd$ be such that $w = u + (0, k; 0,0), w' = u + (0, 0, 0, -k) \in \Zd$ for some $k \in \N$. Let $f(u) = u$ and $f(w) = w'$. Then $f:\{u, w\} \to \{u, w'\} \in \scrF$.
		\item (Disjoint movements) Let $\{U_i : i \in I\}, \{U_i' : i \in I\}$ be two collections of subsets of $\Zd$ indexed by a (finite or countable) index set $I$. Suppose that $U_i \X U_j, U_i' \X U_j' \sset \scrN$ for any $i \ne j$. Suppose also that for all $i$, there exists a function that $f_i: U_i \to U_i' \in \scrF$. Let $f:\bigcup U_i \to \bigcup U_i'$ be the function whose restriction to $U_i$ is $f_i$. Then $f \in \scrF$.
	\end{enumerate}
\end{lemma}

\begin{proof}
	The map in (1) is a composition of a translation and two reflections. We prove (2) in the case when $|I|$ is finite by a straightforward induction. The $|I|=1$ case is evident. Now assume that the statement holds for collections of $j$ sets whenever $j \le k-1$, and consider $U_1, \dots, U_k, U'_1, \dots, U_k'$ satisfying the assumptions of (2) and the function $f$ above. Then $U = (U_1, \dots, U_{k-1}), V = \emptyset, W = U_k$ and $U' = (U'_1, \dots, U'_{k-1}), V' = \emptyset, W' = U'_k$ are two Markov triples and $f|_U, f|_W \in \scrF$ by the inductive hypothesis. Hence $f \in \scrF$. Extension to infinite index sets $I$ follows since $\scrF$ is closed under finite exhaustion.
\end{proof}

The next lemma already illustrates the kinds of nontrivial functions contained in $\scrF$.

\begin{lemma}
	\label{L:G-moves} (Basic transposition)
	Let $u \in \Zd$, and let $v, w, v', w' \sset u$ with $\{u\} \X \{v, w, v', w'\} \sset \scrH$ and $(v, w), (v', w') \in \scrN$. Suppose also that $v'$ is a translate of $v$ and $w'$ is a translate of $w$. Define $f:\{u, v, w\} \to \{u, v', w'\}$ by $f(u) = u, f(v) = v', f(w) = w'$. Then $f \in \scrF$.
\end{lemma}

Note that Lemma \ref{L:G-moves} also holds with $w$ and $w'$ omitted since $\scrF$ is closed under restrictions. 

\begin{proof}
	Without loss of generality, we may assume that  $v^-_1 < w^-_1$ .
	
	\textbf{Case 1:} $v' = T_{(k, 0)} v$ for some $k \in \Z, w' = w,$ and $v'^-_1 < w'^-_1$. Let $b_v$ be the smallest box containing both $v, v'$. In this case, both
	$$
	(U = \{u\}, V = \{b_v \}, W = \{v, w\}) \quad \mathand \quad (U, V, W' = \{v', w'\})
	$$
	are Markov triples. Now extend the map $f$ so that $f(b_v) = b_v$. Then $f|_{U \cup V} \in \scrF$ since it is simply the identity. Also $f|_{V \cup W} \in \scrF$ as it is of the form of Lemma \ref{L:in-F-triv}(2). To see this, observe that $\{v, v', b_v\} \X \{w \} \sset \scrN$ and that $f|_{\{v, b_v\}}$ satisfies Lemma \ref{L:in-F-triv} (i). Hence $f:U \cup V\cup W \to U \cup V \cup W' \in \scrF$ by factoring, and therefore so is $f|_{U \cup W}$. 
	
	\textbf{Case 2:} ${v'}^-_1 < {w'}^-_1$. If $v' = v$, then $f \in \scrF$ by a symmetric version of Case 1. Any other map with ${v'}^-_1 < {w'}^-_1$ can be obtained by composing functions where $v = v'$ and functions where $w = w'$.
	
	\textbf{Case 3:} ${v'}^-_1 > {w'}^-_1$. In this case, the function $f$ is a composition of a translation and two reflections with a function of the form of Case 2.
\end{proof}

We will show that richer sets of functions lie in $\scrF$ by using Corollary \ref{C:2-step-Markov}. Reinterpreted in terms of the set of functions $\scrF$, that corollary says the following.

\begin{corollary}
	\label{C:2-step-markov'}
	Let $(U, V, V', W)$ and $(\bar U, \bar V, \bar V', \bar W)$ be two Markov quadruples (i.e. they satisfy the assumptions of Corollary \ref{C:2-step-Markov}). Let 
	$$
	f: U \cup V \cup V' \cup W \to \bar U \cup \bar V \cup \bar V' \cup \bar W
	$$
	be a bijection mapping $U \mapsto \bar U, V \mapsto \bar V, V' \mapsto \bar V', W \mapsto \bar W$, and suppose that $f|_{U \cup V}, f|_{V \cup V'}, f|_{V' \cup W} \in \scrF$. Then $f \in \scrF$. 
\end{corollary}

Corollary \ref{C:2-step-markov'} follows immediately from property (iv) of $\scrF$ and Corollary \ref{C:2-step-Markov}. Our strategy for applying it will be as follows. Consider a map $f:U \cup W \to \bar U \cup \bar W$ mapping $U \mapsto \bar U, W \mapsto \bar W$. Suppose that we can find connecting sets $V, V'$ and $\bar V, \bar V'$ such that $(U, V, V', W)$ and $(\bar U, \bar V, \bar V', \bar W)$ are Markov quadruples, and that $f$ can be extended to map from $V$ to $\bar V$ and $V'$ to $\bar V'$. Then $f \in \scrF$ if $f|_{U \cup V}, f|_{V \cup V'}, f|_{V' \cup W} \in \scrF$. We will check that each of these restrictions lie in $\scrF$ via Lemma \ref{L:in-F-triv} and Lemma \ref{L:G-moves}.
To enact this strategy, we first make a few straightforward observations about the structure of crossing boxes.

\begin{lemma}
	\label{L:easy-observations}
	\begin{enumerate}[nosep, label=(\roman*)]
		\item Suppose that $u, w, u', w' \in \Zd$ are such that $(u, w), (u, w'), (u', w) \in \scrH$ and $(u, u'), (w, w') \notin \scrN$. Then $(u', w') \notin \scrN$.
		\item Suppose that $U \X W \sset \scrH$. Let $v$ be the smallest box containing $\sqcup W$. Then $U \X \{v\} \sset \scrH$. Also, if $U \X W \sset \scrN$ and $\sqcup W = v \in \Zd$, then $U \X \{v\} \sset \scrN$.
		\item If $(u, v) \in \scrH$ and $(v, w) \in \scrH$, then $(u, w) \in \scrH$.
	\end{enumerate}
\end{lemma}

All parts of Lemma \ref{L:easy-observations} follow immediately from definition. We now prove a technical lemma that will aid in the construction of the connecting sets $V, V'$. For this lemma and throughout this section, we introduce a graph on $\Zd$ (with induced graphs on its subsets) whereby a pair $u, v \in \Zd$ are connected by an edge whenever $(u, v) \notin \scrN$. For the remainder of the paper, the notion of a connected subset of $\Zd$ is with respect to this graph structure.

\begin{lemma}
	\label{L:shift}
	Let $U, W \sset \Zd$. Suppose that we can partition $U = U_1 \cup U_2$ and $W = W_1 \cup W_2$ such that $U_1 \X W_1, W_2 \X U_2 \sset \scrH \cup \scrN$ and $U_1 \X W_2, U_2 \X W_1 \sset \scrN$. Then there exists a set $V \sset \Zd$ such that $(U, V, W)$ is a Markov triple. The set $V$ can be constructed as follow.
	For every $w \in W$, let 
	$$
	G(w) = \{u \in U : (u, w) \text{ cross}\} \qquad \mathand \qquad  \bar G(u) = \{w \in W : (u, w) \text{ cross}\}
	$$
	For $(u, w) \in U \X W \cap (\scrV \cup \scrH)$, let $G(w)_u$ be the connected component of $G(w)$ containing $u$, and let $\bar G(u)_w$ be the connected component of $\bar G(u)$ containing $w$. Then
	$$
	\sqcup G(w)_u \cap \sqcup \bar G(u)_w
	$$
	is a box for every $(u, w) \in U \X W \cap (\scrV \cup \scrH)$. Letting $v(u, w) \in \Zd$ be this box, we can take
	$$
	V  = \{v(u, w) : (u, w) \in U \X W \cap (\scrV \cup \scrH) \}.
	$$
\end{lemma}

\begin{proof}
	Fix a pair $(u, w) \in U \X W \cap (\scrV \cup \scrH)$. We first show that $\sqcup G(w)_u \cap \sqcup \bar G(u)_w$ is a box. Without loss of generality, we may assume $(u, w) \in \scrH$.
	First observe that $G(w)_u \sset U_1, G(u)_w \sset W_1$. Secondly, by appealing to Lemma \ref{L:easy-observations}(i), for any $u' \in G(w)_u, w' \in \bar G(u)_w $ that are connected by edges to $u, w$, respectively, $(u', w') \notin \scrN$, and so since $u' \in U_1, w' \in W_1$, we must have $(u', w') \in \scrH$. Iterating this along the edges of $G(w)_u, \bar G(u)_w$ gives that $G(w)_u \X \bar G(u)_w \sset \scrH$.
	Now define
	$$
	I_{w, u} = \bigcup_{w' \in \bar G(u)_w} [w'^-_1, w'^+_1] \qquad \mathand \qquad J_{w, u} = \bigcup_{u' \in G(w)_u} [u_2'^-, u_2'^+]
	$$
	The connectivity of $\bar G(u)_w$ and $G(w)_u$ implies that $I_{w, u}$ and $J_{w, u}$ are intervals. Moreover, $u' \cap w' = [w'^-_1, w'^+_1] \X [u_2'^-, u_2'^+]$ for all $u', w' \in G(w)_u \X \bar G(u)_w$ since $G(w)_u \X \bar G(u)_w \sset \scrH$. Therefore $\sqcup \bar G(u)_w \cap \sqcup  G(w)_u = I_{w, u} \X J_{w, u}$, and hence is a box.
	
	Now, with $V$ as in the statement of the lemma, we check that $(U, V, W)$ is a Markov triple. Definition \ref{D:crossing-triple}(i) follows since $u \cap w \sset v(u, w)$ for all $u, w$. Next, we show that 
	\begin{equation}
	\label{E:Gempty}
	\text{ for any boxes $v_1 \ne v_2 \in V$, we have $v_1 \cap v_2 = \emptyset$.}
	\end{equation}
	For this, we just check that $v(u, w) = v(\bar u, \bar w)$ whenever $v(u, w) \cap v(\bar u, \bar w) \ne \emptyset$. First observe that since $G(w)_u \X \bar G(u)_w \sset \scrH$ for all $(u, w) \in \scrH$ and $G(w)_u \X \bar G(u)_w \sset \scrV$ for all $(u, w) \in \scrV$ that 
	\begin{equation}
	\label{E:Gwu}
	G(w)_u = G(w')_{u'} \quad \text{ for any } \quad w' \in G(w)_u, u' \in \bar G(u)_w. 
	\end{equation}
	Next, if $v(u, w) \cap v(u', w') \ne \emptyset$, then for some $w' \in G(w)_u, u' \in \bar G(u)_w, \bar w' \in G(\bar w)_u, \bar u' \in \bar G(\bar u)_w$ we have
	$$
	\bar u' \cap u' \cap w' \cap \bar w' \ne \emptyset.
	$$
	In particular, \eqref{E:Gwu} then implies that
	\begin{equation}
	\label{E:Gwuu}
	G(w)_u = G(w')_{u'} = G(\bar w')_{\bar u'} = G(\bar w)_{\bar u}.
	\end{equation}
	Similarly, $\bar G(u)_w = \bar G(\bar u)_{\bar w}$. Combining this with \eqref{E:Gwuu} proves \eqref{E:Gempty}.
	This argument also implies that we can partition $G = G_1 \cup G_2$, where $g(u, w) \in G_1$ if and only if $(u, w) \in \scrH$ and $g(u, w) \in G_2$ if and only if $(u, w) \in \scrV$ and $u \ne w$. 
	
	This partition, along with $U = U_1 \cup U_2, W = W_1 \cup W_2$, satisfies conditions (ii) and (iii) of Definition \ref{D:crossing-triple}. For example, suppose that $u \in U_1$ and $g(u', w') \in G_1$ and $(u, g(u', w')) \notin \scrN$. Then  $u \cap w \cap g(u', w') \ne \emptyset$ for some $w \in G(u')_{w'}$, so $g(u, w) \cap g(u', w') \ne \emptyset$ and hence $g(u, w) = g(u', w')$. Therefore $(u, g(u, w)) \in \scrH$. Checking the other parts of Definition \ref{D:crossing-triple}(ii), (iii) is similar
\end{proof}

We now turn our attention to proving that more complicated functions lie in $\scrF$. The first proposition is a generalization of Lemma \ref{L:G-moves}.

\begin{prop}
	\label{P:basic-perm} (Column Transposition)
	Let
	$U = U_h \cup U_n, W = W_a \cup W_b$ be subsets of $\Zd$. Suppose that $W_a \X W_b, U_n \X W \sset \scrN,$ and $U_h \X W \sset \scrH$. Suppose also that the sets $W_a$ and $W_b$ are connected, and that there is some box $b \in \Zd$ such that $\sqcup W \sset b$ and $\sqcup U_n \sset b^c$.
	
	Let $f$ be a map with domain $U \cup W$ such that $f|_U = \id, f|_{W_a} = T_{(k, 0)},$ and $f|_{W_b} = T_{(\ell, 0)}$ for some $k, \ell \in \Z$. Suppose that $f$ preserves $\scrH$ and $\scrN$, and that $\sqcup f(W) \sset b$. Then $f \in \scrF$.
\end{prop}

Proposition \ref{P:basic-perm} can be thought of in the following way. The set $U$ consists of `horizontal' boxes $U_h$ and possibly some boxes $U_n$ that do not interact with $W$. The set $W$ consists of `vertical' boxes that cross all boxes in $U_h$. Moreover, it can be partitioned into two disjoint components $W_a$ and $W_b$. Proposition \ref{P:basic-perm} then says that we can translate the components $W_a$ and $W_b$ in any way, as long as we do not break the crossing structure of $W$ and $U_h$, and the fact that $W$ does not interact with $U_n$.

\begin{proof}
	Observe that $U, W$ satisfy the assumptions of Lemma \ref{L:shift} with $U_2, W_2 = \emptyset$. Let $V$ be the set constructed from $U$ and $W$ as in that lemma. We can write $V = V_a \cup V_b$, where 
	$$
	V_a = \{v(u, w) : u \in U_h, w \in W_a\} \qquad \mathand \qquad V_b = \{v(u, w) : u \in U_h, w \in W_b\}.
	$$
	Since $W_a \X W_b \sset \scrN$, $V_a$ and $V_b$ are disjoint. Moreover, by the connectedness of $W_a$ and $W_b$, for each $u \in U_h$ there is exactly one set $v_a(u) \in V_a$ that equals $v(u, w)$ for all $w \in W_a$ and exactly one set $v_b(u) \in V_b$ that equals $v(u, w)$ for all $w \in W_b$.
	
	Similarly define $\bar V = \bar V_a \cup \bar V_b$ from $U, f(W)$ via Lemma \ref{L:shift} with $U_2, f(W)_2 = \emptyset.$ Again we can define sets $\bar v_a(u), \bar v_b(u)$ for $u \in U_h$. Note that $\bar v_a(u) = T_{(k,0)}v_a(u)$ and $\bar v_b(u) = T_{(\ell,0)}v_b(u)$.
	
	Now for $u \in U_h$, let $v(u)$ be the smallest box containing $v_a(u), v_b(u), \bar v_a(u), \bar v_b(u)$, and set $\tilde V = \{v(u) : u \in U_h\}$. Observe that the projections of $v(u), v_a(u), v_b(u), \bar v_a(u), \bar v_b(u)$ onto the second coordinate are all equal to 
	$$
	\bigcup_{r \in K(u, U)} [r_2^-, r_2^+],
	$$
	where $K(u, U)$ is the connected component of $U$ containing $u$. Therefore $v(u)$  crosses $v_a(u), v_b(u), \bar v_a(u),$ and $\bar v_b(u)$ horizontally and $\tilde V$ consists of disjoint boxes. Note also that $\sqcup \tilde V \sset b$, so $\tilde V \X U_n \sset \scrN$. By Lemma \ref{L:easy-observations}(ii) and (iii), $U_h \X \tilde V \sset \scrH$ and $W \X \tilde V \sset \scrV$. Finally, $\sqcup (\bar V \cup V) \sset \sqcup \tilde V$, and since $\tilde V \sset U$, we have that $\sqcup \tilde V \cap \sqcup W \sset \sqcup V$, and $\sqcup \tilde V \cap \sqcup f(W) \sset \sqcup \bar V$. Putting all these observations together implies that $(U, \tilde V, V,W)$ and $(U, \tilde V, \bar V, f(W))$ are Markov quadruples.
	
	Extend $f$ to $\tilde V \cup V$ so that $f|_{\tilde V} = \id, f|_{V_a} = T_{(k, 0)}$ and $f|_{V_b} = T_{(\ell, 0)}$, and observe that $f|_{U \cup \tilde V} \in \scrF$ since it is the identity, $f|_{\tilde V \cup V} \in \scrF$ by Lemma \ref{L:G-moves} and Lemma \ref{L:in-F-triv}(2), and $f|_{V \cup W} \in \scrF$ by Lemma \ref{L:in-F-triv} (2). Therefore $f \in \scrF$ by Corollary \ref{C:2-step-markov'}.
\end{proof}

Next, we prove that the class of functions in Theorem \ref{T:puzzle-pieces} lies in $\scrF$. See Figure \ref{fig:shift} for an example.

\begin{prop}
	\label{P:slides} (Slides)
	Let $U, W \sset \Zd$. Suppose that we can partition $U = U_1 \cup U_2$ and $W = W_1 \cup W_2$ such that $U_1 \X W_1, W_2 \X U_2 \sset \scrH \cup \scrN$ and $U_1 \X W_2, U_2 \X W_1 \sset \scrN$. Now let $c \in \{(0, \pm 1), (\pm 1, 0\}$, and define a map $\sig$ with domain $U \cup W$ by  $\sig|_U = T_c, \sig|_W = \id$. Suppose that $\sig$ preserves $\scrH$ and $\scrN$. Then $\sig \in \scrF$.
\end{prop}

To prove Proposition \ref{P:slides}, we need a technical lemma.

\begin{lemma}
	\label{L:shift-part-ii}
	Let $U, W$, and $\sig$ be as in the statement of Proposition \ref{P:slides} with $c = (0, 1)$.
	Then there exists a set $V = V_1 \cup V_2 \sset \Zd$ such that 
	\begin{align*}
	(U, V, (V_1 + (0, c)) \cup (V_2 + (c, 0)), W), \quad \mathand \quad (\sig U, T_c V, (V_1 + (0, c)) \cup (V_2 + (c, 0)), \sig W)
	\end{align*}
	are Markov quadruples.
\end{lemma}

\begin{proof}
	Let $U^* = U \cup [U_1 + (0, c)]$. We claim that the pair $U^*, W$ satisfies the assumptions of Lemma \ref{L:shift}, where $W_1, W_2$ are as in the statement of Proposition \ref{P:slides}, $U^*_1 = U_1 \cup [U_1 + (0, c)]$ and $U^*_2 = U_2$. Now, for $u^* \in U_1 + (0, c)$, we can write $u^* = u \cup T_c u$ for some $u \in U_1$. Since $\sig$ preserves $\scrH$ and $\scrN$, for all $w \in W_1$, either $(u, w)$ and $(T_c u, w)$ are both in $\scrH$ or they are both in $\scrN$. Also, $\{u, T_c u\} \X W_2 \sset \scrN$. Therefore Lemma \ref{L:easy-observations} (ii) implies that $\{u^*\} \X W_1 \sset \scrH \cup \scrN$ and $\{u^*\} \X W_2 \sset \scrN$, so $U^*, W$ satisfy the assumptions of Lemma \ref{L:shift}.
	
	Let $G = (G_1, G_2)$ be the set constructed via Lemma \ref{L:shift} for the pair $(U^*, W)$ that makes $(U^*, G, W)$ a Markov triple. 
	Setting $G^- = (G_1 - (0, c)) \cup G_2$, we claim that $(U, G^-, W)$ is also a Markov triple. For this, observe that if $u \in U_1$, then for all $w \in W_1$ with $(u, w) \in \scrH$, there exists $g \in G_1$ such that $(u + (0, c)) \cap w \sset g$. Since $(u, w) \in \scrH$, this implies that $u \cap w \sset g - (0, c)$. Combining this with the fact that that $\sqcup U_2 \cap \sqcup W_2 \sset G_2$ implies that $\sqcup U \cap \sqcup W \sset \sqcup G^-$.
	
	Definition \ref{D:crossing-triple}(ii) is also satisfied since we have only shrunk or removed boxes from $(U^*, G, W)$. For Definition \ref{D:crossing-triple}(iii), all statements are immediately inherited from $(U^*, W, G)$ except for the fact that $U_1 \X (G_1 - (0, c)) \sset \scrN \X \scrH$. For this, again simply observe that for any $u \in U_1$, that $\{u + (0, c)\} \X G_1 \sset \scrH \cup \scrN$, so therefore $\{u\} \X (G_1 - (0, c)) \sset \scrH \cup \scrN$ as well.
	
	Now let $W^* = W_1 \cup W_2^*$, where $W_2^* = W_2 \cup W_2 - (c, 0)$. By symmetric reasoning, we get a Markov triple $(U, H = H_1 \cup H_2, W)$ constructed as in Lemma \ref{L:shift}, with the property that $(U, H_1 \cup (H_2 + (c, 0)), W)$ is another Markov triple.
	
	We now show that $(U, G_1 \cup (H_2 + (c, 0)), W)$ and  $(U, (G_1 - (0, c)) \cup H_2, W)$ are both Markov triples. All properties of follow straight from the corresponding properties of the Markov triples constructed above except for the fact that
	\begin{equation}
	\label{E:GGG}
	(G_1 - (0, c))\X H_2, G_1 \X (H_2 + (c, 0)) \sset \scrN.
	\end{equation}
	Using that $G$ and $H$ were constructed as in Lemma \ref{L:shift}, we have $\sqcup G_1 \sset \sqcup W_1$ and $\sqcup H_2 \sset \sqcup U_2$. Therefore $G_1 \X H_2 \sset \scrN$ since $W_1 \X U_2 \sset \scrN$, and \eqref{E:GGG} follows. 
	
	Setting $V_1 = G_1, V_2 = (H_2 + (c, 0))$, it is then straightforward to check that $(U, V, (V_1 + (0, c)) \cup (V_2 + (c, 0)), W)$ is a Markov quadruple. The claims for $\sig U, \sig W$ follow from symmetric reasoning.
\end{proof}

\begin{proof}[Proof of Proposition \ref{P:slides}]
	We only prove the proposition when $c = (0, 1)$; the other cases follow by symmetry. Let $V$ be the set constructed from $U$ and $W$ by Lemma \ref{L:shift-part-ii} and set $V' = (V_1 + (0, c)) \cup (V_2 + (c, 0))$. Extend $\sig$ by letting $\sig|_V = T_c$ and $\sig|_{V'} = \id$. Now, $\sig|_{U \cup V} \in \scrF$ since it is a translation, $\sig|_{V' \cup W} \in \scrF$ since it is the identity, and $\sig|_{V \cup V'} \in \scrF$ by Lemma \ref{L:in-F-triv} (it is a collection of disjoint basic reflections). Therefore $\sig \in \scrF$ by Corollary \ref{C:2-step-markov'} since $(U, V, V', W)$ is a Markov quadruple.
\end{proof}

The last two types of functions that we will show are in $\scrF$ are built up from column transpositions and slides.

\begin{prop}
	\label{P:lie-in-F}
	\begin{enumerate}
		\item (Towers) Let $U_1, \dots, U_k, U_1', \dots, U_k' \sset \Zd$ be such that $U_i \X U_{i+1}, U'_i \X U'_{i+1} \sset \scrH$ for all $i$, and such that for all $i$, $U_i' = T_{c_i} U_i$ for some $c_i \in \Z^2$. Let $f:\bigcup U_i \to \bigcup U_i'$ be the function whose restriction to $U_i$ is $T_{c_i}$. Then $f \in \scrF$. 
		\item (Box permutations) Let $U_1, \dots, U_k, V_1, \dots, V_m, U_1', \dots, U_k', V_1', \dots, V_m' \sset \Zd$ be such that $U_i \X V_j, U_i' \X V_j' \sset \scrH$ for any $i, j$, and $U_i \X U_j, V_i \X V_j, U_i' \X U_j', V_i' \X V_j' \sset \scrN$ for any $i \ne j$. Suppose also that for all $i, j$, there exists $c_i, d_j \in \Z^2$ such that $U_i' = T_{c_i} U_i$ and $V_j' = T_{d_j} V_j$. Let $f:\bigcup U_i \cup \bigcup V_i \to \bigcup U_i' \cup \bigcup V_i'$ be the function with $f|_{U_i} = T_{c_i}$ and $f|_{V_j} = T_{d_j}$. Then $f \in \scrF$.
	\end{enumerate}
\end{prop}

\begin{proof}
	\textbf{Proof of 1.} (Towers):  \qquad We proceed inductively starting with $k = 2$ (the $k=1$ case is trivial). First, by translating $U_1', U_2'$ by a common amount, we may assume that $c_2 = 0$. Second, when $c_1 \in \{(\pm 1, 0), (0, \pm 1)\}$, then the fact that $f \in \scrF$ follows from Proposition \ref{P:slides}. For the case of general $c_1$, observe that since $U_1 \X U_2, T_{c_1} U_1 \X U_2 \sset \scrH$, that $T_x U_1 \X U_2 \sset \scrH$ for any $x$ in the box with two diagonally opposite corners given by $c_1$ and $(0, 0)$. Therefore we can compose $f$ of slide maps $f_i:T_{x_i} U_1 \X U_2 \to T_{x_{i+1}} U_1 \X U_2$ where $x_0, \dots, x_k$ is a path from $(0, 0)$ to $c_1$ in that box.
	
	For the inductive step when $k \ge 3$, we can create the map $f:\bigcup U_i \to \bigcup U_i'$ by composing two maps of the form in (1), where the first has $c_1 = c_2$, and the second has $c_2 = c_3 = \dots = c_k$. Each of these maps lies in $\scrF$ by the inductive hypothesis.
	
	\textbf{Proof of 2.} (Box permutations): \qquad  We may assume that each of the sets $U_i, V_i$ is connected. If not, we can simply further break them up into their connected components and prove that the corresponding larger set of functions lies in $\scrF$. 
	
	Fix a box $v = (v^-, v^+) \in \bigcup V_j$. For every $x \in [v^-_2, v^+_2]$, there is at most one value of $i$ for which $[v^-_1, v^+_1] \X \{x\}$ intersects $\sqcup U_i$. Moreover, for each $i$, the connectivity of each of the sets $U_i$ and the fact that $U_i \X \{v\} \sset \scrH$ for all $i$ implies that the set of such $x$ for which  $[v^-_1, v^+_1] \X \{x\}$ intersects $\sqcup U_i$ is an interval $I_i = [I_i^-, I_i^+]$ given by the projection of $\sqcup U_i$ onto the second coordinate. In particular, $I_i$ is independent of the choice of $v$.
	
	We similarly define intervals $J_i = [J_i^-, J_i^+]$ given by projecting $\sqcup V_i$ onto the first coordinate, and intervals $I_i'$ and $J_i'$ corresponding to $U_i'$ and $V_i'$. By possibly relabelling, we may assume that $I_i^+ < I^-_{i+1}$ and $J_i^+ < J^-_{i+1}$ for all $i$.
	
	\textbf{A special case:} $I_i = I_i', J_i = J_j'$ for all $i, j$. \qquad Let $T_i$ be the largest interval such that every $u \in U_i$ crosses $T_i \X I_i$ horizontally and let $S_i$ be the largest interval such that every $v \in V_i$ crosses $J_i \X S_i$ vertically. Similarly define $T_i', S_i'$. Observe that the intervals $T_i, I_i$ determine the location of $U_i$ among the set of all possible translations. Note that $T_i' = \al_i + T_i, S_i' = \be_i + S_i$ for integers $\al_i, \be_i$, and since $I_i = I_i', J_i = J_i'$, we have
	$$
	U_i' = T_{(\al_i, 0)} U_i \qquad \mathand \qquad  V_i' = T_{(0, \be_i)} V_i.
	$$
	Now let $S_1 = \prod_{i=1}^k [0, \al_i], S_2 =  \prod_{i=1}^m [0, \be_i],$ where the order of the endpoints are switched whenever $\be_i, \al_i < 0$.
	For any $(\bx, \by) \in S_1 \X S_2$, observe that $T_{(x_i, 0)} U_i \X T_{(0, y_j)} V_j \sset \scrH$ for any $i, j$, and that 
	$$
	T_{(x_i, 0)} U_i \X T_{(x_j, 0)} U_j, \quad T_{(0, y_i)} V_i \X T_{(0, y_j)} V_j \sset \scrN \quad \mathfor i \ne j.
	$$ 
	Therefore 
	for any $(\bx, \by), (\bx', \by') \in S_1 \X S_2$ with $|(\bx, \by) - (\bx', \by')| = 1$, the slide map
	$$
	g:\bigcup T_{(x_i, 0)} U_i \cup \bigcup T_{(0, y_i)} V_i \to \bigcup T_{(x_i', 0)} U_i \cup \bigcup T_{(0, y_i')} V_i
	$$
	which is the identity everywhere except for at the single coordinate where $(\bx, \by)_i \ne (\bx', \by')_i$, lies in $\scrF$ by Proposition \ref{P:slides}. We can compose the map $f:\bigcup U_i \cup \bigcup V_i \to \bigcup U_i' \cup \bigcup V_i'$ from such maps.
	
	\textbf{The general case:} \qquad By the first case, we can first apply a transformation to get that $T_i = [J_1^-, t_i]$ and $S_j = [I_1^-, s_i]$ for all $i, j$. By translating, we may also assume that 
	$$
	I_1^- = \min \lf\{x \in \Z: x \in \bigcup I_i' \rg\}, \quad \mathand \quad J_1^- = \min \lf\{x \in \Z: x \in \bigcup J_i' \rg\}.
	$$
	By again applying the first case we may then assume that $T'_i = T_i, S_i = S_i'$.
	
	Now, given $S_i, T_i$ as above, letting $\hat s = \min_i s_i$ and $\hat t = \min_i t_i$, the only constraints on the intervals $I_i'$ are that they are disjoint and contained in $[I_1^-, \hat s]$ and the only constraints on the $J_i'$ are that they are disjoint and contained in $[J_1^-, \hat t]$. We can generate all such collections of intervals from $I_i, J_i$ if we can change the order of any two adjacent intervals and translate any interval without changing the order of the intervals. 
	
	We check that such moves are in $\scrF$. By symmetry, it suffices to check that we can move the intervals $J_i$. Fix $i \in \{1, \dots m-1\}$ and let $J \sset [J_1^-, \hat t]$ be the largest interval containing $J_i \cup J_{i+1}$ that is disjoint from all the other $J_j$.
	Then any map $f$ with domain $\bigcup U_i \cup \bigcup V_i$ such that:
	\begin{itemize}[nosep]
		\item $f|_{V_i} = T_{(\ell, 0)}, f|_{V_{i+1}} = T_{(k, 0)}$, and $f$ is the identity everywhere else,
		\item $\sqcup (f(V_1) \cup f(V_2)) \sset J \X \Z$,
		\item $f$ preserves $\scrH$ and $\scrN$,
	\end{itemize}
	is in $\scrF$ by Proposition \ref{P:basic-perm}. In the application of the proposition, we take $\bigcup U_i = U_h, \bigcup_{j \ne i, i+1} V_j = U_n, V_i = W_a,$ and $V_{i+1} = W_b$. The box $b$ in the proposition can be any large enough box of the form $J \X [-r, r]$. Such maps $f$ allow us to change the order of any two intervals and translate any interval without changing the interval order, as desired.
\end{proof}

The main homogeneous theorems from Section \ref{SS:symmetries-lpp} immediately follow from the main results above. Theorem \ref{T:conj-bgw} and Theorem \ref{T:example} follow from Proposition \ref{P:lie-in-F}, and Theorem \ref{T:puzzle-pieces} follows from Proposition \ref{P:slides}. The corresponding parts of Theorem \ref{T:polymers} also follow from these results.

\subsection{Inhomogeneous models}
\label{SS:inhomog}

In this section, we only work with environments and models that fit into the framework of Theorem \ref{T:explicit} for nonconstant sequences $\al, \be$. We write $M_{\al, \be}$ for such environments. Throughout this section, we assume that different environments $M_{\al, \be}, M_{\al', \be'}$ that we establish relationships between are always of the same type. That is, they are both geometric environments, both exponential environments, or both inverse gamma environments.

The type of statements that we can obtain will be of the form
\begin{equation}
\label{E:M-al-be}
Z_{M_{\al, \be}}|_{\scrE_D(S)} \eqd Z_{M_{\phi_f(\al, \be)}}|_{\scrE_D(T)} \circ \bar f,
\end{equation}
where $f:S \to T \in \scrF$, and $\phi_f(\al, \be)$ is another pair of biinfinite sequences obtained in an $f$-dependent way. Note that there may be many choices of $\phi_f(\al, \be)$. Also, it is easy to construct examples of  $f \in \scrF$ and biinifinite sequences $\al, \be$ for which no $\phi_f(\al, \be)$ satisfying \eqref{E:M-al-be} exists. Here we do not attempt to classify when a function $\phi_f(\al, \be)$ exists and exactly what the restrictions on it are. 
Rather we will just illustrate the sort of statements that can be proven by showing the tower case of Proposition \ref{P:lie-in-F}. This will prove Theorem \ref{T:inhom-ex} and the corresponding statement in Theorem \ref{T:polymers}. Similar statements may be obtained for the other explicit classes of functions constructed in Section \ref{SS:homog}.

We start with a few simple lemmas. The first is an analogue of Lemma \ref{L:in-F-triv}(i). For this, recall the notation $\al_u, \hat \be_u$ introduced prior to Theorem \ref{T:decoupled-family}.

\begin{lemma}
	\label{L:no-dependence}
	Let $u \in \Zd$ and suppose $w = u + (1, 0; 0,0), w' = u + (0, 0, -1, 0) \in \Zd$. Let $\al, \al', \be$ be biinfinite sequences such that $\al_u$ is a permutation of $\al'_u$, and $\al_w$ is a permutation of $\al'_{w'}$. Now define a function $f:\{u, w\} \to \{u, w'\}$ by $f(u) = u, f(w) = w'$ and let $\bar f$ be its extension to $\scrE_D(u,w) := \scrE_D(\{u, w\})$. Then
	$$
	Z_{M_{\al, \be}}|_{\scrE_D(u, w)} \eqd Z_{M_{\al', \be}}|_{\scrE_D(u, w')} \circ \bar f.
	$$ 
\end{lemma}

\begin{proof}
	The map $f$ is composed of two reflections and a translation. Applying these reflections and translations to $\al, \be$ to give $\bar \al, \bar \be$ gives that 
	\begin{equation}
	\label{E:albe}
	Z_{M_{\al, \be}}|_{\scrE_D(u, w)} \eqd Z_{M_{\bar \al, \bar \be}}|_{\scrE_D(u, w')} \circ \bar f.
	\end{equation}
	On the other hand, by Theorem \ref{T:decoupled-family}, $Z_{M_{\bar \al, \bar \be}}|_{\scrE_D(u, w')}$ only depends on the ordering of $\be_u$. In particular, in \eqref{E:albe} we can exchange $\bar \be$ for $\be$ without affecting the distributional equality. Next, observe that by Propositions \ref{P:poly-comp} and \ref{P:measurable-fn}, that $Z_{M_{\bar \al, \bar \be}}|_{\scrE_D(u, w')}$ is a function of $M_{\bar \al, \bar \be}|_{\{u^+_1\} \X \Z}$ and $Z_{M_{\bar \al, \bar \be}}|_{H(w')}$. By Theorem \ref{T:decoupled-family}, these two objects are determined by only the sets of numbers in $\al_u, \al_w$. In particular, this allows us to exchange $\bar \al$ for $\al'$ in \eqref{E:albe}.
\end{proof}

The next two lemmas are consequences of Lemma \ref{L:no-dependence} combined with Corollary \ref{C:2-step-Markov}.

\begin{lemma}
	\label{L:one-shift}
	Let $U, W \sset \Zd$ with $U \X W, U \X T_{c}W \sset \scrH$ with $c = (1,0)$. Define $f:U \cup W \to U \cup T_c W$ by $f|_U = \id$ and $f|_W = T_{c}$, and let $\bar f$ be its extension to $\scrE_D(U \cup W)$.
	
	Let $[a, b]$ be the smallest interval containing the projection of $\sqcup W$ onto the first coordinate and suppose that $\al, \al', \be$ are biinfinite sequences with $\al'_a = \al_{b+1}, \al'_{i+1} = \al_i$ for $i \in [a, b]$, and $\al'_j = \al_j$ for all $j \notin [a, b + 1]$. Then
	$$
	Z_{M_{\al, \be}}|_{\scrE_D(U \cup W)} \eqd Z_{M_{\al', \be}}|_{\scrE_D(U \cup T_c W)} \circ \bar f.
	$$ 
\end{lemma}

\begin{proof}
	Let $I$ be smallest interval containing the projection of $\sqcup U$ onto the second coordinate. Set $v = [a, b] \X I, v' = [a, b + 1] \X I \in \Zd$. Then $(U, v', v, W)$ and $(U, v', T_c v, T_c W)$ both satisfy the assumptions of Corollary \ref{C:2-step-Markov}, so
	\begin{align*}
	&(Z_{M_{\al, \be}}|_{\scrE_D(U)}, Z_{M_{\al, \be}}|_{\scrE_D(v')}, Z_{M_{\al, \be}}|_{\scrE_D(v)}, Z_{M_{\al, \be}}|_{\scrE_D(W)}) \qquad \qquad \qquad \mathand \\
	&(Z_{M_{\al', \be}}|_{\scrE_D(U)}, Z_{M_{\al', \be}}|_{\scrE_D(v')}, Z_{M_{\al', \be}}|_{\scrE_D(T_c v)} \circ \bar T_c , Z_{M_{\al', \be}}|_{\scrE_D(T_c W)} \circ \bar T_c ) 
	\end{align*}
	are both Markov chains, where $\bar T_c$ is the extension of $T_c$ to $\scrE$. Moreover, the transition probabilities are the same for these two chains. The first transition probability is the same since $Z_{M_{\al, \be}}|_{\scrE_D(U) \cup \scrE_D(v')}$ does not depend on the order of $\al$ in $[a, b+1]$. The second transition probabilities are the same by Lemma \ref{L:no-dependence}, and the third transition probabilities are the evidently the same.
\end{proof}

\begin{lemma}
	\label{L:permute-ab}
	Let $I = \{i_1 < \dots < i_k\} \sset \Z$. Suppose that $W \sset \Zd$ can be partitioned into two sets $W_1$ and $W_2$ such that
	\begin{itemize}[nosep]
		\item For some $r \in \Z$, every $u \in W_1$ crosses $[i_1, i_k] \X [-r, r]$ horizontally.
		\item $\sqcup W_2 \cap (I \X \Z) = \emptyset$.
		\item $W_1 \X W_2 \sset \scrH$.
	\end{itemize}
	Suppose that $\al, \al', \be$ are bi-infinite sequences such that $\al|_{I^c} = \al'|_{I^c}$ and $\al|_I$ is a permutation of $\al'|_{I}$. Then 
	$$
	Z_{M_{\al, \be}}|_{\scrE_D(W)} \eqd Z_{M_{\al', \be}}|_{\scrE_D(W)}.
	$$
\end{lemma}

\begin{proof}
	First, we may reduce to the case $k = 2$. To see this, observe that any pair of points in $I$ satisfy the assumptions of the lemma for $k = 2$. Therefore the $k=2$ case implies that $\al'|_I$ can be related to $\al|_I$ by any transposition. Since transpositions generate the symmmetric group, this gives the general case.
	
	Now, since $W_1 \X W_2 \sset \scrH$, we have $\scrE_D(W) = \scrE_D(W_1) \cup \scrE_D(W_2)$. Moreover, by Propositions \ref{P:poly-comp} and \ref{P:measurable-fn}, $Z_{M_{\al, \be}}|_{\scrE_D(W_1)}$ is a function of $M_{\al, \be}|_{\Z \X [i_1, i_2]^c}$ and $Z_{M_{\al, \be}}|_{H(S)}$, where 
	$$
	S = \{w \cap ([i_1, i_2] \X [-r, r]) : w \in W_1\}.
	$$
	Since $\sqcup W_2 \cap (\{i_1, i_2\} \X\Z) = \emptyset$, we may also decompose $W_2 = W_{2, a} \cup W_{2, b}$, where $\sqcup W_{2, a} \sset [i_1, i_2]^c \X\Z$ and $\sqcup W_{2, b} \sset (i_1, i_2) \X \Z$. Since $M$ consists of independent entries, it therefore suffices to show that the joint distribution of
	\begin{equation}
	\label{E:ZZZ}
	Z_{M_{\al, \be}}|_{\scrE_D(W_{2, b})}, Z_{M_{\al, \be}}|_{H(S)}
	\end{equation}
	is independent of the order of $i_1, i_2$. In the case $i_2 = i_1 + 1$, $W_{2, b} = \emptyset$ and $Z_{M_{\al, \be}}|_{H(S)}$ is independent of the order of $i_1, i_2$ by Theorem \ref{T:decoupled-family}, so the claim follows. In the case when $i_2 > i_1 + 1$, by Lemma \ref{L:one-shift} the joint distribution in \eqref{E:ZZZ} is the same as the joint distribution of 
	\begin{equation}
	\label{E:ZZY}
	Z_{M_{\al^*, \be}}|_{\scrE_D(T_{(0, 1)} W_{2, b})} \circ \bar T_{(0, 1)}, Z_{M_{\al^*, \be}}|_{H(S)}
	\end{equation}
	where $\al^*_{i_1 + 1} = \al_{i_2}, \al^*_{j+1} = \al_j$ for $j \in [i_1 + 1, i_2-1]$ and $\al^*_j = \al_j$ for all other $j$. By the $i_2 = i_1 + 1$ case, the joint distribution in \eqref{E:ZZY} does not depend on the order of $\al_{i_1}$ and $\al_{i_2}$. Hence nor does the joint distribution in \eqref{E:ZZZ}.
\end{proof}

We can now state an invariance theorem.

\begin{theorem}
	\label{T:towers}
	Let $U_1, \dots, U_k, U_1', \dots, U_k' \sset \Zd$ be such that $U_i \X U_{i+1}, U'_i \X U'_{i+1} \sset \scrH$ and $U_i' = T_{c_i} U_i$ for some $c_i \in \Z^2$. Let $f:\bigcup U_i \to \bigcup U_i'$ be the function whose restriction to $U_i$ is $T_{c_i}$.
	Now let $\al, \be, \al', \be'$ be bi-infinite sequences satisfying the following condition.
	\begin{itemize}
		\item For each $i$, let $I(i)^-_1$ be the smallest interval
		containing $\{u^-_1 : u \in U_i\}$. Similarly define $I(i)^-_2, I(i)^+_1,$ and $I(i)^+_2$. Define the completion of each set $U_i$ by
		$$
		\bar U_i = \{u \in \Zd : u_1^- \in I(i)^-_1, u_2^-  \in I(i)^-_2, u_1^+ \in I(i)^+_1, u_2^+ \in I(i)^+_2\}.
		$$ 
		Then for all $u \in \bigcup \bar U_i$, there exist permutations $\sig_u$ and $\tau_u$ such that
		\begin{equation}
		\label{E:fff}
		\al_u = \sig_u(\al'_{f(u)}) \qquad \mathand \qquad \hat \be_u = \tau_u(\hat \be'_{f(u)})
		\end{equation}
	\end{itemize}
	Then $Z_{M_{\al, \be}}|_{\scrE_D(\bigcup U_i)} \eqd Z_{M_{\al', \be'}}|_{\scrE_D(\bigcup U_i')} \circ \bar f$.
	
\end{theorem}

We remark that the set of permutations $\sig, \tau$ that satisfy \eqref{E:fff} is always nonempty. This is a byproduct of the proof below.

\begin{proof}
	We extend the map $f$ to $\bigcup \bar U_i$. by letting $f = T_{c_i}$ on each of the sets $\bar U_i$. The image of each set $\bar U_i$ under this map is the completion $\bar U_i'$ of $U_i'$, defined analogously to the completion of $U_i$. Moreover, from the definitions we still have $\bar U_i \X \bar U_{i+1}, \bar U_i' \X \bar U_{i+1}' \sset \scrH$ for all $i$.
	
	We first show that for any fixed $\al, \be$, there exists some $\al', \be'$ satisfying the assumptions of the theorem for which 
	$$
	Z_{M_{\al, \be}}|_{\scrE_D(\bigcup \bar U_i)} \eqd Z_{M_{\al', \be'}}|_{\scrE_D(\bigcup \bar U_i')} \circ \bar f.
	$$
	This part is analogous to the inductive proof of Proposition \ref{P:lie-in-F} (1). The $k = 2$ case where $c_2 = 0$ and $c_1 = (0, 1)$ follows from Lemma \ref{L:one-shift}. The cases where $c_1 \in \{(0, \pm 1), (\pm 1, 0)\}$ follow by symmetry, and all other cases can be obtained by composing a sequence of maps with $c_1 \in \{(0, \pm 1), (\pm 1, 0)\}$, as in the proof of Proposition \ref{P:lie-in-F} (1). The result can be extended to the $k > 2$ case by the same inductive argument as in Proposition \ref{P:lie-in-F}(1).
	
	We now have one pair $\al', \be'$ for which the theorem holds. This pair satisfies the assumptions of the theorem since those properties were preserved by the map in Lemma \ref{L:one-shift}. Suppose that $\al'', \be''$ is another pair satisfying those assumptions. We show that we can move from the theorem statement for $(\al', \be')$ to $(\al'', \be')$. A symmetric argument then shows that if the theorem holds for $(\al'', \be')$ then it will also hold for $(\al'', \be'')$.
	
	Letting $v \in \Zd$ be the smallest box containing $U'_i$ for all $i$, we necessarily have that $\al'_v = \sig(\al''_v)$ and $\hat \be'_v = \tau(\hat \be''_v)$ for permutations $\sig$ and $\tau$. Now, let $[x_{i,1}^-, y_{i,1}^-]$ be the smallest interval containing $\{u_1^- : u \in U_i'\}$, and similarly define 
	 $[x_{i,2}^-, y_{i,2}^-], [x_{i,1}^+, y_{i,1}^+], [x_{i,2}^+, y_{i,2}^+]$.
	 Then $v = [x_{1,1}^-, y_{k,1}^+]\X [x_{1,2}^-, y_{k,2}^+]$ is the smallest box containing all boxes in $\bigcup \bar U_i$, and there is a permutation $\sig$ of $[x_{i,1}^-, y_{k,1}^+]$ such that $\al'_v = \sig(\al''_v)$. The function $\sig$ can be any permutation such that $\sig|_{[a, b]}$ is also a permutation whenever $a \in [x_{i,1}^-, y_{i,1}^-]$ and $b \in [x_{i,1}^+, y_{i,1}^+]$ for the same value of $i$.
	 
	 Using that $y_{i, 1}^- \le x_{i+1, 1}^-$ and $y_{i + 1, 1}^+ \le x_{i, 1}^+$ for all $i$, we can conclude that $\sig$ can be any permutation of
	 $[x_{i,1}^-, y_{k,1}^+]$ such that
	 \begin{itemize}
	 	\item For all $i \in [i, k]$, $\sig$ fixes every point in $[x_{i,1}^-, y_{i,1}^-) \cup (x_{i,1}^+, y_{i,1}^+]$. Here we use the convention that $[a, a) = (a, a] = \emptyset$.
	 	\item For all $i \in [1, k - 1]$, $\sig|_{[y_{i, 1}^-, x_{i+1, 1}^-) \cup (y_{i + 1, 1}^+, x_{i, 1}^+]}$ is a permutation of $[y_{i, 1}^-, x_{i+1, 1}^-) \cup (y_{i + 1, 1}^+, x_{i, 1}^+]$.
	 		\item If $y_{k, 1}^- \le x_{k, 1}^+$, then $\sig|_{[y_{k, 1}^-, x_{k, 1}^+]}$ is a permutation of $[y_{k, 1}^-, x_{k, 1}^+]$.
	 \end{itemize}
 In particular, the above decomposition implies that $\sig = \tau_1 \cdots \tau_k$, where for $i \le k-1$, each $\tau_i:\Z \to \Z$ is equal to the identity permutation outside of the set $[y_{i, 1}^-, x_{i+1, 1}^-) \cup (y_{i + 1, 1}^+, x_{i, 1}^+]$, and $\tau_k$ is equal to the identity outside of the set $[y_{k, 1}^-, x_{k, 1}^+]$. To complete the proof, we just need to show that the distribution of $Z_{M_{\al', \be'}}|_{\scrE_D(\bigcup \bar U_i')}$ is unchanged if we permute the coordinates of $\al'$ by any one the functions $\tau_i$. 
 
 For this we appeal to Lemma \ref{L:permute-ab}. Fix $i \in [1, k]$. In that lemma, we set $I = [y_{i, 1}^-, x_{i+1, 1}^-) \cup (y_{i + 1, 1}^+, x_{i, 1}^+]$ (for $i \le k -1$) or $I = [y_{k, 1}^-, x_{k, 1}^+]$ (for $i = k$). We set $W_1 = \bigcup_{j=1}^i \bar U_j$, and $W_2 = \bigcup_{j > i} \bar U_i$. With these choices, the conditions of Lemma \ref{L:permute-ab} are satisfied, so the distribution of $Z_{M_{\al', \be'}}|_{\scrE_D(\bigcup \bar U_i')}$ is unchanged if we permute the coordinates of $\al'$ by $\tau_i$.
%
%
%
%
%
\end{proof}

\section{Consequences}
\label{SS:consequences}

In this section, we prove all of the corollaries in Sections \ref{SS:symmetries}, \ref{SS:combinatorial}, \ref{SS:directed-limit}, and \ref{SS:limits}. We also prove a few other corollaries about the symmetries of disjointness probabilities and last passage values for paths restricted to certain regions.

\subsection{Proofs for Section \ref{SS:symmetries}, restricted paths, and disjointness events}
\label{SS:sym-proof} 

We first prove Corollary \ref{C:conj-bgw-init} and the analogue of Corollary \ref{C:conj-bgw-init} for the log-gamma-polymer (the final part of Theorem \ref{T:polymers}).
\begin{proof}
	Let $V' = V^- \X V^+ \cap \Zd$.
	Observe that since $U \X V, V \X W \sset \scrH$, that 
	$
	U \X V', V' \X W \sset \scrH.
	$
	Identical statements holds with $T_c V'$ in place of $V'$. Therefore the map 
	$$
	f:U \cup V' \cup W \to U \cup T_cV' \cup W
	$$
	given by translating $V'$ lies in $\scrF$ by Proposition \ref{P:lie-in-F}(1). Moreover, for some measurable function $h$, $Z_M(f, g) = h(Z_M|_{V'})$ and $Z_M(f \circ T_{-c}, g \circ T_{-c}) = h(Z_M|_{T_c V'})$. Putting this together with Theorem \ref{T:iid-trans} gives the result.
\end{proof}

Next, we prove Corollary \ref{C:geometry-cor} and Corollary \ref{C:quenched-cor} together, along with the following corollary about restricted paths, as these three results essentially have the same proof.

 To state the restricted path corollary, for a polymer model $Z_M$, a box $u$, and a set $R_u \sset u$ define the restricted partition function
$$
Z_M(u \; | \; R_u) = \bigoplus_\pi M(\pi),
$$
where the sum is taken over all $u$-paths $\pi$ that are contained in $R_u$. For this corollary we say that a path $\pi$ crosses a box $w$ horizontally if $\pi$ intersects both the left and right sides of $w$ (i.e. $\{w^-_1\} \X [w^-_2, w^+_2]$ and $\{w^+_1\} \X [w^-_2, w^+_2]$).  Similarly, a path $\pi$ crosses a box $w$ vertically if it intersects both the top and bottom sides of $w$.

\begin{corollary}[Restricted paths]
	\label{C:restricted-polymer}
	Let $Z_M$ be either i.i.d.\ exponential or geometric last passage percolation, or the i.i.d.\ log-gamma polymer. 
	Consider subsets $R_u \sset u$ and $R_v \sset v$. Suppose that these sets have the following properties for some $c \in \Z^2$:
	\begin{itemize}[nosep]
		\item The set of $u$-paths that lie in $R_u$ is nonempty, as is the set of $v$-paths that lie in $R_v$.
		\item There exist boxes $w \sset R_u, x \sset R_v$ such that every $u$-path in $R_u$ crosses $w$ horizontally and every $v$-path in $R_v$ crosses $x$ vertically.
		\item $R_u \cap R_v, R_u \cap T_c R_v \sset w$.
		\item The pairs $(w, x), (w, T_c x)$ cross horizontally.
	\end{itemize} Then
	$$
	(Z_M(u \;|\; R_u), Z_M(v \;|\; R_v)) \eqd (Z_M(u \;|\; R_u), Z_M(T_cv \;|\; T_c R_{v})).
	$$
\end{corollary}

For this proof, for $w \in \Zd$ and $S \sset \Z^2$ we introduce the notation
$$
L_H(w) = \{w' \sset w : (w', w) \in \scrH\}, \;\; L_V(w) = \{w' \sset w : (w', w) \in \scrV\}.
$$
Also, for a set $S \sset \Z^2$, we let
$
 G(S) = \{(s, s) : s \in S\} \sset \Zd
 $
be the set of single-point boxes formed from points in $S$.
\begin{proof}[Proof of Corollaries \ref{C:geometry-cor}, \ref{C:quenched-cor}, and Corollary \ref{C:restricted-polymer}]
	We start with Corollaries \ref{C:geometry-cor} and \ref{C:quenched-cor}. Define a function
	\begin{equation*}
	\begin{split}
	f:&G(\sqcup U \smin w) \cup L_H(w) \cup G(\sqcup V \smin x) \cup L_V(x) \to \\
	&G(\sqcup U \smin w) \cup L_H(w) \cup G(\sqcup T_c V \smin T_c x) \cup L_V(T_c x), 
	\end{split}
	\end{equation*}
	where $f = \id$ on $G(\sqcup U \smin w) \cup L_H(w)$ and $f = T_{(c, 0)}$ on $G(\sqcup V \smin w) \cup L_V(x)$. The function $f \in \scrF$ by combining Proposition \ref{P:lie-in-F}(1) and Lemma \ref{L:in-F-triv} (ii). This uses that the sets $\sqcup U \smin w, w \cup x,$ and $\sqcup V \smin x$ are disjoint from each other, as are $\sqcup U \smin w, w \cup T_c x,$ and $\sqcup V \smin T_c x$, and that $(w, x), (w, T_c x) \in \scrH$. Now by Proposition \ref{P:poly-comp}, in the context of Corollary \ref{C:geometry-cor} there is a function $g$ such that
	\begin{equation}
	\label{E:piM}
	\begin{split}
	&(\pi_M(u) \smin w : u \in U; \pi_M(v) \smin w : v \in V) = g(Z_M|_{S}) \quad \mathand \\
	&(\pi_M(u) \smin w : u \in U; \pi_M(v) \smin w : v \in V') = g(Z_M|_{f(S)} \circ f),
	\end{split}
	\end{equation}
	where $S$ is the domain of $f$. Together with Theorem \ref{T:iid-trans}, this implies Corollary \ref{C:geometry-cor}. The analogue of \eqref{E:piM} also holds in the context of Corollary \ref{C:quenched-cor} (e.g.\ with $(P_{u \smin w})_*Q_M$ in place of $\pi_M(u) \smin w$ etc.) proving that corollary. For Corollary \ref{C:restricted-polymer}, construct a function $f' \in \scrF$ analogously to $f$ with $\sqcup U$ replaced by $R_u$ and $\sqcup V$ replaced with $R_v$. Again using Proposition \ref{P:poly-comp}, an analogue of \eqref{E:piM} holds expressing $Z_M(u \; | \; R_u), Z_M(v \; | \; R_v)$ as a function of $Z_M$ restriction to the domain of $f'$.
\end{proof}

We end this section by showing how conditional probabilities of geodesics being disjoint are preserved for exponential and geometric last passage percolation.

\begin{corollary} [Disjointness probabilities]
	\label{C:disjointness}
	Let $M$ be an environment of i.i.d. exponential or geometric random variables. Let $U, V, W \sset \Zd$ and $c \in \Z^2$ and suppose that 
	\begin{itemize}
		\item Every pair of boxes $(u, v)$ in either $U \X V$ or $V \X W$ crosses horizontally.
		\item The same is true after $V$ is shifted by $c$. That is, every pair of boxes $(u, v)$ in either $U \X T_cV$ or $T_cV \X W$ crosses horizontally.
	\end{itemize}
	Then for any $v_1, \dots, v_n \in V$, we have
	\begin{align*}
	&(Z_M|_{U \cup V \cup W}, \mathds{1}(\exists \text{ disjoint $v_i$-geodesics } \mathfor i = 1, \dots, n)) \eqd \\
	&(Z_M|_{U \cup T_c V \cup W} \circ f, \mathds{1}(\exists \text{ disjoint $T_c v_i$-geodesics } \mathfor i = 1, \dots, n)).
	\end{align*}
	
\end{corollary}

\begin{proof}[Proof of Corollary \ref{C:disjointness}]
	The event that there exist disjoint geodesics $\pi_i \in \scrP_M(v_i)$ for $i = 1, \dots, n$ is the same as the event that
	$$
	Z_M(v_1, \dots, v_n) = \sum_{i=1}^n Z_M(v_i).
	$$
	This event is a determined by $Z_M|_{\scrE_D(V)}$. Moreover, the  map $f:U \cup V \cup W \to U \cup T_c V \cup W$ lies in $\scrF$ by Proposition \ref{P:lie-in-F} (1). Putting these two facts together with Theorem \ref{T:iid-trans} applied to the last passage case completes the proof.
\end{proof}

\subsection{Proofs from Section \ref{SS:combinatorial}}
\label{SS:comb-proof}

We first prove Theorem \ref{T:rsk-moon-poly}. For this, we first show that any moon polyomino has a box exhaustion (see Section \ref{SS:combinatorial} for relevant definitions).

\begin{lemma}
	\label{L:moon}
	Let $S$ be a moon polyomino. Then $S$ admits a box exhaustion $(u_0, u_1, \dots, u_{k+1}).$
\end{lemma}

\begin{proof}
	Let
	$$
	U = \{u \in \Zd : u \sset S \text{ and if } u \sset v \sset S, \text{ then } v = u \}
	$$ 
	be the set of maximal boxes in $S$. Note that $\sqcup U = S$. We claim that any pair $u, v \in U$ cross. By maximality of each element of $u$ and the intersection-free property of moon polyominoes, there exist $i_v, i_u, j_v, j_v$ such that 
	$$
	S^{i_u} = [u^-_1, u^+_1], S^{i_v} = [v^-_1, v^+_1], S_{j_u} = [u^-_2, u^+_2], S_{j_v} = [v^-_2, v^+_2].
	$$
	Now, we have containment between $S^{i_u},S^{i_v}$ and $S_{j_u}, S_{j_v}$ by the intersection-free property. Since both $u, v$ were maximal, this implies that either $S^{i_u} \sset S^{i_v}$ and $S_{j_u} \supset S_{j_v}$, or vice versa. Either way, $u$ and $v$ cross.
	
	Since all pairs in $U$ cross, we can order $U = \{\hat u_1, \dots, \hat u_\ell\}$ so that $(\hat u_i,\hat u_{i+1}) \in \scrH$ for all $i$. For every $i \in 1, \dots, \ell -1$, we can create a sequence $\hat u_i = v_{i, 1}, \dots, \hat u_{i+1} = v_{i, k(i)}$ such that $v_{i, j}$ is always obtained from $v_{i, j-1}$ by adding a row or subtracting a column. We can similarly create sequences $v_{0, 1}, \dots, v_{0, k(0)} = \hat u_1$ and $\hat u_\ell = v_{\ell, 1}, \dots, v_{\ell, k(\ell)}$ where $v_{0, 1}, v_{\ell, k(\ell)} = \emptyset$. Listing all the $v_{i, j}$ lexicographically gives a box exhaustion.
\end{proof}

Next, we prove a simple lemma allowing us to extraction bijections from probabilistic statements and previously established bijections.

\begin{lemma}
	\label{L:prob-bij}
	Let $M,N, P$ be three finite sets with $|M|=|N|$, and let $\Phi:M \to P, \Psi:N \to P$. Suppose that $\Phi$ is a bijection and that there exist two measures $\mu$ on $M$ and $\nu$ on $N$ with $\supp(\mu) = M$, whose pushforwards under $\Phi$ and $\Psi$ are equal: $\Phi_* \mu = \Psi_* \nu$. Then $\Psi$ is also a bijection.
\end{lemma}

\begin{proof}
	Since $\mu$ has full support and $\Phi$ is a bijection, $\Phi_* \mu$ and hence $\Psi_* \nu$ has full support as well. This implies that $\Psi$ is a surjection, and since $|M| = |N| = |P|$, $\Psi$ is a bijection. 
\end{proof}

\begin{proof}[Proof of Theorem \ref{T:rsk-moon-poly}]
	Letting $U = (u_1, \dots, u_k)$ be any box exhaustion of $S$ with the empty sets $u_0, u_{k+1}$ omitted, and let $U' = (u'_1, \dots, u'_k)$ be the unique set of boxes in $\Zd$ such that $u_i'$ is a translate of $u_i$ and $u_i'^- = (1, 1)$ for all $i$. The sets $U_i = \{u_i\}, U_i' = \{u_i'\}$ satisfy the conditions of Proposition \ref{P:lie-in-F}(1), so the tower map $f:U \to U'$ translating each $u_i$ to $u_i'$ lies in $\scrF$. In particular, by Theorem \ref{T:iid-trans} and Theorem \ref{T:explicit}(i), letting $M$ be an array of i.i.d.\ geometric random variables with $Z_M$ denoting last passage values, we have that
	$
	Z_M|_{D(U)} \eqd Z_M|_{D(U')} \circ \bar f.
	$
	This implies that $\Phi_U(M|_{S}) \eqd \Phi_{U'}(M|_{\sqcup U'})$, which in turns implies that for any $c \in \N$,
	\begin{equation}
	\p \lf( \Phi_U(M|_{S}) \in \cdot \; \bigg| \sum_{s \in S} M_s = c \rg) = \p \lf( \Phi_{U'}(M|_{\sqcup U'}) \in \cdot \; \bigg| \sum_{s \in \cup U'} M_s = c \rg).
	\end{equation}
	Now, since the entries of $M$ are geometric random variables, the law of $M$ restricted to any set $S$ has full support on the space of functions from $S \to \{0, 1, \dots \}.$ Moreover, since $\sqcup U'$ is a Young diagram, the RSK map $\Phi_{U'}$ is a bijection from the set 
	$$
	\scrX_{\sqcup U'} (c) = \lf\{M \in \scrX_{\sqcup U'} : \sum_{s \in \cup U'} M_s \le c\rg\}
	$$
	to the space of partition sequences in the statement of Theorem \ref{T:rsk-moon-poly} with $|\la_i| \le c$ for all $i$ (see Theorem 7 from \citep{krattenthaler2005growth}). Call this target space $P(c)$. Finally, $|\scrX_{\cup U'}(c)| = |\scrX_S(c)|$ since $|S| = |\sqcup U'|$. Applying Lemma \ref{L:prob-bij} then implies that $\Phi_U$ is also a bijection from $\scrX_{S}(c)$ to $P(c)$ for all $c$, and is hence a bijection. The `Moreover' statement follows from the definition of $\Phi_U$.
\end{proof}

\begin{proof}[Proof of Theorem \ref{T:rsk-scramble}]
	The scrambled RSK bijections are special cases of Theorem \ref{T:rsk-moon-poly}. To see this, let $\tilde \Phi_I(M)$ be the vector of partition functions given by reversing the order of entries in $\Phi_I(M)$ and removing the entry $Z_M^\Delta(\bar I_n)$. Then for any $I, J$, the map $M \mapsto (\Psi_J(M), \tilde \Phi_I(M))$ is a particular $\Phi_U$ map from Theorem \ref{T:rsk-moon-poly} for the box exhaustion $(\emptyset, \bar J_1, \bar J_2, \dots, \bar J_m, \bar I_{n-1}, \dots \bar I_1, \emptyset)$.
\end{proof}

\subsection{Proofs for Section \ref{SS:limits}}
\label{SS:lim-proof}

Note that some of the proofs of the polymer statements in this section have been noted previously (Theorem \ref{T:polymers} and Corollary \ref{C:quenched-cor}). The remaining statements are Corollaries \ref{C:airy-sheet} and \ref{C:airy-processes}. Before proving these statements, we carefully introduce the limiting objects.

First, we introduce last passage percolation on $\R \X \Z$. Let $\{F_i: \R\to \R: i \in \Z\}$ be a sequence of continuous functions. We can define a finitely additive signed measure $dF$ on finite unions of intervals in $\R \X \Z$ by setting
$$
dF([a, b] \X \{i\}) = F_i(b) - F_i(a).
$$
Now define 
$$
\scrQ_\uparrow = \{u = (x, n; y, m) \in (\R \X \Z)^2 : x \le y, n \le m\}.
$$
Again, we can think of $\scrQ_\uparrow$ as boxes in $\R \X \Z$. We say that a set $\pi \sset \R \X \Z$ is a $u$-path with $u = (x, n; y, m)$, if there are points $t_n = x \le t_{n+1} \le \dots \le t_{m+1} = y$ such that 
$$
\pi = \bigcup_{i=n}^m [t_i, t_{i+1}] \X \{i\}.
$$
The analogue of path length in this setting is the measure $dF(\pi)$. We say that two paths are \textbf{essentially disjoint} if their intersection is finite. 

\begin{definition}[Brownian last passage percolation]
	Let $B = (B_i: i \in \Z)$ be a sequence of independent two-sided Brownian motions. For $u \in \scrQ_\uparrow$, define the \textbf{Brownian last passage value}
	$$
	B[u^k] = \sup dB(\pi_1) + \dots + dB(\pi_k),
	$$
	where the supremum is over all essentially disjoint $u$-paths $\pi_1, \dots, \pi_k$. This is defined as long as $k$ is small enough so that disjoint paths exist.
\end{definition}

Brownian last passage percolation arises as a distributional limit of geometric last passage percolation on long thin boxes. In particular, Brownian last passage percolation inherits all invariance statements in Theorem \ref{T:iid-trans} from geometric last passage percolation. We state a few select statments here. 

\begin{theorem}
	\label{T:Brown-symmetries}
	Fix $n \in \N$, and set $\tilde B(x, y, \ell) = B[(x, 1; y, n)^\ell]$, where $B$ is a sequence of two-sided independent Brownian motions. Then
	\begin{enumerate}
		\item Let $I_1 = [a_1, b_1], \dots, I_k = [a_k, b_k]$, $J_1 = [c_1, d_1], \dots, J_k = [c_k, d_k]$ be intervals in $\R$ such that $b_i \le a_{i+1}, d_i \le c_{i-1}$ for all $i$, and such that $b_k \le c_k$. Let $I_1', \dots, I_k', J_1', \dots, J_k'$ be intervals satisfying the same ordering properties such that $I_i' \X J_i'$ is a translate of $I_i \X J_i$ for all $i$. Let $A = \bigcup I_i \X J_i \X [1, n]$ and $A' = \bigcup I_i' \X J_i' \X [1, n]$, and let $f: A \to A'$ be the map given by translating each $I_i \X J_i \X [1, n] \mapsto I_i' \X J_i' \X [1, n]$. Then
		\begin{equation}
		\label{E:Ba}
		\tilde B|_{A} \eqd \tilde B|_{A'} \circ f.
		\end{equation}
		\item Let $g:\R \to \R $ be any nonincreasing function and define $\Ga(g), m$ as in Corollary \ref{C:airy-processes}. For $y \ge 0, i \in [1, n]$, define $m^*(0, t, i) =  (m(-t), i)$ and let $\Ga^*(g) = \Ga(g) \cap \{(x, y) : x \le y\} = m((-\infty, 0])$. Then
		\begin{equation}
		\label{E:Bb}
		\tilde B|_{\{0\} \X [0, \infty) \X [1, n]} \eqd \tilde B|_{\Ga^*(g) \X [1, n]} \circ m^*.
		\end{equation}
	\end{enumerate}
\end{theorem}

\begin{proof}
	Let $M$ be an environment of i.i.d. geometric random variables with mean $\mu$ and standard deviation $\sig$. Geometric last passage values $Z_M$ on a box of fixed height and increasing length converge to Brownian last passage values by Donsker's theorem:
	$$
	\tilde B (x, y, \ell)\eqd \lim_{t \to \infty} \frac{Z_M[(\floor{x t}, 1 ; \floor{ y t}, n)^\ell] - (x - y)t \mu}{\sqrt{t} \sig}. 
	$$
	This convergence in distribution is uniform over any choices of $x, y, \ell$ in a compact set. The equalities \eqref{E:Ba} and \eqref{E:Bb} are therefore inherited on finite subsets of $A, \{0\} \X [0, \infty) \X [1, n]$ from corresponding statements for geometric last passage percolation that hold by Theorem \ref{T:iid-trans} (in particular, these are limits of statements in the form of Proposition \ref{P:lie-in-F}(1)). Continuity of $B$ then extends \eqref{E:Ba} and \eqref{E:Bb} from finite subsets to all of $A, \{0\} \X [0, \infty) \X [1, n]$.
\end{proof} 

Theorem \ref{T:Brown-symmetries} allows us to prove invariance statements for universal objects. The following limit theorems are from \citep{DOV} (for the Airy sheet) and \citep{CH} (for the Airy line ensemble).

\begin{theorem}
	\label{T:airy-limit}
	Let $B$ be a sequence of independent two-sided Brownian motions, and define 
	$$
	\scrS_k^n(x, y) = n^{1/6}\lf(B[(2xn^{-1/3}, 1; 1 + 2yn^{-1/3}, n)^k] - 2\sqrt{n} - 2(y-x)n^{1/6}\rg).
	$$
	Then $S^n_1$ converges in distribution in the uniform-on-compact topology on functions from $\R^2 \to \R$ to a limiting object $\scrS:\R^2 \to \R$ known as the \textbf{Airy sheet}. Moreover, the sequence $(\scrS_1(0, \cdot), \scrS_2(0, \cdot), \dots)$,  converges in distribution in the product of uniform-on-compact topologies on functions from $\R \to \R$ to a limiting object $\scrL = (\scrL_1, \scrL_2, \dots)$ known as the \textbf{parabolic Airy line ensemble}.
\end{theorem}

Note that each $\scrS^n_k$ has a domain which is not equal to all of $\R^2$. However, any compact set in $\R^2$ is contained in the domain of $\scrS^n_k$ for large enough $n$, allowing us to make sense of uniform-on-compact convergence. Note that $\scrA(x) = \scrL(x) + x^2$ is the Airy line ensemble, as defined in \cite{prahofer2002scale} and \cite{CH}. It is stationary in $x$, and the function $\scrA_1$ is the Airy$_2$ process. We work with the non-stationary version $\scrL$ here since it is more closely connected to the Airy sheet $\scrS$.

Uniform convergence to the Airy line ensemble has now been established for many models, including geometric and exponential last passage percolation, see \citep{dauvergne2019uniform}. 
The reason we introduced Brownian last passage percolation as a prelimiting model here is to have access to the Airy sheet. We can now prove Corollaries \ref{C:airy-sheet} and \ref{C:airy-processes}. In addition, we have the following new convergence result for the Airy line ensemble.

\begin{theorem}
	\label{T:airy-symm}
	Let $\scrS^n, \scrL$ be as in Theorem \ref{T:airy-limit}, and $g, \Ga(g), m$ be as in Corollary \ref{C:airy-processes}. Then 
	$$
	\scrS^n \circ m = (\scrS^n_1 \circ m, \scrS^n_2 \circ m, \dots) \cvgd \scrL.
	$$
\end{theorem}

\begin{proof}
	Since $\scrS^n_1 \cvgd \scrS$ by Theorem \ref{T:airy-limit}, $\scrS$ inherits the invariances of Theorem \ref{T:Brown-symmetries}. This immediately proves Corollaries \ref{C:airy-sheet} and \ref{C:airy-processes}, noting that the condition $b_k \le c_k$ in Theorem \ref{T:Brown-symmetries}(1) becomes uneccessary in the limiting scaling, and the graph $\Ga^*(g)$ in Theorem \ref{T:Brown-symmetries}(2) becomes the full graph of a nonincreasing function in the limit (i.e. the intersection with $\{(x, y) : x \le y\}$ is unecessary). Similarly, for any compact set $K \sset \R$, Theorem \ref{T:Brown-symmetries}(2) implies that $\scrS^n \circ m|_K \eqd \scrS^n|_{\{0\} \X K}$ as long as $n$ is large enough so that both sides of this equality are well-defined. Applying Theorem \ref{T:airy-limit} then implies Theorem \ref{T:airy-symm}. 
\end{proof}

\section{Concluding remarks}
\label{S:conclude}

We end with a few questions that arise from this work.

\subsection*{Classification of invariances}

A natural problem raised by this paper is the following.

\begin{problem}
	Let $Z_M$ be the last passage function of either i.i.d.\ exponential or geometric last passage percolation, or the partition function of the i.i.d.\ log-gamma polymer. Find the set $\scrG$ of all bijections $g:S \to T$ between subsets of $\Zd$ such that
	$$
	Z_M|_{S} \eqd Z_M|_{T} \circ g.
	$$
\end{problem}

When $|S| = |T| = 1$, it is relatively straightforward to see that the only functions $g$ that work are translations and reflections.  When $|S| = |T| = 2$, a case-by-case analysis based on comparing probabilities of tail events reveals that the only invariances are those in $\scrF$. The situation gets more complicated for sets with more than $2$ elements, where enumerating cases quickly becomes unwieldy. In particular, it is unclear to me whether or not $\scrG = \scrF$.

\subsection*{Other models and other symmetries}

It would be interesting to see if a version of our techniques can work for generalizations of the models studied in this paper, or for other related models. The key tools needed to carry the proofs through in a more general setting would seem to be an analogue of the RSK correspondence that allows for decoupling, together with a version of Theorem \ref{T:encoding-lpp}. Generalizations of the geometric RSK correspondence do exist further up the hierarchy of Macdonald processes (see \citep{borodin2016nearest, matveev2015q, pei2016q, bufetov2018hall}). However, these generalizations involve randomization and are no longer bijective. Moreover, even some basic translation and reflection symmetries become highly nontrivial results in the context of vertex models (see the discussion on page 6 in \citep{borodin2019shift} and the relevant references \citep{borodin2018coloured, borodin2019color}).

Two models where the situation is more promising and the key tools are available are the Sepp\"al\"ainen-Johansson model \citep{seppalainen1998exact, johansson2001discrete} of last passage percolation (where dual RSK, rather than RSK, is the relevant correspondence), and the strict-weak polymer model (where a dual version of geometric RSK is used, see \citep{o2015tracy, corwin2015strict}).

There are also other geometric symmetries of exponential last passage percolation based on the Burke property of this model \citep{cator2012busemann, pimentel2016duality}. The most striking of these is perhaps an equality in law between the infinite geodesic tree in a given direction and its dual \citep{pimentel2016duality}. How are these symmetries related to our results?

\subsection*{Combinatorics}

The bijections found in Section \ref{SS:combinatorial} arose incidentally from the main theorems in the paper and we have not attempted to study them here. It would be interesting explore these bijections further and to see if the methods used to obtain them can be extended to other settings.

For example, are there nice local descriptions of these bijections, analogous to the local descriptions of the usual RSK bijection? More generally, what features of the usual RSK bijection extend to these new bijections? How are the output tableaux coming from different scrambled RSK bijections related to each other? Are there scrambled versions of other RSK-like bijections?

\section{Appendix}
\label{S:appendix}

Here we provide a few basic examples mentioned in the previous sections regarding the limitations of invariance. The first shows how certain invariance statements might not hold for nonintegrable models. 

\begin{example}
	\label{E:nonintegrable}
	Let $M$ be an i.i.d.\ environment of random variables such that
	$$
	\p(M_{i, j} = 1) = 1 - \ep, \quad \p(M_{i, j} = 0) = \ep
	$$
	for some $\ep > 0$. Then, unlike in the case of decoupled polymer models, for small enough $\ep > 0$ the joint distribution of the last passage values
	$$
	Z_M(1, 1; 2, 3) \quad \mathand \quad Z_M(1, 1; 2, 1)
	$$
	is different from the joint distribution of the last passage values
	$$
	Z_M(1, 1; 2, 3) \quad \mathand \quad Z_M(1, 2; 2, 2).
	$$
\end{example}

\begin{proof}
	Observe that
	\begin{equation}
	\label{E:Zm}
	\p( Z_M(1, 1; 2, 3)  = 4 \; | \;Z_M(1, 1; 2, 1) = 1) \le 1/2,
	\end{equation}
	since conditionally on $Z_M(1, 1; 2, 1) = 1$, the random variable $M(1, 1)$ is equal to $0$ with probability $1/2$. On the other hand,
	\begin{equation}
	\label{E:Zm-2}
	\p( Z_M(1, 1; 2, 3)  = 4 \; | \;Z_M(1, 2; 2, 2) = 1) = (1- \ep)^3,
	\end{equation}
	since conditionally on $Z_M(1, 2; 2, 2) = 1$, exactly one of $M(1, 2)$ and $M(2, 2)$ is $1$, and there is one path from $(1, 1)$ to $(2, 3)$ that picks up this $1$, along with three random variables that are independent of $Z_M(1, 2; 2,2)$. When $\ep < 1 - 2^{1/3}$, \eqref{E:Zm} and \eqref{E:Zm-2} imply that the two joint distributions are different.
\end{proof}

A similar construction to Example \ref{E:nonintegrable} is possible in the polymer setting.

Our next example gives a function $f:S \to T$ satisfying the conditions of Lemma \ref{L:F-restrictions} but $Z_M|_{S} \ne Z_M|_{T} \circ f$ when $M$ is geometric or exponential last passage percolation.

\begin{example}
	\label{E:not-in-F}
	Let $u = (1, 1; 2, 3), v = (2, 2; 2, 2), w = (1, 3;2, 3),$ and $w' =(1, 1; 2, 1)$ and let $S = \{u, v, w\}, T = \{u, v, w'\}$. Define $f:S \to T$ by $f(u) = u, f(v) = v,$ and $f(w) = w'$. Then $f$ satisfies (i)-(iv) of Lemma \ref{L:F-restrictions}, but $Z_M|_{S} \ne Z_M|_{T} \circ f$ when $M$ consists of independent exponential or geometric random variables and $Z_M$ is last passage percolation.
\end{example}

\begin{proof}
	First, for any point $x \in \Z^2$, $M_x$ either has support $\{0, 1, 2, \dots\}$ (for geometric last passage percolation) or support $(0, \infty)$ (for exponential last passage percolation). Therefore the events $(Z_M(v) \ge 1, Z_M(w) \ge 1)$ and $(Z_M(v) \ge 1, Z_M(w') \ge 1)$ have positive probability.
	
	Now, since there is a $u$-path $\pi$ containing $w'$ and $v$, we have that
	$$
	\p(Z_M(u) \ge 2 \;|\; Z_M(v) \ge 1, Z_M(w') \ge 1) = 1.
	$$
	On the other hand,
	$$
	\p(Z_M(u) \ge 2 \;|\; Z_M(v) \ge 1, Z_M(w) \ge 1) < 1,
	$$
	since there is no $u$-path $\pi$ containing $v$ and $w$, and there is a positive probability that $ M_x \in [0, \ep)$ for all $x \in u \smin (v \cup w)$, for any $\ep > 0$.
\end{proof}

Example \ref{E:not-in-F} also works for the log-gamma polymer with a variant of the above proof. 

\subsection*{Acknowledgements} I would like to thank the organizers of the Banff International Research Station workshop on `Dimers, Ising Model, and Interactions', where the striking results of \citep{borodin2019shift} were first brought to my attention. I would like to thank Ivan Corwin, Pavel Galashin, Vadim Gorin, B\'alint Vir\'ag, and Nikolaos Zygouras for helpful comments, discussions, and for pointing out connections with other work. I would like to thank Sam Hopkins for suggesting that there might be a connection between this work and \cite{garver}, and Hugh Thomas for helping me to understand this connection precisely. 

\bibliographystyle{plain}
\bibliography{MasterBib}

\end{document}